\documentclass[a4paper,11pt]{plan-base}
\usepackage{tikz}
\usepackage[english]{babel}
\usepackage{ucs}
\usepackage{ae}
\usepackage{comment}
\usepackage{subcaption}
\usepackage{footmisc}
\usepackage{bbm}
\usepackage{amsmath}
\usepackage{amssymb}
\usepackage{amsthm}
\usepackage{ifpdf}
\usepackage{url}

\usepackage{color}
\definecolor{hrefcolor}{rgb}{0.0,0.5,0.8}
\definecolor{hlgreen}{rgb}{0,0.7,0}

\ifdefined\tikz
\else
    \ifpdf
        \usepackage[pdftex]{graphicx}
    \else
        \usepackage{graphicx}
        
    \fi

\fi

\newenvironment{enumroman}
    {
     
     \begin{enumerate}
    }
    {\end{enumerate}}

\newcounter{remcount}

\newtheorem{theorem}{Theorem}

\newtheorem{lemma}{Lemma}
\newtheorem{proposition}{Proposition}

\theoremstyle{definition}

\newtheorem*{assumption*}{Assumption}
\newtheorem{remark}{Remark}
\newtheorem*{remark*}{Remark}
\newtheorem*{definition*}{Definition}
\newtheorem{example}{Example}
\newtheorem{algorithm}{Algorithm}

\numberwithin{equation}{section}
\numberwithin{lemma}{section}
\numberwithin{theorem}{section}
\numberwithin{proposition}{section}
\numberwithin{definition}{section}
\numberwithin{remark}{section}
\numberwithin{example}{section}
\numberwithin{assumption}{section}
\numberwithin{algorithm}{section}
\numberwithin{corollary}{section}

\newcommand{\term}{\emph}

\newcommand{\field}[1]{\mathbb{#1}}

\newcommand{\R}{\field{R}}
\newcommand{\B}{B}

\newcommand{\C}{\field{C}}
\newcommand{\norm}[1]{\|#1\|}

\newcommand{\abs}[1]{|#1|}
\newcommand{\adaptabs}[1]{\left|#1\right|}

\newcommand{\inv}[1]{#1^{-1}}
\newcommand{\grad}[1]{\nabla #1}

\newcommand{\Union}\bigcup
\newcommand{\Isect}\bigcap
\newcommand{\union}\cup
\newcommand{\isect}\cap
\newcommand{\bigunion}\bigcup
\newcommand{\bigisect}\bigcap

\newcommand{\defeq}{:=}

\newcommand{\downto}{\searrow}
\newcommand{\upto}{\nearrow}
\newcommand{\subdiff}{\partial}

\newcommand{\FF}{\mathcal{F}}

\DeclareMathOperator*{\lip}{lip}
\DeclareMathOperator*{\argmin}{arg\,min}

\DeclareMathOperator{\interior}{int}

\DeclareMathOperator{\Dom}{dom}

\DeclareMathOperator{\graph}{Graph}

\makeatletter
\def \uminus@sym{\setbox0=\hbox{$\cup$}\rlap{\hbox 
        to\wd0{\hss\raise0.5ex\hbox{$\scriptscriptstyle{-}$}\hss}}\box0}
    \def \uminus    {\mathrel{\uminus@sym}}
\makeatother

\newcommand{\mathvar}[1]{\textup{#1}}

\newcommand*{\doi}[1]{doi:\detokenize{#1}}

\makeatletter
\newcommand{\raisemath}[1]{\mathpalette{\raisem@th{#1}}}
\newcommand{\raisem@th}[3]{\raisebox{#1}{$#2#3$}}
\makeatother

\def\opt#1{\widetilde #1}
\def\realopt#1{\widehat #1}
\def\NL{\mathvar{NL}}
\def\LIN{\mathvar{L}}
\def\Ynl{Y_{\NL}}
\def\Ylin{Y_{\LIN}}
\def\Pnl{P_{\NL}}
\def\barG{{\overline G}}
\def\barF{{\overline F}}
\def\FF{\mathcal{F}}
\newcommand{\setto}{\rightrightarrows}
\newcommand{\varsym}{\bullet}
\newcommand{\alphavec}{{\vec\alpha}}
\def\ee{\text{\textsc{e}}}
\def\extR{\overline \R}

\def\realoptu{{\realopt{u}}}
\def\realoptx{{\realopt{x}}}
\def\realopty{{\realopt{y}}}
\def\nextu{{{u}^{i+1}}}
\def\nextx{{{x}^{i+1}}}
\def\nexty{{{y}^{i+1}}}
\def\thisu{{{u}^i}}
\def\thisx{{{x}^i}}
\def\thisy{{{y}^i}}

\def\thisoptu{{\opt{u}^i}}
\def\thisoptx{{\opt{x}^i}}
\def\thisopty{{\opt{y}^i}}
\def\nextrealoptu{{\realopt{u}^{i+1}}}

\def\thisrealoptu{{\realopt{u}^i}}

\def\overx{x_\OverRelax}

\def\overnextx{x_\OverRelax^{i+1}}

\def\basex{{\bar x}}

\def\baseu{{\bar u}}
\def\basextwo{{\bar x'}}
\def\baseutwo{{\bar u'}}
\def\MMax{\Theta}
\def\MMin{\theta}
\def\Rad{R}
\def\OverRelax{\omega}
\def\lipnum{\ell}
\renewcommand{\lip}[1]{\lipnum_{#1}}
\def\lipopt{\lip{\inv H_\realoptx}}
\def\lipoptF{\lipnum^*}

\def\MCond{\kappa}
\def\Ti{S_i}
\def\EE{E}
\def\KL{K_\LIN}
\def\gammalin{\gamma}
\def\rhoC{r}
\def\yLIN{\varphi}
\def\yNL{\lambda}

\renewcommand{\tilde}{\widetilde}

\newcommand{\iprod}[2]{\langle #1,#2\rangle}
\DeclareMathOperator{\Sym}{Sym}

\newcommand{\Meas}{\mathcal{M}}

\newcommand{\BD}{\partial}
\newcommand{\TGV}{\mathvar{TGV}}
\newcommand{\TV}{\mathvar{TV}}

\makeatletter
\def \weaktostar@sym{\setbox0=\hbox{$\rightharpoonup$}\rlap{\hbox 
        to\wd0{\hss\raise1ex\hbox{$\scriptscriptstyle{*\,}$}\hss}}\box0}
    \def \weaktostar    {\mathrel{\weaktostar@sym}}
\makeatother

\begin{document}

\title{A primal-dual hybrid gradient method for non-linear operators with applications to MRI}

\author{Tuomo Valkonen\thanks{
    Department of Applied Mathematics and Theoretical Physics,
    University of Cambridge, United Kingdom.
    \\
    E-mail: \texttt{tuomo.valkonen@iki.fi}.
    }
}

\maketitle

\begin{abstract}
    We study the solution of minimax
    problems $\min_x \max_y G(x) + \iprod{K(x)}{y} - F^*(y)$
    in finite-dimensional Hilbert spaces. The functionals $G$ and
    $F^*$ we assume to be convex, but the operator $K$ we allow to
    be non-linear.
    We formulate a natural extension of the modified primal-dual hybrid 
    gradient method (PDHGM), originally for linear $K$, due to
    Chambolle and Pock. We prove the local convergence of the method,
    provided various technical conditions are satisfied. These include
    in particular the Aubin property of the inverse of a monotone operator
    at the solution. Of particular interest to us is the case
    arising from Tikhonov type regularisation of inverse problems with 
    non-linear forward operators. Mainly we are interested in total variation 
    and second-order total generalised variation priors.
    For such problems, we show that our general local
    convergence result holds when the noise level of the data $f$ is low,
    and the regularisation parameter $\alpha$ is correspondingly small.
    We verify the numerical performance of the 
    method by applying it to problems from magnetic resonance imaging
    (MRI) in chemical engineering and medicine. The specific applications
    are in diffusion tensor imaging (DTI) and MR velocity imaging. 
    These numerical studies show very promising performance.
    
    \paragraph{Mathematics subject classification:}
        49M29, % Methods involving duality
        90C26, % Nonconvex programming, global optimization
        92C55. % Biomedical imaging and signal processing.

    \paragraph{Keywords:} 
        primal-dual, 
        non-linear,
        non-convex, 
        convergence, 
        MRI.
    
\end{abstract}

\section{Introduction}

Let us be given convex, proper, lower semicontinuous functionals
$G: X \to \extR$ and $F^*: Y \to \extR$ on finite-dimensional Hilbert
spaces $X$ and $Y$. We then wish to solve the minimax problem
\begin{equation}
    \label{eq:nonlinear-problem-intro}
    \min_x \max_y\ G(x) + \iprod{K(x)}{y} - F^*(y),
\end{equation}
where we allow the operator $K \in C^2(X; Y)$ to be non-linear.
If $K$ were linear, this problem could be solved, among others, 
by the primal-dual method due to Chambolle and Pock \cite{chambolle2010first}. 
In Section \ref{sec:basics} of this paper, we derive two
extensions of the method for non-linear $K$. 

The aforementioned Chambolle-Pock algorithm is an inertial primal-dual backward-backward 
splitting method, classified in \cite{esser2010general} as the modified primal-dual hybrid 
gradient method (PDHGM). It can also seen as a preconditioned ADMM (alternating
directions method of multipliers). In the linear case, for step sizes $\tau,\sigma>0$,
each iteration of the algorithm consists of the updates
\begin{align*}
    \nextx & \defeq (I+\tau \subdiff G)^{-1}(\thisx - \tau K^* y^{i}),
    \\
    \overnextx & \defeq \nextx + \OverRelax(\nextx-\thisx),
    \\
    \nexty & \defeq (I+\sigma \subdiff F^*)^{-1}(\thisy + \sigma K \overnextx).
\end{align*}
The first and last update are the backward (proximal) steps for the primal 
($x$) and dual ($y$) variables, respectively, keeping the other fixed. 
However, the dual step is not taken with the primal variable fixed at $\nextx$
but at the point $\overnextx$. This includes some ``inertia'' or over-relaxation
from the previous iterate $\thisx$, as specified by the parameter $\OverRelax$.
Doing so, the algorithm can to some extend avoid the problem common to first-order
methods that the steps become increasingly shorter. If $G$ or $F^*$ is
uniformly convex, by smartly choosing for each iteration the step length
 parameters $\tau,\sigma$, and the inertia $\OverRelax$, the method can be
shown to have convergence rate $O(1/N^2)$. This is similar to Nesterov's
optimal gradient method \cite{nesterov1983method}. In the general case the
rate is $O(1/N)$. In practise the method has rather good properties on
imaging problems, producing solutions of acceptable visual quality 
relatively quickly.

Our first, simpler extension of the algorithm for non-linear $K$ 
consists of the analogous updates
%\begin{subequations}
%\label{eq:nl-cp-intro}
\begin{align*}
    \nextx & \defeq (I+\tau \subdiff G)^{-1}(\thisx - \tau [\grad K(\thisx)]^* y^{i}),
    \\
    \overnextx & \defeq \nextx + \OverRelax(\nextx-\thisx),
    \\
    \nexty & \defeq (I+\sigma \subdiff F^*)^{-1}(\thisy + \sigma K(\overnextx)).
\end{align*}
The second variant linearises $K(\overnextx)$.
Through a technical analysis in Section \ref{sec:detailed}, we prove 
local convergence of both variants of the method to critical points $(\realoptx, \realopty)$ 
of the system \eqref{eq:nonlinear-problem-intro}. To do this, in addition 
to trivial conditions familiar from the linear case, we need two non-trivial
estimates. For one, defining the set-valued map
\begin{equation}
    %\label{eq:h-intro}
    \notag
    H_\realoptx(x, y) \defeq
        \begin{pmatrix}
            \subdiff G(x) + \grad K(\realoptx)^* y \\
            \subdiff F^*(y) -\grad K(\realoptx) x - c_\realoptx,
        \end{pmatrix},
    \quad
    \text{where}
    \quad
    c_\realoptx \defeq \grad K(\realoptx)\realoptx-K(\realoptx),
\end{equation}
we require that the inverse $\inv H_\realoptx$ is
pseudo-Lipschitz \cite{aubin1990sva}, a condition
also known as the Aubin property \cite{rockafellar-wets-va}.
The second, more severe, estimate is that the dual variable
$\realopty$ has to be small in the range of the non-linear 
part of $K$. In Section \ref{sec:lipschitz}, we study the 
satisfaction of these estimates for $G$, $F^*$ and $K$ of 
forms most relevant to image processing applications
that we study numerically in Section \ref{sec:appl}.

Problems of the form \eqref{eq:nonlinear-problem-intro} with non-linear
$K$ arise, for instance, from various inverse problems in magnetic resonance 
imaging (MRI). As a motivating example, we introduce the following problem 
from velocity-encoded MRI. Other applications include the modelling of
the Stejskal-Tanner equation in diffusion tensor imaging (DTI).
We will discuss this application in more detail in Section \ref{sec:appl}.
In velocity-encoded MRI, we seek to reconstruct a complex image
$v=\rhoC \exp(i\varphi) \in L^1(\Omega; \C)$ from sub-sampled $k$-space
(Fourier transform) data $f$. In this application we are chiefly interested in
the  phase $\varphi$, and eventually the difference of phases of two
suitably acquired images, as the velocity of an imaged fluid can be
encoded into the phase difference \cite{Holland2010}.
Let us denote by $S$ the sub-sampling operator, and by $\FF$ 
the Fourier transform.
We observed in \cite{tuomov-phaserec} that instead of, let's say, 
defining total variation ($\TV$) for complex-valued functions $v$ 
similarly to vector-valued functions, and solving
\begin{equation}
    \label{eq:phase-recon-lin}
    \min_v \frac{1}{2}\norm{f-S\FF v}^2 + \alpha \TV(v),
\end{equation}
it may be better to regularise $\rhoC$ and $\varphi$ differently.
This leads us to the problem
\begin{equation}
    \label{eq:phase-recon}
    \min_{\rhoC, \varphi} \frac{1}{2}\norm{f-T(\rhoC, \varphi)} + \alpha_\rhoC R_\rhoC(\rhoC) + \alpha_\varphi R_\varphi(\varphi).
\end{equation}
Here $R_\rhoC$ and $R_\varphi$ are suitable regularisation functionals
for the amplitude map $\rhoC$ and phase map $\varphi$, respectively, of a complex
image $v=\rhoC \exp(i\varphi)$. Correspondingly, we define the operator $T$ by
\[
    T(\rhoC, \varphi) \defeq S\FF\bigl(x \mapsto \rhoC(x) \exp(i\varphi(x))\bigr).
\]
Observe that we may rewrite
\[
    \frac{1}{2}\norm{f-T(\rhoC,\varphi)}^2
    =
    \max_\lambda\Bigl( \iprod{T(\rhoC, \varphi)}{\lambda}
        - \iprod{f}{\lambda}
        - \frac{1}{2}\norm{\lambda}^2\Bigr).
\]
Generally also the regularisation terms can be written in terms of an
indicator function and a bilinear part in the form
\[
    \alpha_\rhoC R_\rhoC(\rhoC) = \max_{\psi}\, \iprod{K_\rhoC \rhoC}{\psi} - \delta_{C_\rhoC}(\psi).
\]
In case of total variation regularisation, $R_\rhoC(\rhoC)=\TV(\rhoC)$, we
have $C_\rhoC=\{\psi \mid \sup_x \norm{\varphi(x)} \le \alpha_\rhoC\}$ and $K_\rhoC=\grad$.
With these transformations, the problem \eqref{eq:phase-recon} 
can be written in the form \eqref{eq:nonlinear-problem-intro} with $G \equiv 0$,
\[
    K(\rhoC, \varphi)=(T(\rhoC, \varphi), K_\rhoC \rhoC, K_\varphi \varphi),
\]
and
\[
    F^*(\lambda, \psi_r, \psi_\varphi)=\iprod{f}{\lambda} + \frac{1}{2}\norm{\lambda}^2
        + \delta_{C_\rhoC}(\psi_r) + \delta_{C_\rhoC}(\psi_\varphi).
\]
Observe that $F^*$ is strongly convex in the range of the non-linear
part of $K$, corresponding to $T$. Under exactly this kind of structural
assumptions, along with strict complementarity and non-degeneracy
assumptions from the solution, we can show in Section \ref{sec:lipschitz} 
that $\inv H_\realoptx$ possesses the Aubin property required for the 
general convergence theorem, Theorem \ref{thm:conv}, to hold. 
Moreover, in this case the condition on $\realopty$ being small in
the non-linear range of $K$ corresponds to $\norm{f-T(\realopt\rhoC, \realopt\varphi)}$ 
being small. This can be achieved under low noise and a small regularisation parameter.

Computationally \eqref{eq:phase-recon} is significantly more demanding than
\eqref{eq:phase-recon-lin}, as it is no longer a convex problem due to the
non-linearity of $T$. One option for locally solving problems of the
form \eqref{eq:phase-recon} is the Gauss-Newton scheme.
In this, one linearises $T$ at the current iterate, solves the resulting
convex problem, and repeats until a stopping criterion is satisfied. 
Computationally such schemes combining inner and outer iterations are
expensive, unless one can solve the inner iterations to a very low accuracy.
Moreover, the Gauss-Newton scheme is not guaranteed to
converge even locally -- a fact that we did occasionally observe when
performing the numerical experiments for Section \ref{sec:appl}. 
The scheme is however very  useful when combined with iterative 
regularisation, and behaves in that case well for almost linear $K$ 
\cite{blaschke1997convergence,schuster2012regularization}.
It can even be combined with Bregman iterations for contrast 
enhancement \cite{bachmayr2009iterative}.
Unfortunately, our operators of interest are not almost linear.
Another possibility for the numerical solution of \eqref{eq:phase-recon}
would be an infeasible semismooth Newton method,
along the lines of \cite{hintermuller2006infeasible}, extended to
non-linear operators. However, second-order methods quickly become
prohibitively expensive as the image size increases, unless one
can employ domain decomposition techniques -- something that to
our knowledge has not yet been done for semismooth Newton methods
relevant to total variation type regularisation.
Based on this, we find it desirable to start developing more efficient 
and provably convergent methods for non-convex problems on large 
data sets. We now study one possibility.
%By studying one option, we now take up that task.

\section{The basics}
\label{sec:basics}

We describe the proposed method, Algorithm \ref{algorithm:nl-cp},
in Section \ref{sec:proposed} below. To begin its analysis, we study 
in Section \ref{sec:linearised} the application of the Chambolle-Pock
method to linearisations of our original
problem \eqref{eq:nonlinear-problem-intro}. We then derive in 
Section \ref{sec:basic-est} basic descent estimates that motivate a 
general convergence result, Theorem \ref{thm:conv-idea}, stated in 
Section \ref{sec:idea}. This result will form the basis of the proof
of convergence of Algorithm \ref{algorithm:nl-cp}. 
Our task in the  following Section \ref{sec:detailed} will be to to 
derive the estimates required by Theorem \ref{thm:conv-idea}. 
We follow the theorem with a collection of remarks in 
Section \ref{sec:basics-remarks}.

\subsection{The proposed method}
\label{sec:proposed}

Let $X$ and $Y$ be finite-dimensional Hilbert spaces. Suppose
we are given two convex, proper, lower-semicontinuous functionals
$G: X \to \extR$ and $F^*: X \to \extR$, and a possibly non-linear
operator $K \in C^2(X; Y)$. 
We are interested  in solving the problem
\begin{equation}
    \label{eq:nonlinear-problem}
    \tag{P}
    \min_x \max_y\, G(x) + \iprod{K(x)}{y} - F^*(y).
\end{equation}
The first-order optimality conditions
for $(\hat x, \hat y)$ to solve \eqref{eq:nonlinear-problem} 
may be formally derived as
\begin{subequations}
\label{eq:nl-oc}
\begin{align}
    -[\grad K(\hat x)]^* \hat y & \in \subdiff G(\hat x), 
    \\
    K(\hat x) & \in \subdiff F^*(\hat y).
\end{align}
\end{subequations}
% Rockafellar-Wets, page 429 ==> constraint qualification satisfied
%       if $K$ has full range or $F$ has full domain.
% J(x)=G(x)+F(K(x))
% 0 \in \subdiff J(x)
% <=> 0 \in \subdiff G(x) + \subdiff [\grad K(x)]^* (\subdiff F)(K(x))
% <=> -[\grad K(x)]^* y \in \subdiff G(x) where y \in (\subdiff F)(K(x))
% But y \in (\subdiff F)(K(x))  < => K(x) \in \subdiff F^*(y).
Under a constraint qualification, which is satisfied for example
when $G$ is $C^1$ and either $[\grad K(x)]^*$ has empty nullspace or $\Dom F=X$, 
these conditions can be seen to be necessary; cf. \cite[10.8]{rockafellar-wets-va}.
For linear $K$ in particular, the conditions are necessary and reduce to
\begin{subequations}
\label{eq:lin-oc}
\begin{align}
    -K^* \hat y & \in \subdiff G(\hat x), 
    \\
    K \hat x & \in \subdiff F^*(\hat y).
\end{align}
\end{subequations}
The modified primal-dual hybrid gradient method (PDHGM) due to \citet{chambolle2010first},
solves this problem by iterating for $\sigma, \tau > 0$ satisfying
$\sigma\tau\norm{K}^2 < 1$ the system
\begin{subequations}
\label{eq:cp-lin}
\begin{align}
    \nextx & \defeq (I+\tau \subdiff G)^{-1}(\thisx - \tau K^* y^{i}),
    \\
    \overnextx & \defeq \nextx + \OverRelax(\nextx-\thisx),
    \\
    \nexty & \defeq (I+\sigma \subdiff F^*)^{-1}(\thisy + \sigma K \overnextx).
\end{align}
\end{subequations}
As such, the method is closely related to a large class
of methods including in particular the Uzawa method and the alternating direction 
method of multipliers (ADMM). For an overview, we recommend \cite{esser2010general}. 

In case the reader is wondering, the order of the primal ($x$) and dual ($y$)
updates in \eqref{eq:cp-lin} is reversed from the original presentation 
in \cite{chambolle2010first}. The reason is that reordered the updates can,
as discovered in \cite{he2012convergence}, be easily written in a proximal 
point form. We will exploit this. Indeed, \eqref{eq:cp-lin} already contains 
two proximal point sub-problems, specifically the computation of the resolvents
$(I+\tau \subdiff G)^{-1}$ and $(I+\sigma \subdiff F^*)^{-1}$. We recall that they may be written
as
\[
    (I+\tau \subdiff G)^{-1}(x) 
    = \argmin_{x'}\left\{\frac{\norm{x'-x}^2}{2\tau} + G(x')\right\}.
\]
For the good performance of \eqref{eq:cp-lin}, it is crucial that 
these sub-problems can be solved efficiently. Usually in applications,
they turn out to be simple projections or linear operations.
Resolvents reducing to small pointwise quadratic semidefinite 
problems have also been studied \cite{ipmsproc,tuomov-scaleproj-report}.

Observe the correspondence between the (merely necessary) optimality 
conditions \eqref{eq:nl-oc} for the problem \eqref{eq:nonlinear-problem}
with non-linear $K$ and the optimality conditions \eqref{eq:lin-oc} for 
the linear case. It suggests that we could obtain a numerical method
for solving \eqref{eq:nl-oc} by replacing the applications
$K^* \thisy$ and $K\overnextx$ in \eqref{eq:cp-lin}
by $[\grad K(\thisx)]^* \thisy$ and $K(\overnextx)$.
We would thus linearise the dual application, but keep the primal
application non-linear. We do exactly that and propose 
the following method.

\begin{algorithm}[Exact NL-PDHGM]
    \label{algorithm:nl-cp}
    Choose $\OverRelax \ge 0$ and $\tau, \sigma > 0$.
    Repeat the following steps until a convergence criterion is satisfied.
    \begin{subequations}
    \label{eq:nl-cp}
    \begin{align}
        \label{eq:nl-cp1}
        \nextx & \defeq (I+\tau \subdiff G)^{-1}(\thisx - \tau [\grad K(\thisx)]^* y^{i}),
        \\
        \overnextx & \defeq \nextx + \OverRelax(\nextx-\thisx),
        \\
        \label{eq:nl-cp3}
        \nexty & \defeq (I+\sigma \subdiff F^*)^{-1}(\thisy + \sigma K(\overnextx)).
    \end{align}
    \end{subequations}
\end{algorithm}

In practise we require $\OverRelax=1$. 
Exact conditions on the step length parameters $\tau$ and $\sigma$ will be 
derived in Section \ref{sec:detailed} along the course of the proof of 
local convergence; in the numerical experiments of Section \ref{sec:appl}, 
we make $\tau$ and $\sigma$ depend on the iteration, choosing 
$\tau^i$ and $\sigma^i$ to satisfy
\[
    \sigma^i \tau^i \left(\sup_{k=1,\ldots,i} \norm{\grad{K(x^k)}}^2\right) < 1
\]
with the ratio $\sigma^i/\tau^i$ unaltered. We discuss the justification for this kind of 
strategies in Remark \ref{rem:steplength} after the convergence proof.

We will base the convergence proof on the following fully linearised 
method, where we replace the application $K(\overnextx)$ 
also by a linearisation. 

%Usually
%$\OverRelax=1$, and in the case of linear $K$, the primal and dual
%step length parameters $\tau$ and $\sigma$ are required to satisfy
%\begin{equation}
%    \sigma\tau\norm{K}^2 < 1.
%\end{equation}
%In the non-linear case, since $K$ is not generally Lipschitz on $X$,
%but merely on compact subsets, we will localise this condition.

%Indeed, our proof
%will also show the convergence of the fully linearised approach
%to solution of \eqref{eq:nl-oc}, which may be stated as follows.

\begin{algorithm}[Linearised NL-PDHGM]
    \label{algorithm:nl-cp-lin}
    Choose $\OverRelax \ge 0$ and $\tau, \sigma > 0$.
    % such that \eqref{eq:tau-sigma-cond} holds.
    Repeat the following steps until a convergence criterion is satisfied.
    \begin{subequations}
    \label{eq:nl-cp-lin}
    \begin{align}
        \label{eq:nl-cp1-lin}
        \nextx & \defeq (I+\tau \subdiff G)^{-1}(\thisx - \tau [\grad K(\thisx)]^* y^{i}),
        \\
        \overnextx & \defeq \nextx + \OverRelax(\nextx-\thisx),
        \\
        \label{eq:nl-cp3-lin}
        \nexty & \defeq (I+\sigma \subdiff F^*)^{-1}(\thisy + \sigma[K(\thisx)+\grad K(\thisx)(\overnextx-\thisx)]).
    \end{align}
    \end{subequations}
\end{algorithm}

In numerical practise, as we will see in Section \ref{sec:appl}, the convergence
rate of both variants of the algorithm is the same. Algorithm \ref{algorithm:nl-cp}
is however faster in terms of computational time, as it needs less operations per iteration
in the evaluation of $K(\overnextx)$ versus $K(\thisx)+\grad K(\thisx)(\overnextx-\thisx)$.

\subsection{Linearised problem and proximal point formulation}
\label{sec:linearised}

To start the convergence analysis of Algorithm \ref{algorithm:nl-cp},
we study the application of the standard Chambolle-Pock
method \eqref{eq:cp-lin} to linearisations of problem
\eqref{eq:nonlinear-problem} at a base point $\basex \in X$.
Specifically, we define
\[
    K_\basex \defeq \grad K(\basex),
    \quad
    \text{and}
    \quad
    c_\basex \defeq K(\basex) - K_\basex \basex.
\]
Then we consider
\begin{equation}
    \label{eq:linearised-problem}
    \min_x \max_y G(x) + \iprod{c_\basex+K_\basex x}{y} - F^*(y).
\end{equation}
This problem is of the form required by the method \eqref{eq:cp-lin}, 
if we write $F_\basex^*(y)=F^*(y)-\iprod{c_\basex}{y}$.
%The necessary and sufficient optimality conditions for
%\eqref{eq:linearised-problem} are
%\begin{subequations}
%\label{eq:nl-lin-oc}
%\begin{align}
%    \notag
%    -K_\basex^* \hat y & \in \subdiff G(\hat x), 
%    \\
%    \notag
%    c_\basex + K_\basex \hat x & \in \subdiff F^*(\hat y).
%\end{align}
%\end{subequations}
Indeed, we may write the updates \eqref{eq:cp-lin} for this problem as
\begin{subequations}
\label{eq:nl-lin-cp}
\begin{align}
    \label{eq:nl-lin-cp1}
    \nextx & \defeq (I+\tau \subdiff G)^{-1}(\thisx - \tau K_\basex^* y^{i}),
    \\
    \overnextx & \defeq \nextx + \OverRelax(\nextx-\thisx),
    \\
    \label{eq:nl-lin-cp3}
    \nexty & \defeq (I+\sigma \subdiff F^*)^{-1}(\thisy + \sigma (c_\basex + K_\basex \overnextx)).
\end{align}
\end{subequations}
Observe how \eqref{eq:nl-cp3-lin} corresponds to \eqref{eq:nl-lin-cp3} with $\basex=\thisx$.

From now on we use the general notation
\[
    u=(x, y),
\]
and define
\[
    H_\basex(u) \defeq
        \begin{pmatrix}
            \subdiff G(x) + K_\basex^* y \\
            \subdiff F^*(y) -K_\basex x - c_\basex
        \end{pmatrix}
\]
as well as
\[
    M_\basex \defeq
        \begin{pmatrix}
            I/\tau & -K_\basex^* \\
            -\OverRelax K_\basex & I/\sigma
        \end{pmatrix}.
\]
With these operators, $0 \in H_\basex(\realoptu)$ characterises solutions
$\realoptu$ to \eqref{eq:linearised-problem}, and $\nextu$ computed 
by \eqref{eq:nl-lin-cp} is according to \cite{he2012convergence}
characterised as the unique solution to the proximal point problem 
\begin{equation}
    \label{eq:prox-update}
    0 \in H_\basex(\nextu) + M_\basex(\nextu-\thisu).
\end{equation}
%From this, it is easy to show for $\OverRelax=1$ the convergence of the system
%of updates \eqref{eq:nl-lin-cp} to a solution of the linearised 
%problem \eqref{eq:linearised-problem}. 
%
In fact, returning to the original problem \eqref{eq:nonlinear-problem},
the optimality conditions \eqref{eq:nl-oc} may be written
\[
    0 \in H_\realoptx(\realoptu),
\]
and \eqref{eq:prox-update} with $\basex=\thisx$ characterises the 
update \eqref{eq:nl-cp-lin} of Linearised NL-PDHGM (Algorithm \ref{algorithm:nl-cp-lin}). 
For the update \eqref{eq:nl-cp} of Exact NL-PDHGM (Algorithm \ref{algorithm:nl-cp}),
we derive the characterisation
\begin{equation}
    \label{eq:nl-pseudoprox-update}
    0 \in H_{\thisx}(\nextu) + D_{\thisx}(\nextu) + M_{\thisx}(\nextu-\thisu)
\end{equation}
with
\begin{equation}
    \label{eq:dbasex}
    D_{\basex}(x, y) \defeq 
        \begin{pmatrix}
            0 \\
            K_\basex \overx + c_\basex - K(\overx)
        \end{pmatrix}
        =
        \begin{pmatrix}
            0 \\
            K(\basex) + \grad K(\basex)(\overx -\basex) - K(\overx)
        \end{pmatrix}
\end{equation}
and $\overx \defeq x + \OverRelax(x - \basex)$.
We therefore study next basic estimates that can be obtained 
from \eqref{eq:prox-update} with an additional general discrepancy term $\nu^i$.
These form the basis of our convergence proof.

\subsection{Basic descent estimate}
\label{sec:basic-est}

We now fix $\OverRelax=1$ in order to force $M_\basex$ symmetric,
and study solutions $\nextu$ to the general system
\begin{equation}
    \label{eq:prox-update-discrepancy}
    0 \in H_{\basex}(\nextu) + \nu^i + M_{\basex}(\nextu-\thisu).
\end{equation}
In Lemma \ref{lemma:descent} below, we show that $\nextu$ is better
than $\thisu$ in terms of distance to the ``perturbed
local solution'' $\thisoptu$ solving $0 \in H_\basex(\thisoptu) + \nu^i$.
Here we use the word \emph{perturbation} to refer to $\nu^i$, and 
\emph{local} to refer to the linearisation point $\basex$.
Observe that $\thisoptu$ depends on both $\thisu$ and $\nextu$ in case of
Algorithm \ref{algorithm:nl-cp}, resp.~\eqref{eq:nl-pseudoprox-update}.
In Section \ref{sec:detailed} we will lessen these dependencies, and
convert the statement to be in terms of local (unperturbed) optimal 
solutions $\thisrealoptu$, satisfying $0 \in H_\basex(\thisrealoptu)$.

For the statement of the lemma, we use the notation
\[
    \iprod{a}{b}_M \defeq \iprod{a}{Mb},
    \quad
    \text{and}
    \quad
    \norm{a}_M \defeq \sqrt{\iprod{a}{a}_M},
\]
and denote by $P_V$ the linear projection operator into a subspace $V$ of $Y$.
We also say that $F^*$ is strongly convex on the subspace $V$ with constant $\gamma > 0$ if
\[
    F^*(y')-F^*(y) \ge \iprod{z}{y'-y} + \frac{\gamma}{2}\norm{P_V(y'-y)}^2
    \quad
    \text{for all } y, y' \in Y \text{ and } z \in \subdiff F^*(y).
\]
This is equivalent to saying that the operator $\subdiff F^*$ is strongly monotone 
on $V$ in the sense that
\[
    \iprod{\subdiff F^*(y') - \subdiff F^*(y)}{y'-y}
    \ge
     \frac{\gamma}{2}\norm{P_V(y'-y)}^2
    \quad
    \text{for all } y, y' \in Y.
\]

\begin{lemma}
    \label{lemma:descent}
    Fix $\OverRelax=1$.
    Let $\thisu \in X \times Y$ and $\basex \in X$.
    Suppose $\nextu \in X \times Y$ solves \eqref{eq:prox-update-discrepancy} 
    for some $\nu^i \in X \times Y$, and that $\thisoptu \in X \times Y$ is 
    a solution to
    \begin{equation}
        \label{eq:optu-di}
        0 \in H_\basex(\thisoptu) + \nu^i.
    \end{equation}
    Then
    \begin{equation}
        \label{eq:linear-descent}
        \tag{$\mathrm{\opt{D}^2}$-loc}
        \norm{\thisu-\thisoptu}_{M_\basex}^2
        \ge \norm{\nextu-\thisu}_{M_\basex}^2
            + \norm{\nextu-\thisoptu}_{M_\basex}^2.
            % - 2\iprod{\nextu-\thisoptu}{\nu^i}.
    \end{equation}
    If $F^*$ is additionally strongly convex on a subspace $V$ of $Y$ 
    with constant $\gamma > 0$, then we have
    \begin{equation}
        \label{eq:linear-descent-strong}
        \tag{$\mathrm{\opt{D}^2}$-loc-$\gamma$}
        \norm{\thisu-\thisoptu}_{M_\basex}^2
        \ge \norm{\nextu-\thisu}_{M_\basex}^2
            + \norm{\nextu-\thisoptu}_{M_\basex}^2
            % - 2\iprod{\nextu-\thisoptu}{\nu^i}
            + \frac{\gamma}{2}\norm{P_V(\nexty-\thisopty)}^2.
    \end{equation}
\end{lemma}    
\begin{proof}
    Since the operator $H_\basex$ is monotone, we have
    \begin{equation}
        \label{eq:h-monotonicity-iterate}
        \iprod{(H_\basex(\nextu)+\nu^i)-(H_\basex(\thisoptu)+\nu^i)}{\nextu-\thisoptu} 
        =
        \iprod{H_\basex(\nextu)-H_\basex(\thisoptu)}{\nextu-\thisoptu} 
        \ge 0.
    \end{equation}
    It thus follows from \eqref{eq:prox-update-discrepancy}, \eqref{eq:optu-di},
    and the symmetricity of $M_\basex$ that
    \[
        0 \ge \iprod{\nextu-\thisoptu}{\nextu-\thisu}_{M_\basex} = \iprod{\nextu-\thisu}{\nextu-\thisoptu}_{M_\basex}.
    \]
    Consequently
    \begin{equation}
        \label{eq:initial-descent-estimate}
        \begin{split}
        %-\iprod{\nextu-\optu}{\nu^i}
        0
        &
        \ge \norm{\nextu-\thisu}_{M_\basex}^2
             + \iprod{\thisu-\thisoptu}{\nextu-\thisu}_{M_\basex}
        \\
        &
        = \norm{\nextu-\thisu}_{M_\basex}^2
            - \norm{\thisu-\thisoptu}_{M_\basex}^2
             + \iprod{\thisu-\thisoptu}{\nextu-\thisoptu}_{M_\basex},
        \\
        &
        = \norm{\nextu-\thisu}_{M_\basex}^2
            - \norm{\thisu-\thisoptu}_{M_\basex}^2
            + \norm{\nextu-\thisoptu}_{M_\basex}^2
             + \iprod{\thisu-\nextu}{\nextu-\thisoptu}_{M_\basex},
        \\
        &
        \ge \norm{\nextu-\thisu}_{M_\basex}^2
            - \norm{\thisu-\thisoptu}_{M_\basex}^2
            + \norm{\nextu-\thisoptu}_{M_\basex}^2.
        \end{split}
    \end{equation}
    This yields \eqref{eq:linear-descent}.
    The strong convexity estimate \eqref{eq:linear-descent-strong} is proved analogously,
    using the fact that instead of \eqref{eq:h-monotonicity-iterate}, we have the stronger
    estimate
    \[
        \iprod{H_\basex(\nextu)-H_\basex(\thisoptu)}{\nextu-\thisoptu} 
        \ge \frac{\gamma}{2}\norm{P_V(\nexty-\thisopty)}^2.
        \qedhere
    \]
\end{proof}

Following \cite{konnov2003proximal}, see also \cite{rockafellar1976monotone},
if $\nu^i=0$, then it is not difficult to show from \eqref{eq:linear-descent} the 
convergence of the iterates $\{\thisu\}_{i=0}^\infty$ generated by \eqref{eq:prox-update-discrepancy} 
to a solution %(possibly distinct from $\optu$) 
of the linearised problem \eqref{eq:linearised-problem}.
The estimate \eqref{eq:linear-descent} also forms the basis of our proof of
local convergence of Algorithm \ref{algorithm:nl-cp} and Algorithm \ref{algorithm:nl-cp-lin}.
However, we have to improve upon it to take into account 
that $\basex=\thisx$ changes on each iteration in \eqref{eq:nl-cp}, and
that the dual update in \eqref{eq:nl-cp3} is not linearised.
The consequence of these changes is that also the weight operator
$M_\thisx$ of the local norm $\norm{\cdot}_{M_\thisx}$ changes
on each iteration, as do the local perturbed solution $\thisoptu$
and the local (unperturbed) solution $\thisrealoptu$.
Taking these differences into account, it turns out that the correct 
estimate that we have to derive is \eqref{eq:nonlinear-descent-general}
in the next theorem. There we have improved \eqref{eq:linear-descent} 
by these changes and additionally, due to proof-technical reasons, 
by the removal of the squares on the norms.

%It turns out, that it doesn't matter that the local solution $\thisoptu$
%depend on $\nu^i$, as long as they remain bounded and $\nu^i \to 0$.

\subsection{Idea of convergence proof}
\label{sec:idea}

\begin{theorem}
    \label{thm:conv-idea}
    Suppose that the operator $K \in C^1(X; Y)$ and the constants
    $\sigma,\tau >0$ satisfy for some $\MMax > \MMin > 0$ the bounds
    \begin{equation}
        \label{eq:m-bounds}
        \tag{C-M}
        \MMin^2 I \le M_{\thisx} \le \MMax^2 I,
        \quad (i=1,2,3,\ldots).
    \end{equation}
    Let the sequence $\{\thisu\}_{i=1}^\infty$ solve
    \eqref{eq:prox-update-discrepancy}
    %\begin{equation}
    %    \label{eq:conv-idea-satisfy}
    %    0 \in H_{\thisx}(\nextu)+\nu^i + M_{\thisx}(\nextu-\thisu),
    %\end{equation}
    for $\basex=\thisx$ and some $\{\nu^i\}_{i=1}^\infty \subset X \times Y$ satisfying
    \begin{equation}
        \label{eq:conv-idea-di}
        \tag{C-$\nu^i$}
        \lim_{i \to \infty} \nu^i = 0.
    \end{equation}
    Suppose, moreover, that for some constant $\zeta > 0$ and points $\{\thisrealoptu\}_{i=1}^\infty$ 
    %with
    %\begin{equation}
    %    \label{eq:conv-idea-realoptu-bound}
    %    \tag{C-B}
    %    \sup_i \norm{\thisrealoptu} < \infty, 
    %\end{equation}
    we have the estimate
    \begin{equation}
        \label{eq:nonlinear-descent-general}
        \tag{$\mathrm{\realopt{D}}$}
        \norm{\thisu-\thisrealoptu}_{M_{\thisx}}
        \ge \zeta \norm{\nextu-\thisu}_{M_{\thisx}}
            + \norm{\nextu-\nextrealoptu}_{M_{\nextx}}.
    \end{equation}
    Then the iterates $\thisu \to \realoptu$ for some $\realopt u=(\realopt x, \realopt y)$
    that solves \eqref{eq:nl-oc}.
\end{theorem}
\begin{proof}
    %Observe that it follows from \eqref{eq:nonlinear-descent-general} that
    %\[
    %    \OverRelax \sup_i \norm{\thisu-\thisrealoptu}
    %    \le
    %    \sup_i \norm{\thisu-\thisrealoptu}_{M_{\thisx}} < \infty.
    %\]
    %Consequently, the sequence $\{v^i \defeq \thisu-\thisrealoptu\}_{i=1}^\infty$ 
    %is bounded, therefore has limit points. In fact, since
    %$\norm{v^i}_{M_{\thisx}}$ is decreasing, we have
    %\[
    %    \lim_{i \to \infty} \norm{v^i}_{M_{\thisx}} =: \mu \ge 0.
    %\]
    %
    It follows from \eqref{eq:nonlinear-descent-general} that
    \[
        \sum_{i=1}^\infty \norm{\nextu-\thisu}_{M_{\thisx}} < \infty.
    \]
    %because otherwise eventually $\norm{\nextu-\nextrealoptu}_{M_{\nextx}} < 0$,
    %which cannot happen.
    Consequently, an application of \eqref{eq:m-bounds} shows that
    \begin{equation}
        \sum_{i=1}^\infty \norm{\nextu-\thisu}
        \le \MMax \sum_{i=1}^\infty \norm{\nextu-\thisu}_{M_{\thisx}} < \infty.
    \end{equation}
    This says that $\{\thisu\}_{i=1}^\infty$ is a Cauchy sequence, and hence
    converges to some $\realoptu$.
    
    %Let now $(v, \realoptu)$ be a limit point of $\{(v^i, \thisrealoptu)\}_{i=1}^\infty$,
    %that is $v^{i_j} \to v$ and $\realoptu^{i_j} \to \realoptu$. Then 
    %also $u^{i_j} \to v+\realoptu$. Necessarily then $u=v+\realoptu$.
    
    Let
    \[
        z^i \defeq \nu^i + M_{\thisx}(\nextu-\thisu).
    \]
    Since $\nu^i \to 0$ by \eqref{eq:conv-idea-di}, and
    \[
        \norm{M_{\thisx}(\nextu-\thisu)} \le \MMax \norm{\nextu-\thisu} \to 0,
    \]
    it follows that $z_i \to 0$.
    By \eqref{eq:prox-update-discrepancy}, we moreover have $-z^i \in H_{x^{i}}(\nextu)$.
    Using $K \in C^1(X; Y)$ and the outer semicontinuity of the subgradient mappings
    $\subdiff G$ and $\subdiff F^*$, we see that
    \[
        \limsup_{i \to \infty} H_{x^{i}}(\nextu) \subset H_{\realoptx}(\realoptu).
    \]
    Here the $\limsup$ is in the sense of an outer limit \cite{rockafellar-wets-va},
    consisting of the limits of all converging subsequences of elements $v^i \in H_{x^{i}}(\nextu)$.
    As by \eqref{eq:prox-update-discrepancy} we have $-z^i \in H_{x^{i}}(\nextu)$,
    it follows in particular that $0 \in H_{\realoptx}(\realoptu)$.
    This says says precisely that \eqref{eq:nl-oc} holds.
\end{proof}

\subsection{A few remarks}
\label{sec:basics-remarks}

\begin{remark}[Reference points]
    %Observe that the only assumption we needed on the
    %sequence $\{\thisrealoptu\}_{i=1}^\infty$ in the above proof
    %was that of boundedness; we did not need $\thisrealoptu$ to solve
    %anything.
    Observe that we did not need to assume the reference points
    $\{\thisrealoptu\}_{i=1}^\infty$ in the above proof to solve anything.
    We did not even need to assume boundedness, which follows from
    \eqref{eq:nonlinear-descent-general} and \eqref{eq:m-bounds}.
\end{remark}

\begin{comment}
\begin{remark}[Inexact solutions]
    The primary reason for the discrepancy $\nu^i$ 
    is the term $D_{\thisx}(\nextu)$ in \eqref{eq:nl-pseudoprox-update},
    but it can also be used to model inexact solutions
    to the iterates that asymptotically become more accurate.
\end{remark}
\end{comment}

\begin{remark}[Gauss-Newton]
    Let $\nu^i=-M_{\thisx}(\nextu-\thisu)$ and $\basex=\thisx$.
    In fact, we can even replace $M_{\thisx}$ by the identity $I$. 
    Then $\nextu$ solves the local linearised 
    problem \eqref{eq:linearised-problem}, that is
    \[
        0 \in H_{\thisx}(\nextu),
    \]
    and Lemma \ref{lemma:descent} together with
    Theorem \ref{thm:conv-idea} show the convergence of the Gauss-Newton method
    to a critical point of \eqref{eq:nonlinear-problem}, provided $\nu^i \to 0$.
    That is, either the iterates diverge, or the Gauss-Newton method converges
    to a solution.
    %\TODO{Do not use $\thisrealoptu$, above but $\thisoptu$. Then
    %    boundedness of $\thisoptu$ enough after square-removal.
    %    No sensitivity analysis needed.}
    %\TODO{Or maybe $\nu^i \in H_{x^*}(\thisu)-H_{\bar x}(\thisu)$}.
    %Then $u^{i+2}$ not introduced in sensitivity.
\end{remark}

\begin{remark}[Varying over-relaxation parameter]
    We have assumed that $\OverRelax=1$. It is however possible to
    accommodate a varying parameter $\OverRelax^i \to 1$ through $\nu^i$.
    In particular, one can easily show the convergence (albeit merely sublinear)
    of the accelerated algorithm of \cite{chambolle2010first} this way.
    In this variant, dependent on the strong convexity of $F$ or $G$, 
    one updates the parameters at each step as $\OverRelax^i = 1/\sqrt{1+2\gamma\tau^i}$,
    $\tau^{i+1}=\OverRelax^i \tau_i$, and, $\sigma^{i+1}=\sigma^i/\tau^i$
    for $\gamma$ the factor of strong convexity of either $F^*$ or $G$.
\end{remark}

\begin{remark}[Interpolated PDHGM for non-linear operators]
    \label{remark:nl-cp-interp}
    Our analysis, forthcoming and to this point, applies to a yet further variant of
    Algorithm \ref{algorithm:nl-cp}. Here we replace $K(\overnextx)$
    in \eqref{eq:nl-cp3} by
    \[
        (1+\OverRelax)K(\nextx)-\OverRelax K(\thisx) = K(\thisx)+(1+\OverRelax)(K(\nextx)-K(\thisx)) \approx K(\overnextx).
    \]
    For this method
    \[
        \nu^i 
        = \bigl(0, \grad K(\thisx)(\overnextx-\thisx) - (1+\OverRelax)(K(\nextx)-K(\thisx))\bigr)
        =(1+\OverRelax)D_{\thisx}(\nextx).
    \]
    %This is also easily seen to satisfy \eqref{eq:ass-di} under \eqref{eq:ass-k}.
    %Both this variant and Algorithm \ref{algorithm:nl-cp} turn out to be numerically
    %more efficient than the linearised Algorithm \ref{algorithm:nl-cp-lin}, which 
    %involves more operations in \eqref{eq:nl-cp3-lin}. 
    %Algorithm \ref{algorithm:nl-cp-lin} is however important for the convergence
    %analysis.
\end{remark}

\begin{remark}[An alternative update]
    It is also interesting to consider using $[\grad K(\overx^i)]^*$ instead of
    $[\grad K(\thisx)]^*$ in \eqref{eq:nl-cp1}. From the point of view of the
    convergence proof, this however introduces major difficulties, as 
    ($\nextx, \nexty$) no longer depends on just ($\thisx, \thisy$),
    but also on $x^{i-1}$ through $\overx^i$. This kind of dependence 
    also makes analysis using the original ordering of the PDHGM updates 
    in \cite{chambolle2010first} difficult, but is avoided by the reordering
    due to \cite{he2012convergence} that we employ in \eqref{eq:cp-lin} and \eqref{eq:nl-cp}.
\end{remark}

\section{Detailed analysis of the non-linear method}
\label{sec:detailed}

We now proceed to verifying the assumptions of Theorem \ref{thm:conv-idea}
for Algorithm \ref{algorithm:nl-cp} and Algorithm \ref{algorithm:nl-cp-lin},
provided our initial iterate is close enough to a solution which satisfies 
certain technical conditions, to be derived along the course of the proof.
We will begin with the formal statement of our running assumptions in 
Section \ref{sec:ass}, after which we prove some auxiliary results 
in Section \ref{sec:aux}.
Our first task in verifying the assumptions of Theorem \ref{thm:conv-idea} 
is to show that the discrepancy term $\nu^i = D_{\thisx}(\nextu) \to 0$.
This we do in Section \ref{sec:discrepancy}. 
Then in Section \ref{sec:local-norms} we begin deriving the estimate
\eqref{eq:nonlinear-descent-general} by analysing the switch to 
the new local norm at the next iterate. In Section \ref{sec:aubin} 
we introduce and study Lipschitz type estimates on $\inv H_\realoptx$.
We then then use these in Section \ref{sec:rm-squares} and 
Section \ref{sec:bridging}, respectively, to remove the squares from 
the estimate \eqref{eq:linear-descent-strong} and to bridge
from one local solution to the next one.
The Lipschitz type estimates themselves we will derive
in Section \ref{sec:lipschitz} to follow. We state and prove 
our main convergence theorem, combining all the above-mentioned 
estimates, in Section \ref{sec:combined-estimate}. This we 
follow by a collection of remarks in Section \ref{sec:details-remarks}.

\subsection{General assumptions}
\label{sec:ass}

We take $G: X \to \extR$ and $F^*: Y \to \extR$ 
to be convex, proper, lower semicontinuous functionals
on finite-dimensional Hilbert spaces $X$ and $Y$,
satisfying $\interior \Dom G, \interior \Dom F^* \ne \emptyset$.
We fix $\OverRelax=1$ and study the sequence of iterates 
generated by solving
\begin{equation}
    %\label{eq:prox-update-discrepancy-i}
    %\tag{P-$\mathrm{D}^i$}
    \notag
    0 \in \Ti(\nextu) \defeq H_{\thisx}(\nextu) + D^i(\nextu) + M_{\thisx}(\nextu-\thisu),
\end{equation}
where we expect the discrepancy functional
\[
    D^i: X \times Y \to \{0\} \times Y,\quad (i=1,2,3,\ldots),
\]
to satisfy for any fixed $i$, $C, \epsilon > 0$, 
the existence of $\rho > 0$ such that
\begin{equation}
    \tag{A-$\mathrm{D}^i$}
    \label{eq:ass-di}
    %\norm{D^i(u)-D^i(\thisu)} \le \epsilon\norm{x-\thisx}
    \norm{D^i(u)} \le \epsilon\norm{\thisx-x},
    \quad (\norm{\thisx-x} \le \rho, \norm{\thisu} \le C).
\end{equation}
For brevity, we denote
\[
    \nu^i \defeq D^i(\nextu).
\]
We then aim to take $\thisoptu$ as a solution of the perturbed linearised 
problem
\[
    0 \in H_{\thisx}(\thisoptu) + \nu^i,
\]
and $\thisrealoptu$ as a solution of the linearised problem
\[
    0 \in H_{\thisx}(\thisrealoptu).
\]
Note that these solutions do not necessarily exist.
We will later prove that they exist near a solution $\realoptu$
for which $\inv H_\realoptx$ satisfies the Aubin property.
As Theorem \ref{thm:conv-idea} makes no particular
requirement on $\thisrealoptu$, we will begin with 
$\thisoptu$ and $\thisrealoptu$ arbitrary elements 
of $X \times Y$, and state the above inclusions
explicitly when we require or have them.

Regarding the operator $K: X \to Y$ and the step length
parameters $\sigma, \tau > 0$, we require that
\begin{equation}
    \label{eq:ass-k}
    \tag{A-K}
    K \in C^2(X; Y)
    \quad
    \text{and}
    \quad
    \sigma\tau\left(\sup_{\norm{x} \le \Rad}\norm{\grad K(x)}^2\right) < 1.
\end{equation}
\begin{subequations}
\label{eq:constants}
Here we fix $\Rad > 0$ such that there exists a solution $\realoptu$ to
\begin{equation}
    \label{eq:rad}
    0 \in H_{\realoptx}(\realoptu) \quad \text{with} \quad \norm{\realoptu} \le \Rad/2.
\end{equation}
We then denote by $L_2$ the Lipschitz factor of $x \mapsto \grad K(x)$ on the
closed ball $\B(0, \Rad) \subset X$, namely
\begin{equation}
    \label{eq:ltwo}
    L_2 := \sup_{\norm{x} \le \Rad} \norm{\grad^2 K(x)}
\end{equation}
in operator norm. By \eqref{eq:ass-k}, the supremum is bounded.
Finally, we define the ``linear'' and ``non-linear'' subspaces
\[
    \Ylin \defeq  \{ z \in Y \mid \text{the map } x \mapsto \iprod{z}{K(x)} \text{ is linear} \},
    \quad
    \text{and}
    \quad
    \Ynl \defeq \Ylin^\perp,
\]
and denote by $\Pnl$ the orthogonal projection into $\Ynl$.

\subsection{Auxiliary results}
\label{sec:aux}

The assumption \eqref{eq:ass-k} guarantees in particular that
the weight operators $M_{x^i}$ are uniformly bounded, as we
state in the next lemma. For $\MMax$ and $\MMin$ as in the lemma, 
we also define the uniform condition number
\begin{equation}
    \label{eq:mcond}
    \MCond \defeq \MMax/\MMin.
\end{equation}
Observe that $\MCond \to 1$ as $\tau,\sigma \to 0$.

\begin{lemma}
    \label{lemma:m}
    Suppose \eqref{eq:ass-k} holds. Then
    there exist $\MMax \ge \MMin > 0$ such that
    \begin{equation}
        \label{eq:m-bounds-lemma}
        \MMin^2 I \le M_{x} \le \MMax^2 I,
        \quad
        (\norm{x} \le \Rad).
    \end{equation}
\end{lemma}
\end{subequations}
\begin{proof}
    This follows immediately from the fact that $\grad K$ is bounded 
    on $\B(0, \Rad)$.
\end{proof}

In case of Algorithm \ref{algorithm:nl-cp-lin}, showing 
%\eqref{eq:ass-d}, 
\eqref{eq:ass-di}
is trivial because $\nu^i=0$. For Algorithm \ref{algorithm:nl-cp}, we show
this in the next lemma.

\begin{lemma}
    \label{lemma:discrepancy-estimate}
    Suppose \eqref{eq:ass-k} holds, and let $D^i = D_{\thisx}$,
    i.e., suppose $\{\thisu\}_{i=1}^\infty$ is generated by
    Algorithm \ref{algorithm:nl-cp}. 
    %Then \eqref{eq:ass-d}, \eqref{eq:ass-di} hold.
    Then \eqref{eq:ass-di} holds.
    The same is the case with $D^i = 0$, i.e., Algorithm \ref{algorithm:nl-cp-lin}.
\end{lemma}
\begin{proof}
    %That \eqref{eq:ass-d} holds is trivial, because $D_{\thisx}(\thisu)=0$.
    We recall the definition of $D_\basex$ from \eqref{eq:dbasex}, and notice
    that
    \[
        D_\basex(x,y)=(0, Q_\basex(x+\OverRelax(x-\basex)))
    \]
    for
    \[
        Q_\basex(x) \defeq K(\basex) + \grad K(\basex)(x-\basex) - K(x).
    \]
    To show \eqref{eq:ass-di}, we observe that
    thanks to $K$ being twice continuously differentiable, 
    for any $C > 0$ there exist $L > 0$ and $\rho_1 > 0$ such that
    \[
        \norm{Q_\basex(x)} \le L \norm{x-\basex}^2,
        \quad (\norm{x-\basex} \le \rho_1,\, \norm{\basex} \le C).
    \]
    In particular, minding that
    \[
        %\overnextx-\thisx=(1+\OverRelax)(\nextx-\thisx),
        [\nextx+\OverRelax(\nextx-\thisx)]-\thisx=(1+\OverRelax)(\nextx-\thisx),
    \]
    setting $\rho_2 \defeq \rho_1/(1+\OverRelax)$, we find
    \[%begin{equation}
        %\label{eq:d-approx}
        \norm{D_{\thisx}(\nextu)}
        \le L(1+\OverRelax)^2 \norm{\nextx-\thisx}^2,
        \quad
        (\norm{\nextx-\thisx} \le \rho_2,\, \norm{\thisx} \le C).
    \]%end{equation}
    Choosing $\rho \in (0, \rho_2)$ small enough, 
    \eqref{eq:ass-di} follows.
\end{proof}

\begin{comment}
\begin{remark}[Interpolated method]
    For the method of Remark \ref{remark:nl-cp-interp}, we have
    \[
        \begin{split}
        \nu^i & = \grad K(\thisx)(\overnextx-\thisx) - (1+\OverRelax)(K(\nextx)-K(\thisx))
            \\
            & = (1+\OverRelax)(\grad K(\thisx)(\nextx-\thisx) + K(\thisx) - K(\nextx))
              = (1+\OverRelax)Q_{\thisx}(\nextx).
        \end{split}
    \]
    This is also easily seen to satisfy 
    \eqref{eq:ass-di} under \eqref{eq:ass-k}.
\end{remark}
\end{comment}

We will occasionally use the auxiliary mapping $\EE$ defined in the following 
lemma. The motivation behind it is that if
$v \in H_{\basextwo}(u)$ then $v + \EE(u; \baseu, \baseutwo) \in H_{\basex}(u)$. 

\begin{lemma}
    \label{lemma:e-term}
    Suppose \eqref{eq:ass-k} holds. Let $\baseu, \baseutwo$ satisfy
    $\norm{\baseu}, \norm{\baseutwo} \le \Rad$. Define
    \begin{equation}
        \label{eq:def-e}        
        \EE(u) \defeq 
            \EE(u; \baseu, \baseutwo)
            \defeq
            \begin{pmatrix}
            ( K_{\basex} - K_{\basextwo})^* y \\
            (K_{\basextwo} - K_{\basex}) x + c_\basextwo - c_\basex
            \end{pmatrix}.
    \end{equation}
    Then $\EE$ is Lipschitz with Lipschitz factor $\lip \EE \le L_2 \norm{\baseu-\baseutwo}$, and
    \begin{equation}
        \label{eq:e-estim0}
        \norm{\EE(u)} \le L_2 \norm{\basex - \basextwo}
            \bigl(
                \norm{\Pnl y} + \norm{x-\basextwo} + \norm{\basex - \basextwo}
            \bigr).
    \end{equation}
    %In particular
    %\begin{equation}
    %    \label{eq:e-estim0-basex}
    %    \norm{\EE(\baseu)} \le L_2 \norm{\basex - \basextwo}
    %        \bigl(
    %            \norm{\Pnl \basey} + 2 \norm{\basex-\basextwo}
    %        \bigr).
    %\end{equation}
\end{lemma}
\begin{proof}
    The fact that $\EE$ is Lipschitz with the claimed factor is easy to see; 
    indeed
    \[
        \EE(u)-\EE(u') = \begin{pmatrix}
            ( K_{\basex} - K_{\basextwo})^* (y-y') \\
            (K_{\basextwo} - K_{\basex})  (x-x')
            \end{pmatrix}.
    \]
    Since \eqref{eq:ass-k} holds and $\norm{\basex}, \norm{\basextwo} \le \Rad$, 
    as we have assumed, we have
    \begin{equation}
        \label{eq:k-estim}
        \norm{K_\basex-K_{\basextwo}} \le L_2 \norm{\basex-\basextwo},
    \end{equation}
    It follows that
    \[
        \norm{\EE(u)-\EE(u')} \le L_2 \norm{\basex-\basextwo} \norm{u-u'}.
    \]
    That is, the claimed Lipschitz estimate holds.
    
    Regarding \eqref{eq:e-estim0}, we may rewrite
    \[
        \EE(u)=
        \begin{pmatrix}
            (K_{\basex} - K_{\basextwo})^* y \\
            (K_{\basextwo} - K_{\basex}) (x - \basextwo) - Q_{\basex}(\basextwo)
        \end{pmatrix}.
    \]
    As in the proof of Lemma \ref{lemma:discrepancy-estimate}, the quantity
    \[
        %\begin{split}
        Q_{\basex}(\basextwo) 
        %&
        = (K_{\basex} - K_{\basextwo}) \basextwo + c_{\basex} - c_{\basextwo} \\
        %& 
        = K(\basex) + \grad K(\basex)(\basextwo-\basex) - K(\basextwo) 
        %\end{split}
    \]
    satisfies for some $L_2'>0$ that
    \[
        \norm{Q_{\basex}(\basextwo)} \le L_2' \norm{\basex - \basextwo}^2.
    \]
    In fact, since $\norm{\basex}, \norm{\basextwo} \le \Rad$, 
    we may choose $L_2'=L_2$ .
    We then observe that
    \[
        (K_{\basex} - K_{\basextwo})^* y = (K_{\basex} - K_{\basextwo})^* \Pnl y.
    \]
    Thus
    \[
        \norm{\EE(u)} \le \norm{K_\basex-K_{\basextwo}} \bigl(
                    \norm{\Pnl y} + \norm{x-\basextwo}\bigr) + L_2 \norm{\basex - \basextwo}^2.
    \]
    Using \eqref{eq:k-estim}, we get \eqref{eq:e-estim0}. 
    %The claim \eqref{eq:e-estim0-basex} follows immediately from \eqref{eq:e-estim0}.
\end{proof}

\subsection{Convergence of the discrepancy term}
\label{sec:discrepancy}

The convergence $\nu^i \to 0$ required by Theorem \ref{thm:conv-idea} 
follows from the following simple result that we will also use later on.

\begin{lemma}
    \label{lemma:basic-bounds}
    Let $\Rad$ and $\MCond$ be as in \eqref{eq:constants}.
    Suppose \eqref{eq:ass-di} and \eqref{eq:ass-k} hold,
    and let $u^1 \in X \times Y$ and $\nextu \in \inv \Ti(0)$, ($i=1,\ldots,k-1$).
    If \eqref{eq:nonlinear-descent-general} holds for $i=1,\ldots,k-1$
    for some $\realoptu^1, \ldots, \realoptu^{k}$ and $\zeta>0$ with
    \begin{equation}
        \label{eq:di-lemma-ass}
        %(u^1, \realoptu^1) \in \CC(\Rad/2, \Rad\zeta\MMin/2; M_{x^1}),
        %(u^1, \realoptu^1) \in \CC(\Rad/2, \Rad\zeta\MMin/2; M_{x^1}),
        \norm{u^1 - \realoptu} \le \Rad/4
        \quad
        \text{and}
        \quad
        \norm{u^1 - \realoptu^1} \le \Rad\zeta/(4\MCond),
    \end{equation}
    then
    \begin{subequations}
    \begin{align}
        \label{eq:bb-ui}
        \norm{u^k} & \le \Rad,
        \\
        \label{eq:bb-ui-uk}
        \norm{u^k-u^1} & \le (\kappa/\zeta) \norm{u^1-\realoptu^1},
        \quad
        \text{and}
        \\
        \label{eq:bb-hat-ui-uk}
        \norm{\realoptu^k-u^k} & \le \kappa \norm{u^1-\realoptu^1}.
    \end{align}
    \end{subequations}
\end{lemma}

\begin{proof}
    Choose $\epsilon > 0$, and let $\rho>0$ be prescribed
    by \eqref{eq:ass-di} for $C=\Rad$ and $\epsilon$.
    Since \eqref{eq:nonlinear-descent-general} holds, we have
    \[
        \norm{\thisu-\thisrealoptu}_{M_{\thisx}}
        \ge \zeta \norm{\nextu-\thisu}_{M_{\thisx}}
            + \norm{\nextu-\nextrealoptu}_{M_{\nextx}},
        \quad (i=1,\ldots,k-1).
    \]
    Taking $j \in \{1,\ldots,k\}$, this gives
    \begin{equation}
        \label{eq:disc-convergence-down}
        \norm{u^1-\realoptu^1}_{M_{x^1}}
        \ge
        \zeta \sum_{i=1}^{j-1} \norm{u^{i+1}-u^i}_{M_{x^i}}
        +
        \norm{u^j-\realoptu^j}_{M_{x^j}}.
    \end{equation}
    By \eqref{eq:di-lemma-ass}, we have $\norm{u^1} \le 3\Rad/4$.
    Making the induction assumption
    \begin{equation}
        \label{eq:dc-ind-ass}
        \sup_{i=1, \ldots, j} \norm{u^i} \le \Rad,
    \end{equation}
    we get using Lemma \ref{lemma:m} and \eqref{eq:disc-convergence-down} that
    \begin{equation}
        \label{eq:disc-convergence-sum}
        \MMax \norm{u^1-\realoptu^1}
        \ge
        \norm{u^1-\realoptu^1}_{M_{x^1}}
        \ge
        \MMin \zeta \sum_{i=1}^{j-1} \norm{u^{i+1}-u^i}
        %+
        %\norm{u^k-\realoptu^k}_{M_{x^k}}.
        \ge
        \MMin \zeta \norm{u^j-u^1}.
    \end{equation}
    This gives
    % shows \eqref{eq:bb-ui-uk} for $k=j$, and gives
    \[
        %\begin{split}
        \norm{u^j}
        \le \norm{u^j-u^1} + \norm{u^1}
        %&
        \le (\MCond/\zeta) \norm{u^1-\realoptu^1} + \norm{u^1}.
        %\\
        %&
        %\le \frac{\MCond}{\zeta} \norm{u^1-\realoptu^1} + \norm{u_1-\realoptu} + \Rad/2.
        %\end{split}
    \]
    Using \eqref{eq:di-lemma-ass}, we thus have $\norm{u^j} \le \Rad$,
    so that the induction assumption \eqref{eq:dc-ind-ass} is satisfied
    for $j$. Taking $j=k$, \eqref{eq:dc-ind-ass} shows \eqref{eq:bb-ui}
    and \eqref{eq:disc-convergence-sum} shows \eqref{eq:bb-ui-uk}.
    Furthermore, using \eqref{eq:disc-convergence-down} and Lemma \ref{lemma:m}, 
    we get
    \[
        \MCond \norm{u^1-\realoptu^1}
        \ge
        \norm{u^k-\realoptu^k}.
    \]
    This shows \eqref{eq:bb-hat-ui-uk}.
\end{proof}

\begin{lemma}
    \label{lemma:disc-convergence}
    Under the assumptions of Lemma \ref{lemma:basic-bounds},
    $\nu^i \to 0$.
    % and $\sup_i \norm{\thisrealoptu} < \infty$.
    %In particular, the conditions \eqref{eq:m-bounds}, \eqref{eq:conv-idea-di} 
    %and \eqref{eq:conv-idea-realoptu-bound} of Theorem \ref{thm:conv-idea} hold.
    In particular, the conditions \eqref{eq:m-bounds} and \eqref{eq:conv-idea-di} 
    of Theorem \ref{thm:conv-idea} hold.
\end{lemma}
\begin{proof}
    Firstly, applying \eqref{eq:bb-ui} and Lemma \ref{lemma:m}, 
    we see that \eqref{eq:m-bounds} holds.
    %An application of \eqref{eq:bb-ui}, \eqref{eq:bb-hat-ui-uk} gives
    %\[
    %    \norm{\realoptu^k} 
    %    \le
    %    \norm{\realoptu^k-u^k} + \norm{u^k}
    %    \le
    %        \Rad\zeta/4 + \Rad < \infty.
    %\]
    %This proves \eqref{eq:conv-idea-realoptu-bound}.
    Secondly, by \eqref{eq:disc-convergence-down} we find that
    $\norm{\nextu-\thisu} \to 0$. 
    Thus eventually $\norm{\nextu-\thisu} < \rho$.  
    Using \eqref{eq:ass-di}, we now see that $\nu^i = D^i(\nextu) \to 0$. 
    This proves \eqref{eq:conv-idea-di}.
\end{proof}

\subsection{Switching local norms: estimates from strong convexity}
\label{sec:local-norms}

We finally begin deriving the descent inequality
\eqref{eq:nonlinear-descent-general} that we so far have assumed.
As a first step, we set $\basex=\thisx$ in \eqref{eq:linear-descent-strong},
and use the extra slack that strong convexity gives to replace
$M_\thisx$ by $M_\nextx$ in the term $\norm{\nextu-\thisoptu}_{M_{\thisx}}^2$.

%\begin{lemma}
%    \label{lemma:neigh}
%    Suppose \eqref{eq:ass-k} holds, and let
%    $\{\thisu\}_{i=1}^\infty$ be generated by \eqref{eq:prox-update-discrepancy-i}.
%    If $(\thisu, \thisoptu) \in \CC(\Rad/2, \epsilon; M_{\thisx})$ for $\epsilon < \Rad\MMin/2$.
%    Then $(\thisu, \nextu) \in \CC(\Rad/2, \epsilon; M_{\thisx})$ and
%    $(\nextu, \thisoptu) \in \CC(\Rad-\epsilon/r, \epsilon; M_{\thisx}) \subset \CC(\Rad, \epsilon; M_{\thisx})$.
%\end{lemma}
%\begin{proof}
%    Follows from the estimate \eqref{eq:linear-descent} of Lemma \ref{lemma:descent}.
%\end{proof}

\begin{lemma}
    \label{lemma:norm-switch}
    Suppose \eqref{eq:ass-k} and \eqref{eq:linear-descent-strong} hold
    for some $\thisoptu$.
    Let $\Rad$, $L_2$, $\MCond$ be as in \eqref{eq:constants},
    and choose $\zeta_1 \in (0, 1)$. 
    % and that $F^*$ is strongly convex on the sub-space $\Ynl$ 
    %with constant $\gamma > 0$. 
    %Let $\nextu \in \inv \Ti(0)$. 
    %Choose $\zeta_1 \in (0, 1)$. 
    If 
    \begin{equation}
        \label{eq:norm-switch-ass}
        %(\thisu, \thisoptu) \in \CC(\Rad/2, \epsilon; M_{\thisx}),
        \norm{\thisu-\realoptu} \le \Rad/4
        \quad
        \text{and}
        \quad
        \norm{\thisu-\thisoptu} \le 
        \min\{
        \sqrt{\gamma(1-\zeta_1)} \MMin/(\sqrt{2}L_2 \MCond),\,
        \Rad/(4\MCond)\},
        %\epsilon,
    \end{equation}
    then
    \begin{equation}
        \label{eq:descent-norm-switch}
        \tag{$\mathrm{\opt{D}^2}$-M}
        \norm{\thisu-\thisoptu}_{M_{\thisx}}^2
        \ge \zeta_1 \norm{\nextu-\thisu}_{M_{\thisx}}^2
            + \norm{\nextu-\thisoptu}_{M_{\nextx}}^2.
    \end{equation}
\end{lemma}
\begin{proof}
    %shows that $M_\thisx$ is positive definite.
    Using \eqref{eq:norm-switch-ass} and $\norm{\realoptu} \le \Rad/2$,
    we have
    \begin{equation}
        \label{eq:norm-switch-thisu-est}
        \norm{\thisu} \le 
        \norm{\thisu-\realoptu} + \norm{\realoptu}
        \le 3\Rad/4.
    \end{equation}
    %The estimate \eqref{eq:linear-descent-strong} holds by Lemma \ref{lemma:descent}.
    The estimate %this 
    \eqref{eq:linear-descent-strong}  implies
    \[
        \norm{\nextu-\thisu} \le \MCond \norm{\thisoptu-\thisu}.
    \]
    Using \eqref{eq:norm-switch-ass} and \eqref{eq:norm-switch-thisu-est}, we thus get
    \[
        %\begin{split}
        \norm{\nextu}
        %&
        \le
        \norm{\nextu-\thisu}+\norm{\thisu}
        %\\
        %&
        \le
        \MCond \norm{\thisoptu-\thisu}+\norm{\thisu}
        \le \Rad.
        %\end{split}
    \]
    As both $\norm{\thisu}, \norm{\nextu} \le \Rad$, by \eqref{eq:ass-k} we have again
    \[
        \norm{K_{\nextx}-K_{\thisx}} \le L_2 \norm{\nextx-\thisx}.
    \]
    This allows us to deduce
    \begin{equation}
        \label{eq:norm-switch-1}
        \begin{split}
        \norm{\nextu-\thisoptu}_{M_{\thisx}}^2 - \norm{\nextu-\thisoptu}_{M_{\nextx}}^2
        &
        =
        -2 \iprod{\nexty-\thisopty}{(K_{\nextx}-K_{\thisx})(\nextx-\thisoptx)}
        \\
        &
        =
        -2 \iprod{\nexty-\thisopty}{\Pnl(K_{\nextx}-K_{\thisx})(\nextx-\thisoptx)}
        \\
        & \ge
        -2\norm{\Pnl(\nexty-\thisopty)}\norm{K_{\nextx}-K_{\thisx}}\norm{\nextx-\thisoptx}
        \\
        & \ge
        -2L_2\norm{\Pnl(\nexty-\thisopty)}\norm{\nextx-\thisx}\norm{\nextx-\thisoptx}.
        \end{split}
    \end{equation}
    Using Young's inequality we therefore have 
    %for any $\mu > 0$ that
    %\[
    %    \norm{\nextu-\thisoptu}_{M_{\thisx}}^2 - \norm{\nextu-\thisoptu}_{M_{\nextx}}^2
    %    \ge
    %    -\frac{L_2}{2\mu}\norm{\Pnl(\nexty-\thisopty)}^2
    %    -\frac{L_2\mu}{2}\norm{\nextx-\thisx}^2\norm{\nextx-\thisoptx}^2
    %\]
    %With $\mu=L_2/\gamma$, it follows that
    \begin{equation}
        %\label{eq:norm-switch-intermed}
        \notag
        %\begin{split}
        \norm{\nextu-\thisoptu}_{M_{\thisx}}^2 - \norm{\nextu-\thisoptu}_{M_{\nextx}}^2
        +\gamma\norm{\Pnl(\nexty-\thisopty)}^2
        %&
        %\ge
        %-\frac{L_2^2}{2\gamma}\norm{\nextx-\thisx}^2\norm{\nextx-\thisoptx}^2
        %\\
        %&
        \ge
        -\frac{2L_2^2}{\gamma}\norm{\nextu-\thisu}^2\norm{\nextu-\thisoptu}^2.
        %\end{split}
    \end{equation}
    By \eqref{eq:linear-descent-strong} it follows
    \begin{equation}
        \label{eq:norm-switch-intermed}
        \norm{\thisu-\thisoptu}_{M_{\thisx}}^2 
        \ge
        \norm{\nextu-\thisu}_{M_{\thisx}}^2 
        +
        \norm{\nextu-\thisoptu}_{M_{\nextx}}^2
        -\frac{2L_2^2}{\gamma}\norm{\nextu-\thisu}^2\norm{\nextu-\thisoptu}^2.
    \end{equation}
    An application of Lemma \ref{lemma:m} and \eqref{eq:linear-descent-strong}
    shows that
    \begin{align}
        \notag
        \norm{\nextu-\thisoptu}^2 & 
        \le \MMin^{-2} \norm{\nextu-\thisoptu}_{M_{\thisx}}^2
        \le \MCond^2 \norm{\thisu-\thisoptu}^2,
        \quad\text{and}
        \\
        \notag
        \norm{\nextu-\thisu}^2 & 
        \le \MMin^{-2} \norm{\nextu-\thisu}_{M_{\thisx}}^2.
    \end{align}
    Using \eqref{eq:norm-switch-ass}, therefore
    \[
        \frac{2L_2^2}{\gamma}\norm{\nextu-\thisu}^2\norm{\nextu-\thisoptu}^2
        \le \frac{2L_2^2\MCond^2}{\gamma\MMin^2} \norm{\nextu-\thisu}_{M_{\thisx}}^2 \norm{\thisu-\thisoptu}^2
        \le (1-\zeta_1) \norm{\nextu-\thisu}_{M_{\thisx}}^2.
    \]
    Applying this to \eqref{eq:norm-switch-intermed} 
    yields \eqref{eq:descent-norm-switch}.
\end{proof}

\subsection{Aubin property of the inverse}
\label{sec:aubin}

In \cite{rockafellar1976monotone}, the Lipschitz continuity of the equivalent of 
the map $H_{\realoptx}^{-1}$ was used to prove strong convergence properties
of basic proximal point methods for maximal monotone operators. 
In order to remove the squares from \eqref{eq:descent-norm-switch}, and to bridge
with $\thisoptu$ between the local solutions $\thisrealoptu$ and $\nextrealoptu$,
we follow the same rough ideas. 
We however replace the basic form of Lipschitz continuity by a weaker version that is
localised in the graph of $H_{\realoptx}$. Namely, with $0 \in H_{\realoptx}(\realoptu)$,
we assume that the map $H_{\realoptx}^{-1}$ has the Aubin property at $0$ for $\realoptu$ 
\cite{rockafellar-wets-va,aubin1990sva}. This is also called \term{metric regularity}.

Generally, the inverse $\inv S$ of a set-valued map $S: X \setto Y$ having the 
\emph{Aubin property} at $\realopt w$ for $\realoptu$ means that $\graph S$ is 
locally closed, and there exist $\rho, \delta, \lipnum > 0$ such that
%\[
%    \inv S(w') \isect \B(\realoptu, \rho)
%    \subset
%    \inv S(w) + \B(w'-w, \lipnum),
%    \quad
%    (\norm{w-\realopt{w}}, \norm{w'-\realopt{w}} \le \delta).
%\]
%Written in terms of $S$, the property says that
\begin{equation}
    \label{eq:inverse-aubin}
    \inf_{v\,:\, w \in S(v)} \norm{u-v}
        \le \lipnum \norm{w-S(u)},
    \quad
    (
     \norm{u-\realoptu} \le \delta,
     \,
     \norm{w-\realopt{w}} \le \rho
     ).
\end{equation}
%In particular, with $u=\realoptu$, this says
%\begin{equation}
%    \label{eq:inverse-aubin-thisoptu}
%    \inf_{v\,:\, w \in S(v)} \norm{\realoptu-v}
%        \le \lipnum \norm{w-\realopt{w}},
%    \quad
%    (
%     \norm{w-\realopt{w}} \le \delta
%     ).
%\end{equation}
We denote the infimum over valid constants $\lipnum$ by $\lip{\inv S}(\realopt w|\realoptu)$,
or $\lip{\inv S}$ for short, when there is no ambiguity about the point $(\realopt w, \realoptu)$.
For bijective single-valued $S$ the Aubin property reduces to Lipschitz-continuity of the inverse $\inv S$. 
For single-valued Lipschitz $S$, we also use the notation $\lip{S}$ for the (local) Lipschitz factor.

We can translate the Aubin property of $\inv H_{\realoptx}$ to $\inv H_{\thisx}$,
based on the following general lemma. This is needed in order to perform
estimation at an iterate $\thisu$, only assuming the Aubin property of $H_{\realoptx}$
at a known solution $\realoptu$. 

\begin{lemma}
    \label{lemma:translated-aubin-general}
    Suppose $\inv S$ has the Aubin property at $0$ for $\realoptu$.
    Let $T(u)=S(u)+\Delta(u)$ for a single-valued Lipschitz map $\Delta: X \to Y$
    with $\lip{\inv S}\lip{\Delta} < 1$.
    %There exists a constant $\bar \lipnum$ such that if $\lipnum_\Delta < \bar \lipnum$,
    Then $\inv T$ has the Aubin property at $\Delta(\realoptu)$ for $\realoptu$,
    %and $\lip{\inv T} = \lip{\inv S}{$.
    and
    \begin{equation}
        \label{eq:lip-inv-t}
        \lip{\inv T} \le \frac{\lip{\inv S}}{1-\lip{\inv S}\lip{\Delta}}.
    \end{equation}
\end{lemma}

The proof is a minor modification of the proof of the 
following result.

\begin{lemma}[\cite{dontchev1994inverse}]
    \label{lemma:translated-aubin-strictstat}
    Suppose $\inv S$ has locally closed values and the Aubin property
    at $0$ for $\realoptu$. 
    Let $\Delta: X \to Y$ be a single-valued Lipschitz map,
    supposing that $\Delta$ is \emph{strictly stationary at $\realoptu$}.
    This means that for every $\epsilon > 0$, there exists $\delta >0$ 
    such that
    \[
        \norm{\Delta(u)-\Delta(u')} \le \epsilon\norm{u-u'}, 
        \quad (\norm{\realoptu-u}, \norm{\realoptu-u'} < \delta).
    \]
    Let $T(u)=S(u)+\Delta(u)$.
    Then $\inv T$ has the Aubin property at $\Delta(\realoptu)$ for $\realoptu$,
    and $\lip{\inv T} = \lip{\inv S}$.
\end{lemma}

\begin{proof}
    This result is the main theorem of \cite{dontchev1994inverse}
    combined with the remark following the theorem. 
\end{proof}

\begin{proof}[Proof of Lemma \ref{lemma:translated-aubin-general}]
    %\TODO{full proof?}
    We show how to modify the proof of
    Lemma \ref{lemma:translated-aubin-strictstat} 
    in \cite{dontchev1994inverse} into a proof of
    Lemma \ref{lemma:translated-aubin-general}.
    There are two differences in assumptions between the lemmas.
    The first is the apparently different closedness condition on $S$.
    But, obviously, $\inv S$ has locally closed values if $\graph S$ 
    is locally closed, which our definition of the Aubin property
    includes. Therefore that part of the assumptions of
    Lemma \ref{lemma:translated-aubin-strictstat} is satisfied.

    The second difference is the strict stationarity condition on $\Delta$. 
    In Lemma \ref{lemma:translated-aubin-general}, this is replaced 
    by the weaker condition $\lipnum_{\inv S} \lipnum_\Delta  < 1$. 
    We however observe that, indeed, the strict stationarity condition
    is only used in the proof of Lemma \ref{lemma:translated-aubin-strictstat}
    to show that $\Delta$ is Lipschitz 
    with a given constant $\epsilon>0$ in a neighbourhood 
    of $\realoptu$ \cite[(3)]{dontchev1994inverse}, and then
    \[
        \lip{\inv T} < \frac{\lip{\inv S}}{1-\epsilon\lip{\inv S}}.
    \]
    With our assumptions, we may only take $\epsilon > \lip{\Delta}$.
    This gives us the weaker result \eqref{eq:lip-inv-t} instead of
    $\lip{\inv T}=\lip{\inv S}$.
    
    We conclude that the proof of 
    Lemma \ref{lemma:translated-aubin-strictstat} 
    in \cite{dontchev1994inverse} is easily modified
    into a proof of Lemma \ref{lemma:translated-aubin-general}.
\end{proof}

\begin{lemma}
    \label{lemma:translated-aubin}
    Suppose $H_{\realoptx}^{-1}$ has the Aubin property at $0$ for $\realoptu$,
    and that \eqref{eq:ass-k} holds. 
    Given $\lipoptF > \lipopt$, there exists $\delta \in (0, \Rad/2)$ 
    and $\rho > 0$ such that if
    \[
        %(\thisu, \realoptu) \in \CC(\Rad/2, \delta),
        \norm{\thisu-\realoptu} \le \delta,
    \]
    then
    \begin{equation}
        \label{eq:inverse-aubin-translated-noshift}
        \inf_{v\,:\, w \in H_{\thisx}(v)} \norm{u-v}
            \le \lipoptF \norm{w-H_{\thisx}(u)},
        \quad
        (
         \norm{u-\realoptu} \le \delta,
         \,
         \norm{w} \le \rho
         ).
    \end{equation}
\end{lemma}

\begin{remark}
The property \eqref{eq:inverse-aubin-translated-noshift} is
formally the Aubin property of $H_{\thisx}^{-1}$ at $0$ for $\realoptu$,
but cannot strictly be called that, because generally
$0 \not\in H_{\thisx}(\realoptu)$.
\end{remark}

%\begin{remark}
%Observe that with $S=H_\basex$, the infimum in \eqref{eq:inverse-aubin} is reached
%since $G$ and $F^*$ are convex, proper, and lower semicontinuous, and $K_\basex$ continuous.
%Likewise the infimum is reached in \eqref{eq:inverse-aubin-translated-noshift}.
%\end{remark}

\begin{proof}
    We use Lemma \ref{lemma:translated-aubin-general} on $S=H_{\realoptx}$, and
    $\Delta=\EE(\cdot; \realoptu, \thisu)$, where $\EE$  is defined in \eqref{eq:def-e}.
    We pick $\bar\lipnum > 0$ such that
    \[
        \lipoptF \ge \frac{\lipopt}{1-\lipopt\bar\lipnum} > 0.
    \]
    Clearly the same condition then holds with $\bar\lipnum$ replaced
    by any $\lipnum \in (0, \bar\lipnum)$. We exploit this.
    By Lemma \ref{lemma:e-term}, $\Delta$ is Lipschitz with factor
    $\lip{\Delta} \le L_2 \delta$, 
    so that $\lip{\Delta} < \bar \lipnum$ provided $\delta < \bar \lipnum/L_2$.
    Thus by Lemma \ref{lemma:translated-aubin-general}, we have
    \begin{equation}
        \label{eq:inverse-aubin-translated}
        \inf_{v\,:\, w \in H_{\thisx}(v)} \norm{u-v}
            \le \lipoptF \norm{w-H_{\thisx}(u)},
        \quad
        (
         \norm{u-\realoptu} \le \delta',
         \,
         \norm{w-\Delta(\realoptu)} \le \rho'
         )
    \end{equation}
    for some $\rho', \delta' > 0$.
    Referring to Lemma \ref{lemma:e-term}, we have
    \begin{equation}
        \notag
        \norm{\Delta(\realoptu)} \le L_2 \norm{\realoptx - \thisx}
            \bigl(
                \norm{\Pnl \realopty} + 2 \norm{\realoptx-\thisx}
            \bigr).
    \end{equation}
    For $\delta > 0$ small enough, we can thus force
    $\norm{\Delta(\realoptu)} \le \rho'/2$. Thus $\norm{w}\le \rho$ guarantees
    $\norm{w-\Delta(\realoptu)} \le \rho'$ for $\rho<\rho'/2$.
    This proves \eqref{eq:inverse-aubin-translated-noshift}.
\end{proof}

The next lemma bounds step lengths near a solution.

\begin{lemma}
    \label{lemma:steplength}
    Suppose \eqref{eq:ass-di} and \eqref{eq:ass-k} hold, and that
    $H_{\realoptx}^{-1}$ has the Aubin  property at $0$ for $\realoptu$.
    Given  $\epsilon > 0$, there exists $\delta > 0$ such that the
    following holds. If
    \begin{subequations}
    \label{eq:steplength-ass}
    \begin{equation}
        \label{eq:steplength-ass-1}
        %(\thisu, \realoptu) \in \CC(\Rad/2, \delta),
        \norm{\thisu-\realoptu} \le \delta,
    \end{equation}
    then there exist $\thisrealoptu \in \inv H_\thisx(0)$ 
    and $\thisoptu \in \inv H_\thisx(-\nu^i)$. If also
    \begin{equation}
        \label{eq:steplength-ass-2}
        \norm{\thisrealoptu-\realoptu} \le \delta,
        %\quad
        %\text{and}
        %\\
        %
        %\tau \in (0, \tau_0), \text{ and }
        %\sigma \in (0, \sigma_0),
        %
        %\label{eq:steplength-di-ass}
        %\norm{D^i(u^i)} & \le \delta,
    \end{equation}
    \end{subequations}
    then with the specific choice
    \begin{equation}
        \label{eq:tilde-u-i-argmin}
        \thisoptu = \argmin\{\norm{\thisrealoptu-v} \mid v \in \inv H_{\thisx}(-\nu^i)\},
    \end{equation}
    we have the estimates
    \begin{subequations}
        \label{eq:nestim}
    \begin{align}
        \label{eq:nestim-bound}
        %(\thisu, \nextu) & \in \CC(\Rad/2, \delta'),
        \norm{\thisu} & \le 3\Rad/4,
        \\
        \label{eq:nestim-next}
        %(\thisu, \nextu) & \in \CC(\Rad/2, \delta'),
        \norm{\thisu-\nextu} &\le \epsilon,
        \\
        \label{eq:nestim-tilde-hat-this}
        %(\thisrealoptu, \thisoptu) & \in \CC(\Rad, 2\delta').
        \norm{\thisrealoptu-\thisoptu} &\le \epsilon,
        \quad
        \\
        \label{eq:nestim-tilde}
        %(\thisu, \thisoptu) & \in \CC(\Rad/2, \delta'),
        \norm{\thisu-\thisoptu} &\le \epsilon,
        \\
        \label{eq:nestim-tilde-hat}
        %(\realoptu, \thisoptu) & \in \CC(\Rad/2, \delta'),
        \norm{\realoptu-\thisoptu} &\le \epsilon,
        \quad
        \text{and}
        \\
        \label{eq:nestim-nu}
        \norm{\nu^i} & \le \epsilon.
    \end{align}
    \end{subequations}
\end{lemma}

%\begin{remark}
%    Observe that \eqref{eq:steplength-di-ass}
%    is automatically satisfied by Algorithm \ref{algorithm:nl-cp}
%    and Algorithm \ref{algorithm:nl-cp-lin}.
%    The assumption is only needed to provide a convergence
%    theory when $\nextoptu$ is solved inexactly.
%\end{remark}

\begin{proof}
    Observe that in each of the estimates \eqref{eq:nestim},
    we may take $\epsilon>0$ smaller than prescribed.
    Along the course of the proof, we will accordingly assume
    $\epsilon$ as small as required.
    
    We begin by proving \eqref{eq:nestim-bound},
    \eqref{eq:nestim-next} and \eqref{eq:nestim-nu},
    using \eqref{eq:steplength-ass-1} without \eqref{eq:steplength-ass-2}.
    Indeed, observe that since $\norm{\realoptu} \le \Rad/2$, \eqref{eq:nestim-bound} 
    trivially holds by choosing $\delta \in (0, \Rad/4)$.
    To bound $\norm{\thisu-\nextu}$, we return to the
    algorithmic approach for computing $\nextu = (\nextx, \nexty) \in \inv \Ti(0)$.
    Indeed, since $D^i(\nextu)=(0, v^i) \in X \times Y$ for some $v^i=v^i(\nextx)$, we have
    \begin{subequations}
    \begin{align}
        \label{eq:nl-cp1-gen}
        \nextx & \defeq (I+\tau \subdiff G)^{-1}(\thisx - \tau [\grad K(\thisx)]^* y^{i}),
        \\
        \overnextx & \defeq \nextx + \OverRelax(\nextx-\thisx),
        \\
        \label{eq:nl-cp3-gen}
        \nexty & \defeq (I+\sigma \subdiff F^*)^{-1}(\thisy + \sigma[K(\thisx)+K_{\thisx}(\overnextx-\thisx)+v^i]).
    \end{align}
    \end{subequations}
    From \eqref{eq:nl-cp1-gen} we see that $\nextx$ solves for $x$ the problem
    \begin{equation}
        \label{eq:steplength-xupdate}
        \min_x \left\{
                \norm{x-\thisx+\tau K_{\thisx}^* \thisy}^2 + \tau G(x)
            \right\}.
    \end{equation}
    Therefore
    \[
        \norm{\nextx-\thisx+\tau K_{\thisx}^* \thisy}^2 + \tau G(\nextx)
        \le
        \norm{\realoptx-\thisx+\tau K_{\thisx}^* \thisy}^2 + \tau G(\realoptx),
    \]
    which leads to
    \[
        \norm{\nextx-\thisx}^2
        \le
        \norm{\realoptx-\thisx}^2
        + \tau \bigl(G(\realoptx)-G(\nextx)+
        \iprod{(\realoptx-\thisx)+(\thisx-\nextx)}{K_{\thisx}^* \thisy}\bigr).
    \]
    We have $-K_\realoptx^* \realopty \in \subdiff G(\realoptx)$.
    Therefore
    \[
        G(\realoptx)-G(\nextx) \le \iprod{(\realoptx-\thisx)+(\thisx-\nextx)}{-K_{\realoptx}^* \realopty},
    \]
    so that
    \[
        %\begin{split}
        \norm{\nextx-\thisx}^2
        %&
        %\le
        %\norm{\realoptx-\thisx}^2
        %+ \tau\iprod{\realoptx-\nextx}{K_{\thisx}^* \thisy-K_\realoptx^* \realopty}
        %\\
        %&
        =
        \norm{\realoptx-\thisx}^2
        + \tau\iprod{\realoptx-\thisx}{K_{\thisx}^* \thisy-K_\realoptx^* \realopty}
        + \tau\iprod{\thisx-\nextx}{K_{\thisx}^* \thisy-K_\realoptx^* \realopty}.
        %\end{split}
    \]
    An application of Young's inequality gives
    \begin{equation}
        \label{eq:steplength-estim-x1}
        \frac{1}{2} \norm{\nextx-\thisx}^2
        \le
        \frac{3}{2} \norm{\realoptx-\thisx}^2
        +
        \tau^2\norm{K_{\thisx}^* \thisy-K_\realoptx^* \realopty}^2.
    \end{equation}
    The right hand side of \eqref{eq:steplength-estim-x1}
    can be made arbitrarily close to zero by application
    of \eqref{eq:ass-k} and \eqref{eq:steplength-ass-1}.
    That is, for any $\epsilon'>0$, there exists $\delta'>0$ such that
    if \eqref{eq:steplength-ass} holds for some $\delta \in (0, \delta')$, then
    \begin{equation}
        \label{eq:steplength-xstep}
        \norm{\nextx-\thisx} \le \epsilon'.
    \end{equation}
    
    We now have to bound $\norm{\nexty-\thisy}$ through \eqref{eq:nl-cp3-gen}.
    Similarly to \eqref{eq:steplength-xupdate}, $\nexty$ solves for $y$ the problem
    \begin{equation}
        \label{eq:steplength-yupdate}
        \min_y \left\{
                \norm{y-\thisy+\sigma[K(\thisx)+K_{\thisx}(\overnextx-\thisx)+v^i]}^2 + \sigma F(x)
            \right\}.
    \end{equation}
    Proceeding as above, using the fact that
    $K(\realoptx) \in \subdiff F^*(\realopty)$
    we get
    \[
        \frac{1}{2}\norm{\nexty-\thisy}^2
        \le
        \frac{3}{2}\norm{\realopty-\thisy}^2
        +
        \sigma^2 \norm{K(\thisx)+K_{\thisx}(\overnextx-\thisx)+v^i-K(\realoptx)}^2.
    \]
    We approximate
    \begin{equation}
        \label{eq:steplength-k-estim}
        \norm{K(\thisx)+K_{\thisx}(\overnextx-\thisx)+v^i-K(\realoptx)}
        \le
        \norm{K(\thisx)-K(\realoptx)}
        +\OverRelax\norm{K_{\thisx}(\nextx-\thisx)}
        +\norm{v^i}.
    \end{equation}
    By taking $\delta, \epsilon>0$ small enough, which we may do,
    we can make the term $\norm{K(\thisx)-K(\realoptx)}$ arbitrarily small by application of
    \eqref{eq:ass-k} and \eqref{eq:steplength-ass-1}. 
    Likewise, we can make the term $\norm{K_{\thisx}(\nextx-\thisx)}$ approach zero 
    by application \eqref{eq:ass-k}, \eqref{eq:steplength-ass-1} and \eqref{eq:steplength-xstep}. 
    Finally, the term $\norm{v^i}=\norm{\nu^i}$ we can be make small 
    by additionally  using \eqref{eq:ass-di}. Indeed, choosing $\epsilon'' > 0$, 
    and employing \eqref{eq:ass-di}, we have
    \begin{equation}
        \label{eq:steplength-nui}
        \norm{\nu^i} \le \epsilon'' \norm{\thisx-\nextx},
    \end{equation}
    provided $\norm{\thisx-\nextx}$ is small enough. 
    This can be guaranteed by \eqref{eq:steplength-xstep} above.
    Thus, taking $\epsilon>0$ small enough, we can
    make $\norm{\nu^i}$ arbitrarily small. This proves \eqref{eq:nestim-nu},
    and shows that \eqref{eq:steplength-k-estim} can be made arbitrarily small.
    In summary, there exists $\delta'' \in (0, \delta')$ such that if \eqref{eq:steplength-ass-1} 
    holds for $\delta \in (0, \delta'')$, then
    \[
        \norm{\nexty-\thisy}^2 + \norm{\nextx-\thisx}^2 \le \epsilon^2.
    \]
    This proves \eqref{eq:nestim-next}. 

    To prove the remaining estimates in \eqref{eq:nestim},
    we require \eqref{eq:steplength-ass-2} as well
    as the existence of $\thisoptu$ and $\thisrealoptu$.
    In fact, since $\interior \Dom G \ne \emptyset$ and
    $\interior \Dom F^* \ne \emptyset$, we may find a
    point $u' \in (\interior \Dom G) \times (\interior \Dom F^*)$
    arbitrarily close to $\realoptu$.
    Then the set $H_{\thisx}(u')$ is bounded.
    If $\delta > 0$ is small enough, we thus find from
    \eqref{eq:inverse-aubin-translated-noshift} that
    \[
        \inf_{v\,:\, w \in H_{\thisx}(v)} \norm{u'-v} < \infty,
        \quad
        (\norm{w} \le \rho).
    \]
    Minding that we have shown \eqref{eq:nestim-nu}, choosing
    $\epsilon>0$ small enough and setting $w=0$ resp.~$w=-\nu^i$ shows
    the existence of $\thisrealoptu \in \inv H_\thisx(0)$ resp.~$v \in \inv H_\thisx(-\nu^i)$.
    In fact, we have the existence of a minimiser $\thisoptu$ to \eqref{eq:tilde-u-i-argmin}. 
    This follows simply from $v \mapsto \norm{\thisrealoptu-v}$ being level bounded,
    $\subdiff G$ and $\subdiff F^*$ outer semicontinuous
    set-valued mappings, and $K_\basex$ a continuous linear operator.
    
    %With $\thisrealoptu \in \inv H_{\thisx}(0)$ fixed, we pick
    %\[
    %    \thisoptu \in \argmin\{\norm{\thisrealoptu-v} \mid v \in \inv H_{\thisx}(-\nu^i)\}
    %\]
    We may now move on to the proof of \eqref{eq:nestim-tilde-hat}. 
    By Lemma \ref{lemma:translated-aubin}, for any $\lipoptF > \lipopt$,
    there exist $\bar\rho, \bar\delta>0$ such that
    \begin{equation}
        \notag
        \inf_{v\,:\, w \in H_{\thisx}(v)} \norm{u-v}
            \le \lipoptF \norm{w-H_{\thisx}(u)},
        \quad
        (
         \norm{u-\realoptu} \le \bar\delta,
         \,
         \norm{w} \le \bar\rho
         ).
    \end{equation}
    Minding the choice of $\thisoptu$ in \eqref{eq:tilde-u-i-argmin},
    with $w=-\nu^i$, $v=\thisoptu$ and $u=\thisrealoptu$, we thus have
    \begin{equation}
        %\label{eq:aubin-thisoptu-thisrealoptu}
        \notag
        \norm{\thisoptu-\thisrealoptu}
        \le
        \lipoptF \norm{\nu^i},
    \end{equation}
    provided that $\norm{\thisrealoptu-\realoptu} \le \bar\delta$ and
    $\norm{\nu^i} \le \bar\rho$. 
    The former follows from \eqref{eq:steplength-ass-2} and taking
    $\delta$ small enough. % \le \min\{\delta'',\bar\delta\}$.
    The latter follows from \eqref{eq:steplength-xstep}, \eqref{eq:steplength-nui}
    and choosing $\epsilon' > 0$  small enough.
    In fact, taking $\epsilon' \le \min\{\epsilon/(\lipoptF\epsilon''), \bar\rho/\epsilon''\}$, 
    we also show that $\norm{\thisoptu-\thisrealoptu} \le \epsilon$.
    This completes the proof of \eqref{eq:nestim-tilde-hat-this}. 
    
    Finally, to show \eqref{eq:nestim-tilde} and \eqref{eq:nestim-tilde-hat}
    we simply bound
    \begin{align}
        \notag
        \norm{\thisu-\thisoptu}
        &
        \le
        \norm{\realoptu-\thisu}
        +
        \norm{\realoptu-\thisoptu},
        \quad
        \text{and}
        \\
        \notag
        \norm{\realoptu-\thisoptu}
        &
        \le
        \norm{\realoptu-\thisrealoptu}
        +
        \norm{\thisrealoptu-\thisoptu}.
    \end{align}
    Then we use \eqref{eq:steplength-ass} and
    \eqref{eq:nestim-next}--\eqref{eq:nestim-tilde-hat-this},
    assuming that these estimates hold to the higher accuracy $\epsilon/2$
    instead of $\epsilon$. By the arguments above, this can be done by making
    $\delta>0$ small enough.
\end{proof}

\subsection{Removing squares}
\label{sec:rm-squares}

With the Aubin property assumed, we are now able to remove the squares from
\eqref{eq:descent-norm-switch}. Later in Section \ref{sec:lipschitz} we will
prove the Aubin property for important classes of $G$, $F^*$ and $K$.

\begin{lemma}
    \label{lemma:rmsquares}
    %Assume that \eqref{eq:ass-d}, \eqref{eq:ass-k}, and \eqref{eq:prox-update-discrepancy-i} hold
    %along with \eqref{eq:descent-norm-switch}.
    Suppose \eqref{eq:ass-k} holds, and that $H_{\realoptx}^{-1}$ 
    has the Aubin property at $0$ for $\realoptu$. 
    Pick $\lipoptF > \lipopt$. 
    Then there exists $\delta>0$ such that if 
    \eqref{eq:steplength-ass} holds for $\delta$, and
    \eqref{eq:descent-norm-switch} holds for 
    $\thisoptu \in \inv H_\thisx(-\nu^i)$ and some $\zeta_1 > 0$, then
    \begin{equation}
        \label{eq:unsquared-descent1}
        \tag{$\mathrm{\opt{D}}$-M}
        \norm{\thisoptu-\thisu}_{M_{\thisx}}
        \ge
        \zeta_2
        \norm{\nextu-\thisu}_{M_{\thisx}}
        +
        \norm{\thisoptu-\nextu}_{M_{\nextx}}
    \end{equation}
    for $\zeta_2 \defeq 1-1/\sqrt{1+\zeta_1/(\lipoptF\MMax^2)^2}$.
\end{lemma}
\begin{proof}
    Choosing $\delta>0$ small enough and applying Lemma \ref{lemma:translated-aubin}, 
    we have for some $\rho', \delta'>0$ that
    \begin{equation}
        \label{eq:translated-aubin-rmsquares}
        \inf_{v\,:\, w \in H_{\thisx}(v)} \norm{u-v}
            \le \lipoptF \norm{w-H_{\thisx}(u)},
        \quad
        (
         \norm{u-\realoptu} \le \delta',
         \,
         \norm{w} \le \rho'
         ).
    \end{equation}
    Lemma \ref{lemma:steplength} provides a further upper bound on $\delta$ 
    such that if \eqref{eq:steplength-ass} holds with such a $\delta$, then
    $\thisoptu \in H_{\thisx}^{-1}(-\nu^i)$ exists, 
    \begin{subequations}
    \begin{align}
        \label{eq:nestim-bound-rmsquares}
        \norm{\thisu} & \le 3\Rad/4,
        \\
        \label{eq:nestim-next-rmsquares}
        \norm{\thisu-\nextu} & \le \rho'/(2\MMax),
        \\
        \label{eq:nestim-nu-rmsquares}
        \norm{\nu^i} & \le \rho'/2,
        \quad\text{and}
        \\
        \label{eq:nestim-thisopt-rmsquares}
        \norm{\thisoptu-\realoptu} & \le \delta'.
    \end{align}
    \end{subequations}
    Moreover,
    $\nextu \in H_{\thisx}^{-1}(w)$ for $w=M_{\thisx}(\thisu - \nextu)-\nu^i$
    is unique, as can be seen from the strictly convex
    problems \eqref{eq:steplength-xupdate}, \eqref{eq:steplength-yupdate} 
    for calculating $\nextu$. Thus $H_{\thisx}^{-1}(w)$ is single-valued. 
    By \eqref{eq:nestim-bound-rmsquares}, \eqref{eq:ass-k} and Lemma \ref{lemma:m}, 
    we have $M_{\thisx} \le \MMax^2 I$.
    Therefore, using \eqref{eq:nestim-next-rmsquares} and \eqref{eq:nestim-nu-rmsquares}, we have
    \[
        \norm{w}
        \le
        \norm{M_{\thisx}(\thisu - \nextu)}
        +\norm{\nu^i}
        \le
        \rho'.
    \]
    Minding \eqref{eq:nestim-thisopt-rmsquares}, we may thus apply
    \eqref{eq:translated-aubin-rmsquares} to get
    \begin{equation}
        \label{eq:update-lip}
        \begin{split}
        \norm{\thisoptu-\nextu}_{M_{\nextx}}
        &
        \le
        \MMax \norm{\thisoptu-\nextu}
        \\
        &
        \le
        \lipoptF \MMax
        \norm{(M_{\thisx}(\thisu - \nextu)-\nu^i)+\nu^i}
        \\
        &
        \le
        \lipoptF \MMax^{2}
        \norm{\thisu - \nextu}_{M_{\thisx}}.
        \end{split}
    \end{equation}
    Squaring this and applying it to \eqref{eq:descent-norm-switch}, we get
    for $\lambda \defeq \sqrt{1+\zeta_1/(\lipoptF\MMax^2)^2}$ that
    %\begin{equation}
    %    \label{eq:descent---lip}
    %    (1+\zeta_1/(CR^3)) \norm{\optu - \nextu}_{M_{\nextx}}^2
    %    \le
    %    \norm{\thisoptu-\thisu}_{M_{\thisx}}^2,
    %\end{equation}
    \begin{equation}
        \label{eq:lipschitz-estim-opt}
        \lambda \norm{\thisoptu-\nextu}_{M_{\nextx}} \le \norm{\thisoptu-\thisu}_{M_{\thisx}}.
    \end{equation}
    By application of \eqref{eq:linear-descent}, minding that $\lambda > 1$, 
    we also have
    \begin{equation}
        \label{eq:descent2--lip}
        \begin{split}
        \norm{\thisoptu-\thisu}_{M_{\thisx}}
        &
        =
        \lambda\norm{\thisoptu-\thisu}_{M_{\thisx}}-(\lambda-1)\norm{\thisoptu-\thisu}_{M_{\thisx}}
        \\
        &
        \le
        \lambda\norm{\thisoptu-\thisu}_{M_{\thisx}}-(\lambda-1)\norm{\nextu-\thisu}_{M_{\thisx}}.
        \end{split}
    \end{equation}
    Together, \eqref{eq:lipschitz-estim-opt} and \eqref{eq:descent2--lip} give
    \eqref{eq:unsquared-descent1}
    for $\zeta_2=1-1/\lambda$.
    %
    %We still have to show \eqref{eq:optu-from-nextu}. This follows simply from
    %\eqref{eq:update-lip}, as 
    %\[
    %    \norm{\thisoptu-\thisu}_{M_{\thisx}} 
    %    \le 
    %    \norm{\thisoptu-\nextu}_{M_{\thisx}} + \norm{\nextu-\thisu}_{M_{\thisx}}
    %    \le
    %    (1 + \lipnum \MMax^{3/2}) \norm{\nextu-\thisu}_{M_{\thisx}}
    %    \qedhere
    %\]
\end{proof}

\subsection{Bridging local solutions}
\label{sec:bridging}

In order to finalise the proof of \eqref{eq:nonlinear-descent-general},
we will now use the perturbed local linearised solution $\thisoptu$ to bridge 
between the local linearised solutions $\thisrealoptu$ and $\nextrealoptu$ .
For this we again need the Aubin property of $H_{\realoptu}^{-1}$.

\begin{lemma}
    \label{lemma:sensitivity}
    Let $\MCond$ and $L_2$ be as in \eqref{eq:constants}.
    Assume that \eqref{eq:ass-di} and \eqref{eq:ass-k} hold,
    and that $H_{\realoptx}^{-1}$ has the Aubin property at $0$ for $\realoptu$.
    Suppose, moreover, that \eqref{eq:unsquared-descent1} holds for any choice of
    $\thisoptu \in \inv H_{\thisx}(-\nu^i)$, and that
    \begin{equation}
        \label{eq:sensitivity-y}
        \lipopt \MCond L_2 \norm{\Pnl \realopty} < \zeta_2.
    \end{equation}
    Under these conditions there exist $\delta > 0$ and $\zeta_3 \in (0, 1]$ such 
    that if $\nextu \in \inv \Ti(0)$ and $\thisrealoptu \in H_{\thisx}^{-1}(0)$
    are given and \eqref{eq:steplength-ass} holds with this $\delta$, 
    then there exists $\nextrealoptu \in \inv H_{\nextx}(0)$ satisfying
    \begin{equation}
        \label{eq:descent-opt-switch}
        %\notag
        \norm{\thisrealoptu-\thisu}_{M_{\thisx}}
        \ge
        \zeta_3
        \norm{\nextu-\thisu}_{M_{\thisx}}
        +
        \norm{\nextrealoptu-\nextu}_{M_{\nextx}}.
    \end{equation}
    That is, the descent inequality \eqref{eq:nonlinear-descent-general} holds.
    %Moreover, for a choice of $\thisoptu \in H_{\thisx}^{-1}(-\nu^i)$, it holds
    %\begin{equation}
    %    \label{eq:realopt-opt-estimate}
    %    \norm{\thisrealoptu-\thisu}_{M_{\thisx}}
    %    \ge
    %    (1-\eta)\norm{\thisoptu-\thisu}_{M_{\thisx}},
    %\end{equation}
    %for some $\eta < 1-1/\lambda$. 
\end{lemma}
\begin{proof}
    Suppose there exists $\eta \in (0, \zeta_2)$, independent of $i$,
    and some $\thisoptu \in \inv H_{\thisx}(-\nu^i)$ and $\nextrealoptu \in \inv H_{\nextx}(0)$, 
    such that the double sensitivity estimate
    \begin{equation}
        \label{eq:linear-sensitivity-estimate}
        \norm{\thisoptu-\nextrealoptu}_{M_{\nextx}}
        +
        \norm{\thisoptu-\thisrealoptu}_{M_{\thisx}}
        \le
        \eta \norm{\thisu-\nextu}_{M_{\thisx}}
    \end{equation}
    holds.
    Then we may proceed as follows.
    We recall that by \eqref{eq:unsquared-descent1}, we have
    \begin{equation}
        \notag
        \norm{\thisoptu-\thisu}_{M_{\thisx}}
        \ge
        \zeta_2
        \norm{\nextu-\thisu}_{M_{\thisx}}
        +
        \norm{\thisoptu-\nextu}_{M_{\nextx}}.
    \end{equation}
    Two applications of the triangle inequality give
    \[
        \norm{\thisrealoptu-\thisu}_{M_{\thisx}}
        \ge
        \left(
        \zeta_2
        \norm{\nextu-\thisu}_{M_{\thisx}}
        -
        \norm{\thisoptu-\thisrealoptu}_{M_{\thisx}}
        -
        \norm{\thisoptu-\nextrealoptu}_{M_{\nextx}}
        \right)
        +
        \norm{\nextrealoptu-\nextu}_{M_{\nextx}}.
    \]
    Employing the sensitivity estimate \eqref{eq:linear-sensitivity-estimate},
    we now get \eqref{eq:descent-opt-switch} with $\zeta_3 \defeq (\zeta_2-\eta) > 0$.
    
    We still have to show \eqref{eq:linear-sensitivity-estimate}.
    Using Lemma \ref{lemma:steplength}, with $\delta>0$ small enough,
    we may take
    \[
        \thisoptu \in \argmin\{\norm{\thisrealoptu-v} \mid v \in \inv H_{\thisx}(-\nu^i)\}.
    \]
    We then have $\EE -\nu^i \in H_{\nextx}(\thisoptu)$ for
    \[
        \EE \defeq \EE(\thisoptu; \nextx; \thisx)
        \defeq
        \begin{pmatrix}
            (K_{\nextx} - K_{\thisx})^* \thisopty \\
            (K_{\thisx} - K_{\nextx}) \thisoptx + c_{\thisx} - c_{\nextx}
        \end{pmatrix}.
    \]
    We pick $\lipoptF > \lipopt$. Again with $\delta>0$ small enough,
    an application of Lemma \ref{lemma:translated-aubin} yields for 
    some $\rho', \delta' > 0$ that
    \begin{equation}
        \label{eq:bridge-aubin1}
        \inf_{v\,:\, w \in H_{\thisx}(v)} \norm{u-v}
            \le \lipoptF \norm{w-H_{\thisx}(u)},
        \quad
        (
         \norm{u-\realoptu} \le \delta',
         \,
         \norm{w} \le \rho'
         ).
    \end{equation}
    With $w=-\nu^i$, $v=\thisoptu$ and $u=\thisrealoptu$, 
    minding that $\thisrealoptu \in \inv H_{\thisx}(0)$,
    we thus have
    \begin{equation}
        \label{eq:aubin-thisoptu-thisrealoptu}
        \norm{\thisoptu-\thisrealoptu}_{M_{\thisx}}
        \le
        \lipoptF \MMax \norm{\nu^i},
    \end{equation}
    provided that $\norm{\thisrealoptu-\realoptu} \le \delta'$ and $\norm{\nu^i} \le \rho'$.
    These conditions are guaranteed by Lemma \ref{lemma:steplength} and choosing
    $\epsilon, \delta > 0$ small enough. 
    
    With $\epsilon,\delta > 0$ small, applying \eqref{eq:steplength-ass} and \eqref{eq:nestim-next},
    we can also make $\norm{\realoptu-\nextu}$ small enough that another application 
    Lemma \ref{lemma:translated-aubin} 
    guarantees the existence of
    \[
        \nextrealoptu \in \argmin\{\norm{\thisoptu-v} \mid v \in \inv H_{\nextx}(0)\},
    \]
    and shows
    \begin{equation*}
        \inf_{v\,:\, w \in H_{\nextx}(v)} \norm{u-v}
            \le \lipoptF \norm{w-H_{\nextx}(u)},
        \quad
        (
         \norm{u-\realoptu} \le \delta',
         \,
         \norm{w} \le \rho'
         ).
    \end{equation*}
    We may assume that $\rho', \delta'>0$ are the same as in \eqref{eq:bridge-aubin1}.
    With $w=0$, $v=\nextrealoptu$ and $u=\thisoptu$, we obtain
    \begin{equation}
        \label{eq:aubin-thisoptu-nextrealoptu}
        \norm{\thisoptu-\nextrealoptu}_{M_{\nextx}}
        \le
        \lipoptF \MMax(\norm{\EE}+\norm{\nu^i}),
    \end{equation}
    provided $\norm{\thisrealoptu-\realoptu} \le \delta'$.
    This condition was already verified for \eqref{eq:aubin-thisoptu-thisrealoptu}.

    To start approximating $\norm{\EE}$ and $\norm{\nu^i}$, we use Lemma \ref{lemma:e-term} 
    with $\basex= \nextx$ and $\basextwo=\thisx$, to obtain
    \begin{equation}
        \label{eq:e-estim1}
        \norm{\EE} \le L_2 \norm{\thisx - \nextx}
            \bigl(
                \norm{\Pnl \thisopty} + \norm{\thisoptx-\thisx} + \norm{\thisx - \nextx}
            \bigr).
    \end{equation}
    We approximate
    \[
        \norm{\Pnl \thisopty}
        \le
        \norm{\Pnl \realopty} + \norm{\Pnl(\realopty-\thisopty)},
    \]
    and
    \[
        \norm{\thisoptx-\thisx}
        \le
        \norm{\thisx-\realoptx} + \norm{\realoptx-\thisoptx}.
    \]
    %It follows for some constant $C > 0$ that
    %\[
    %    \norm{\Pnl \thisopty} + \norm{\thisoptx-\thisx}
    %    \le
    %    \norm{\Pnl \realopty} + C\norm{\realoptu-\thisoptu}
    %    + \norm{\thisu-\realoptu}.
    %\]
    Inserting these estimates back into \eqref{eq:e-estim1}, it follows 
    for some constant $C>0$ that
    \begin{equation}
        \label{eq:sens-e-estim}
        \norm{\EE} 
            \le
            L_2 \norm{\thisx-\nextx} A,
    \end{equation}
    where
    \[
        A \defeq
             \norm{\Pnl \realopty} + C\norm{\realoptu-\thisoptu}
                + \norm{\thisu-\realoptu} 
                + \norm{\thisx - \nextx}.
    \]
    Using \eqref{eq:ass-di}, we can for any $\epsilon > 0$ find $\delta''>0$ such that
    \[
        2\norm{\nu^i} \le \epsilon L_2 \norm{\thisx-\nextx}, \quad (\norm{\thisu-\nextu} < \delta'').
    \]
    The condition $\norm{\thisu-\nextu} < \delta''$ can be guaranteed through
    Lemma \ref{lemma:steplength} and choosing $\delta > 0$ small enough.
    Thus, using \eqref{eq:aubin-thisoptu-thisrealoptu}, \eqref{eq:aubin-thisoptu-nextrealoptu} 
    and \eqref{eq:sens-e-estim}, we have
    \[
        \begin{split}
        \norm{\thisoptu-\nextrealoptu}_{M_{\nextx}}
        +
        \norm{\thisoptu-\thisrealoptu}_{M_{\thisx}}
        &
        \le
        2 \lipoptF \MMax \norm{\nu^i}
        +
        \lipoptF \MMax L_2 \norm{\thisx-\nextx} A
        \\
        &
        \le
        \lipoptF \MMax L_2 \norm{\thisx-\nextx} (A + \epsilon)
        \\
        &
        \le
        \lipoptF \MCond L_2 (A + \epsilon) \norm{\thisu-\nextu}_{M_\thisx}.
        \end{split}
    \]
    In order to prove \eqref{eq:linear-sensitivity-estimate}, we thus need to force
    $
        \eta \defeq \lipoptF \MCond L_2(A+\epsilon) < \zeta_2.
    $
    As $\epsilon>0$ and $\lipoptF > \lipopt$ were arbitrary, it suffices to show
    $
        \lipopt \MCond L_2 A < \zeta_2.
    $
    Minding \eqref{eq:sensitivity-y}, we have
    \[
        0 < \zeta' \defeq \zeta_2 - \lipopt \MCond L_2 \norm{\Pnl \realopty}.
    \]
    Thus it remains to force
    \[
         C\norm{\realoptu-\thisoptu}
                + \MCond \norm{\thisu-\realoptu} 
                + \norm{\thisx - \nextx}
                < \zeta'/(\lipopt \MCond L_2).
    \]
    By Lemma \ref{lemma:steplength}, this holds for $\delta>0$ small enough.
    Thus \eqref{eq:linear-sensitivity-estimate} holds, and we may conclude the proof.
\end{proof}

\subsection{Combining the estimates}
\label{sec:combined-estimate}

\begin{lemma}
    \label{lemma:induction}
    Suppose \eqref{eq:ass-k} holds, and that 
    given any choice of $\realoptu^1 \in \inv{H_{x^1}}(0)$,
    \eqref{eq:nonlinear-descent-general} holds 
    for some $\nextrealoptu \in \inv{H_\nextu}(0)$, ($i=1,\ldots,k-1$).
    Suppose, moreover, that $0 \in H_{\realoptx}(\realoptu)$ and that
    $H_{\realoptx}^{-1}$ has the Aubin  property at $0$ for $\realoptu$.
    In this case, given $\epsilon > 0$, there exists $\delta_1 > 0$, 
    independent of $k$, such that if 
    \begin{equation}
        \label{eq:ind-ass}
        \norm{u^1-\realoptu} \le \delta_1,
    \end{equation}
    then there exists some $\realoptu^1 \in \inv{H_{x^1}}(0)$ 
    satisfing the bounds
    \begin{subequations}
    \label{eq:nestim-ind}
    \begin{align}
        \label{eq:nestim-2}
        %(u^k, \realoptu) & \in \CC(\Rad/2, \delta'),
        \norm{u^k-\realoptu} & \le \epsilon,
        \\
        \label{eq:nestim-1}
        %(u^k, \realoptu^k) & \in \CC(\Rad/2, \delta'),
        \norm{u^k-\realoptu^k} & \le \epsilon,
        \quad
        \text{and}
        \\
        \label{eq:nestim-3}
        %(\realoptu^k, \realoptu) & \in \CC(\Rad, 2\delta').
        \norm{\realoptu^k-\realoptu} & \le \epsilon.
    \end{align}
    \end{subequations}
\end{lemma}
\begin{proof}
    First of all, we require that $\delta_1 \in (0, \Rad/4)$, so
    that $\norm{u^1} \le 3\Rad/4$. We then show that
    \begin{equation}
        \label{eq:nestim-realoptu1}
        \norm{\realoptu^1-u^1} \le c \delta_1.
    \end{equation}
    for a choice of $\realoptu^1 \in \inv{H_{x^1}}(0)$ and some constant $c > 0$. 
    Indeed, letting $w \defeq \EE(\realoptu; u^1, \realoptu) \in H_{x^1}(\realoptu)$
    and choosing $\delta_1 \in (0, \delta)$ for $\delta$ small enough, 
    by Lemma \ref{lemma:translated-aubin} we have
    \[
        \inf_{\realoptu^1 \in \inv{H_{x^1}}(0)}
        \norm{\realoptu-\realoptu^1}
        \le
        \lipnum \norm{w}.
    \]
    Referring to Lemma \ref{lemma:e-term}, there exists $\delta' > 0$ such that
    if $\delta_1 \in (0, \delta')$, and \eqref{eq:ind-ass} holds, 
    then $\lipnum \norm{w} < L_2(\Rad+1)\norm{u^1-\realoptu}$.
    Consequently we see that \eqref{eq:nestim-realoptu1} by estimating
    \[
        \norm{\realoptu^1-u^1}
        \le
        \norm{u^1-\realoptu}
        +
        \norm{\realoptu-\realoptu^1}
        \le
        (1+L_2(\Rad+1)) \delta_1.
    \]
    
    Choosing 
    $
        \delta_1 \le \delta_1' \defeq \min\{\delta, \delta', \Rad/4, \Rad\zeta/(4\MCond c)\},
    $
    Lemma \ref{lemma:basic-bounds} now shows that
    \begin{align}
        \label{eq:bb-ui-uk-x}
        \norm{u^k-u^1} & \le (\MCond/\zeta) \norm{\realoptu^1-u^1},
        \quad
        \text{and}
        \\
        \label{eq:bb-hat-ui-uk-x}
        \norm{u^k-\realoptu^k} & \le \MCond \norm{\realoptu^1-u^1}.
    \end{align}
    Thus, using \eqref{eq:ind-ass}, \eqref{eq:nestim-realoptu1},
    and \eqref{eq:bb-ui-uk-x}, we get
    \[
        \norm{u^k - \realoptu} \le \norm{u^k - u^1} + \norm{u^1-\realoptu}
        \le 
        (\MCond c/\zeta+1)\delta_1.
    \]
    Likewise \eqref{eq:ind-ass}, \eqref{eq:nestim-realoptu1}, and \eqref{eq:bb-hat-ui-uk-x} give
    \[
        \norm{\realoptu^k-u^k}
        \le
        \MCond \norm{\realoptu^1-u^1}
        \le
        c \MCond \delta_1.
    \]
    Finally, these two estimates give
    \[
        \norm{\realoptu^k-\realoptu}
        \le
        \norm{\realoptu^k-u^k} + \norm{u^k-\realoptu}
        \le
        C \delta_1
    \]
    for
    $
        C \defeq 1+\MCond c (1+1/\zeta).
    $
    Choosing $\delta_1 \le \min\{\delta_1', \epsilon/C\}$, we get \eqref{eq:nestim-ind}.
\end{proof}

The following two theorems form our main convergence result.

\begin{theorem}
    \label{thm:descent}
    Let $\Rad$, $\MMax$, $\MCond$ and $L_2$ be as in \eqref{eq:constants}.
    Suppose \eqref{eq:ass-di} and \eqref{eq:ass-k} hold, 
    $\OverRelax=1$, and that $F^*$ is strongly convex on the subspace $\Ynl$.
    Let $\realoptu$ solve $0 \in H_{\realoptx}(\realoptu)$,
    and $\inv H_{\realoptx}$ has the Aubin property at $0$ for $\realoptu$ with 
    \begin{equation}
        \label{eq:pnl-realopty-bound}
         \lipopt \MCond L_2 \norm{\Pnl \realopty} < 1-1/\sqrt{1+1/(2\lipopt^2\MMax^4)}.
    \end{equation}
    Under these conditions, there exist $\delta_1 > 0$ and $\zeta \in (0, 1)$, such
    that if the initial iterate $u^1 \in X \times Y$ satisfies
    \begin{equation}
        \label{eq:u1-assumption}
        %(u^1, \realoptu) \in \CC(\Rad/4, \epsilon).
        \norm{u^1-\realoptu} \le \delta_1,
    \end{equation}
    and we let $\nextu \in \inv \Ti(0)$, ($i=1,2,3,\ldots$),
    then \eqref{eq:nonlinear-descent-general} holds, i.e.,
    \[
        \norm{\thisu-\thisrealoptu}_{M_{\thisx}}
        \ge \zeta \norm{\nextu-\thisu}_{M_{\thisx}}
            + \norm{\nextu-\nextrealoptu}_{M_{\nextx}},
        \quad
        (i=1,2,3,\ldots).
    \]
\end{theorem}

\begin{proof}
    We prove the claim inductively using Lemma \ref{lemma:induction}.
    Indeed, we let $\delta > 0$ be small enough for \eqref{eq:steplength-ass} 
    to be satisfied for application in Lemma \ref{lemma:rmsquares} and 
    Lemma \ref{lemma:sensitivity}, and for \eqref{eq:norm-switch-ass} to be
    satisfied for Lemma \ref{lemma:norm-switch}.
    Then we first we apply Lemma \ref{lemma:descent} to get the estimate
    \eqref{eq:linear-descent-strong}.
    Then we pick $\zeta_1=1/2$ in Lemma \ref{lemma:norm-switch}, so that
    \[
        \zeta_2=1-1/\sqrt{1+1/(2 (\lipoptF)^2\MMax^4)}.
    \]
    Choosing $\lipoptF > \lipopt$ small enough, \eqref{eq:pnl-realopty-bound} 
    then guarantees \eqref{eq:sensitivity-y}.
    Inductively assuming that \eqref{eq:nonlinear-descent-general} holds
    for $i=1,\ldots,k-1$, we use %($k=1,2,3,\ldots$), we use 
    Lemma \ref{lemma:induction} to show that
    \[
        %(u^k, \realoptu) \in \CC(\Rad/2, \delta),
        \norm{u^k-\realoptu} \le \delta
        \quad\text{and}\quad
        \norm{\realoptu^k-\realoptu} \le \delta,
    \]
    provided $\delta_1>0$ is small enough.
    We then use Lemma \ref{lemma:norm-switch} to show that
    the squared local norm descent inequality \eqref{eq:descent-norm-switch} holds. 
    Next we apply Lemma \ref{lemma:rmsquares} to show that unsquared local descent inequality
    \eqref{eq:unsquared-descent1} holds. Finally, we employ Lemma \ref{lemma:sensitivity} 
    to derive \eqref{eq:nonlinear-descent-general} for $i=k$. As $k$ was arbitrary,
    we may conclude the proof.
\end{proof}

\begin{theorem}
    \label{thm:conv}
    Suppose the conditions of Theorem \ref{thm:descent} hold for
    $\realoptu=(\realoptx, \realopty)$ and $u^1=(x^1, y^1)$.
    In that case there exists $\delta_1 > 0$ such that provided the initialisation $u^1$ 
    satisfies $\norm{u^1-\realoptu} \le \delta_1$, then the iterates $(\thisx, \thisy)$ 
    produced by Algorithm \ref{algorithm:nl-cp} or Algorithm \ref{algorithm:nl-cp-lin} 
    converge to a solution $u^*=(x^*, y^*)$ of \eqref{eq:nl-oc}, i.e., a critical point 
    of the problem \eqref{eq:nonlinear-problem}.
\end{theorem}

\begin{proof}
    We pick $\delta_1>0$ small enough that \eqref{eq:u1-assumption}
    is satisfied, and that Lemma \ref{lemma:induction}
    guarantees the assumption \eqref{eq:di-lemma-ass} of
    Lemma \ref{lemma:disc-convergence}.
    Then Theorem \ref{thm:descent} and Lemma \ref{lemma:disc-convergence} 
    verify the assumptions of the general convergence result
    Theorem \ref{thm:conv-idea}, from which the claim follows.
\end{proof}

\subsection{Some remarks}
\label{sec:details-remarks}

\begin{remark}[Convergence to another solution]
    In principle the solution $u^*$ discovered in Theorem \ref{thm:conv}
    may be distinct from $\realoptu$, also solving \eqref{eq:nl-oc}.
    %\TODO{can say something to the contrary?}
\end{remark}

\begin{remark}[Small non-linear dual]
    \label{rem:nldual}
    The condition \eqref{eq:pnl-realopty-bound} forces
    $\norm{\Pnl \realopty}$ to be small. As we will 
    further discuss in Section \ref{sec:l1reg}, in the applications 
    that we are primarily interested in, involving solving
    $\min_x \norm{f-T(x)}^2/2 + \alpha R(x)$,
    this this corresponds to $\norm{f-T(\realoptx)}$
    being small. This says that the noise level of the data $f$
    and regularisation parameter $\alpha$ have to be low.
\end{remark}

\begin{remark}[Inexact solutions]
    It is possible to accommodate for inexact solutions $\nextu$ in the
    proof, if we relax the requirement $D^i(\thisu)=0$ to 
    $D^i(\thisu) \to 0$, $\norm{D^i(\thisu)} \le \epsilon$, 
    and take $\thisrealoptu$ to solve $0 \in H_{\thisx}(\thisoptu) + D^i(\thisu)$.
    Since this involves significant additional technical detail that complicates
    the already very technical proof, we have opted not to include this
    generalisation.
\end{remark}

\begin{remark}[Switch of local metric]
    The shift to the new local metric $M_{\nextx}$, 
    done in Lemma \ref{lemma:norm-switch} using the
    strong convexity of $F^*$ on $Y_\NL$, 
    can also be done similarly to the removal of squares
    in Lemma \ref{lemma:rmsquares}.
    %if $\lambda\MMin/\MMax \ge 1$.
    %This is the case if $1+1/(\lipnum\MMax)^2 \ge \MMax^2/\MMin^2$, i.e.,
    %$(\MMin\MMax)^2 + \MMin^2/\lipnum^2 \ge \MMax^4$, i.e.,
    %$\MMin^2/\lipnum^2 \ge \MMax^4(1-\MMin^2)$, i.e.,
    %$\lipnum^2 \le \MMin^2/\MMax^4(1-\MMin^2)$, i.e.,
    %$\lipnum \le \MMin/(\MMax^2\sqrt{1-\MMin^2})$. This follows
    %if $\lipnum\MMax \le \MMin$. This says that the Lipschitz 
    %factor $\lipnum$ times $\MMax$ must be bounded by the maximum of 
    %the reciprocal of the condition numbers of the operators $M_{\thisx}$ 
    %in a neighbourhood of a solution.
    This suggests that the strong convexity might not be necessary.
    In practise we however need the strong convexity for the 
    Lipschitz continuity, so there is little benefit from that.
    Moreover, the required strong convexity exists in case of
    regularisation problems of the form discussed in 
    Remark \ref{rem:nldual}. As we are primarily interested in
    applying the method to such problems, assuming strong convexity
    is natural.
\end{remark}

\begin{remark}[Varying step lengths]
    \label{rem:steplength}
    Strictly speaking, we do not need the whole force of
    \eqref{eq:ass-k} for the bound \eqref{eq:m-bounds}.
    We can allow for $\sigma, \tau$ vary on each iteration,
    and even preconditioning operators in place of simple 
    step length parameters, cf.~\cite{pock2011iccv}. 
    Indeed, with $\sigma^i, \tau^i$ dependent on $\thisu$,
    we see that $\MMin^2 I \le M_{x}$ if
    \[
        \frac{1}{2} \MMin^2 \le \frac{1}{\sigma^i\tau^i} - \norm{K_{x^i}}^2.
    \]
    If now $\epsilon,\theta>0$ are such that
    $\epsilon \norm{K_{x^i}}^2 > \MMin^2/2 > 0$, we 
    see that it suffices to have
    \begin{equation}
        \label{eq:stepsize-i}
        \sigma^i\tau^i (1+\epsilon) \norm{K_{x^i}}^2 < 1.
    \end{equation}
    Likewise $M_{x^i} \le \MMax^2 I$ if $\sigma^i, \tau^i$ are bounded from
    below, and $\{\norm{K_{x^i}}\}_{i=1}^\infty$ is bounded.
    Since $\norm{u^i} \le \Rad$, the latter follows from the 
    assumption $K \in C^2(X; Y)$. Likewise, 
    $\{\norm{K_{x^i}}\}_{i=1}^\infty$ is bounded away from
    zero if
    \begin{equation}
        \label{eq:below-bound}
        d \defeq \inf_{\norm{x} \le \Rad} \norm{\grad K(x)} > 0.
    \end{equation}
    This is the case with the operators in Section \ref{sec:appl}.
    If \eqref{eq:below-bound} is satisfied, to obtain \eqref{eq:m-bounds}, 
    it thus suffices to choose $\sigma^i$ and $\tau^i$ such that
    \eqref{eq:stepsize-i} holds for fixed $\epsilon>0$, and to choose
    $\MMin$ such that
    \[
        \epsilon d > \MMin^2/2 > 0.
    \]
    
    This alone is however not enough to show convergence of the method
    with varying step lengths. There is one further difficulty in the 
    switch of local norms from  $\norm{\cdot}_{M_{\thisx}}$ to
    $\norm{\cdot}_{M_{\nextx}}$ in Lemma \ref{lemma:norm-switch}.
    Specifically, the first equality in \eqref{eq:norm-switch-1} 
    does not hold. Defining
    \[
        M_{\nextx}' \defeq
            \begin{pmatrix}
                I/\tau^i & -K_\basex^* \\
                -\OverRelax K_\basex & I/\sigma^i
            \end{pmatrix},
    \]
    and otherwise using $\sigma=\sigma^i$ and $\tau=\tau^i$
    in the definition of $M_{\thisx}$ for each $i$,
    we can however calculate
    \[
        \norm{\nextu-\thisoptu}_{M_{\nextx}'}^2
        -
        \norm{\nextu-\thisoptu}_{M_{\nextx}}^2
        \ge
        \left(\frac{1}{\tau^i}-\frac{1}{\tau^{i+1}}\right)\norm{\nextx-\thisoptx}^2
        +
        \left(\frac{1}{\sigma^i}-\frac{1}{\sigma^{i+1}}\right)\norm{\nexty-\thisopty}^2.
    \]
    Then estimating
    $
        \norm{\nextu-\thisoptu}_{M_{\thisx}}^2
        -
        \norm{\nextu-\thisoptu}_{M_{\nextx}'}^2
    $
    similarly to \eqref{eq:norm-switch-1},
    we can prove following Lemma \ref{lemma:norm-switch} that
    \begin{equation*}
        %\label{eq:descent-norm-switch-remark}
        \tag{$\mathrm{\opt{D}^2}$-M}
        \norm{\thisu-\thisoptu}_{M_{\thisx}}^2
        \ge \frac{\zeta_1}{2} \norm{\nextu-\thisu}_{M_{\thisx}}^2
            + \norm{\nextu-\thisoptu}_{M_{\nextx}}^2
    \end{equation*}
    provided
    \begin{equation}
        \label{eq:tau-sigma-ratio}
        \adaptabs{\frac{1}{\tau^i}-\frac{1}{\tau^{i+1}}}, \adaptabs{\frac{1}{\sigma^i}-\frac{1}{\sigma^{i+1}}} < (1-\zeta_1)\MMin/2.
    \end{equation}
    
    Since $x \mapsto \norm{K_{x}}$ is Lipschitz close to $\realoptx$,
    if we start close enough to $\realoptu$ that
    $\norm{\thisx-\nextx}$ stays small, \eqref{eq:tau-sigma-ratio} can be
    made to hold for good choices of $\tau^i$, $\sigma^i$.
    Indeed, let us choose $\epsilon, \tau_0, \sigma_0 > 0$ 
    such that $\tau_0 \sigma_0 < 1/(1+\epsilon)$, and then
    \begin{equation}
        \label{eq:stepsize-rule-adapt}
        \tau^i \defeq \tau_0/L^i,
        \text{ and }
        \sigma^i \defeq \sigma_0/L^i,
        \quad
        \text{for}
        \quad
        L^i \defeq \sup_{j=1,\ldots,i} \norm{K_{x^j}}.
    \end{equation}
    Then \eqref{eq:stepsize-i} holds, $\{\tau^i\}_{i=1}^\infty$
    and $\{\sigma^i\}_{i=1}^\infty$ are decreasing,
    and \eqref{eq:tau-sigma-ratio} amounts to
    \[
        \adaptabs{L^i-L^{i+1}} < \frac{(1-\zeta_1)\MMin}{2\max\{\sigma_0,\tau_0\}}.
    \]
    By the above considerations, this holds if we initialise 
    $u^1$ close enough to $\realoptu$. 
\end{remark}

\section{Lipschitz estimates}
\label{sec:lipschitz}

We now need to show the Aubin property of the inverse $H_{\thisx}^{-1}$
of the set-valued map $H_{\thisx}$, and to bound $\norm{\Pnl \realopty}$.
We will calculate
$\lip{\inv H_{\thisx}}(0|\realoptu)$ through the Mordukhovich criterion,
which brings us to the topic of graphical differentiation of set-valued maps. 
We  will introduce the necessary tools in Section \ref{sec:graphdiff},
following \cite{rockafellar-wets-va}; another treatment also covering the
infinite-dimensional case can be found in \cite{mordukhovich2006variational}.
%Before this we perform the sensitivity analysis based on the Lipschitz
%estimates in \ref{sec:bridging}, and 
Afterwards we derive bounds on $\lip{\inv H_{\thisx}}(0|\realoptu)$ for some general 
class of maps in Section \ref{sec:lipbounds}. These will then be used in the 
following Section \ref{sec:l1reg} and Section \ref{sec:l2l1reg}
 to study important
special cases. These include in particular total variation ($\TV$) and 
total generalised variation ($\TGV^2$) \cite{bredies2009tgv} regularisation.

\subsection{Differentials of set-valued maps}
\label{sec:graphdiff}

%In order to derive Lipschitz factors, we will use the Mordukhovich criterion,
%for which we need to introduce various concepts of differentials of set-valued
%maps. We try to be brief, and for full details refer the reader to
%\cite{rockafellar-wets-va, mordukhovich2006variational}.

Let $S: X \setto Y$ be a set-valued map between finite-dimensional 
Hilbert spaces $X, Y$. The graph of $S$ is
\[
    \graph S \defeq \{(x,y) \mid y \in H(x) \}.
\]
The \emph{outer limit} at $x$ is defined as
\[
    \limsup_{x' \to x} S(x) \defeq \{ y \in Y \mid \text{there exist } \thisx \to x \text{ and } \thisy \in S(\thisx) \text{ with } \thisy \to y \}.
\]
The \emph{inner limit} is defined as
\[
    \liminf_{x' \to x} S(x) \defeq \{ y \in Y \mid \text{for every } \thisx \to x \text{ there exist } \thisy \in S(\thisx) \text{ with } \thisy \to y \}.
\]

Pick $x \in X$ and $y \in S(x)$. The \emph{graphical derivative of $S$ at $x$ for $y$},
denoted $DS(x|y): X \setto Y$, is defined by
\[
    DS(x|y)(w) = \limsup_{\tau \downto 0,\, w' \to w} \frac{S(x+\tau w')-y}{\tau}.
\]
Geometrically, $\graph DS(x|y)$ is the tangent cone to $\graph S$ at $(x, y)$.
If $S$ is single-valued and differentiable, then $DS(x|y)=\grad S(x)$ for $y=S(x)$.
Observe that $DS(x|y)$ satisfies
\[
    z \in DS(x|y)(w) \iff w \in D(\inv S)(y|x)(z).
\]

\def\regD{{\widehat D}}
There are also various other definitions of differentials for set-valued maps.
In particular, the \emph{regular derivative of $S$ at $x$ for $y$}, denoted 
${\regD}S(x|y): X \setto Y$, is defined by
\[
    {\regD}S(x|y) = \liminf_{(x^i, y^i) \to (x, y)} \graph(DS(x^i|y^i)).
\]
The importance of the regular derivative to us lies in following.
The map $S$ is said to be \emph{graphically regular} at $(x, y)$ if ${\regD}S(x|y)=DS(x|y)$.
We stress that this correspondence does not hold generally. 
If it does, we may express the \emph{coderivative} $D^*S(x|y)$ as
\begin{equation}
    \label{eq:deriv-coderiv}
    D^*S(x|y)=[D S(x|y)]^{*+},
\end{equation}
where
\[
    H^{*+}(w) \defeq \{ z \mid \iprod{z}{q} \le \iprod{w}{v} \text{ when } v \in H(q)\}.
\]
Geometrically, $\graph D^*S(x|y)$ is a normal cone to $\graph S$ at $(x, y)$
rotated such that it becomes adjoint to $DS(x|y)$ in the sense \eqref{eq:deriv-coderiv}.
Without graphical regularity, the coderivative has to be defined through other means
\cite{rockafellar-wets-va}; we will however always assume graphical regularity.

In our forthcoming analysis, we will occasionally employ the tangent and normal
cones to a convex set $A$ at $y$. These are denoted $T_A(y)$ and $N_A(y)$, respectively.

Finally, with the above concepts defined, we may state a version of
the Mordukhovich criterion \cite[Theorem 9.40]{rockafellar-wets-va}
sufficient for our purposes.

\begin{theorem}
    \label{thm:morduk}
    Let $S: X \setto Y$. Suppose $\graph S$ is locally closed at $(x, y)$ and
    \begin{equation}
        \label{eq:mordukhovich-qc-1}
         D^*S(x|y)(0) = \{0\}.
    \end{equation}
    Then
    \begin{equation}
        \label{eq:morduk-lip}
        \lip S(x|y) = |D^* S(x|y)|^+, 
    \end{equation}
    where the \emph{outer norm}
    \[
        |H|^+ \defeq \sup_{\norm{w}\le 1} \sup_{z \in H(w)} \norm{z}.
    \]
\end{theorem}

We want to translate this result to be stated in terms of $DS(x|y)$
for $\inv S$, since the graphical derivative is easier to obtain
than the coderivative.

\begin{proposition}
    \label{proposition:inv-lip}
    Suppose $S$ is graphically regular 
    and $\graph S$ locally closed at $(x, y)$. Then
    \begin{equation}
        \label{eq:inv-lip-expr}
        \lip{\inv S}(y|x) \le
            \sup\{
                \norm{z}
                \mid
                 \iprod{z}{q} \le \norm{v} \text{ when } q \in D S(x|y)(v)
                \}.
    \end{equation}
\end{proposition}
\begin{proof}
    From \eqref{eq:deriv-coderiv}, we have that $z \in D^*{\inv S}(y|x)(w)$ if
    \[
        \iprod{z}{q} \le \iprod{w}{v} \text{ when } (q, v) \text{ satisfy } v \in D{\inv S}(y|x)(q).
    \]
    By the symmetricity of graphical differentials for $S$ and $\inv S$, this is the same as
    \[
        \iprod{z}{q} \le \iprod{w}{v} \text{ when } (q, v) \text{ satisfy } q \in D S(x|y)(v),
    \]
    Thus, if \eqref{eq:mordukhovich-qc-1} is satisfied, 
    Theorem \ref{thm:morduk} gives
    \[
        \begin{split}
        \lip{\inv S}(y|x) & =
            \sup_{\norm{w}\le 1} \sup\{ \norm{z} \mid z \in D^*{\inv S}(y|x)(w)\}.
            \\
            &
            =
            \sup_{\norm{w}\le 1} \sup\{
                \norm{z}
                \mid
                 \iprod{z}{q} \le \iprod{w}{v} \text{ when } q \in D S(x|y)(v)
                \}
        \\
         & \le
            \sup\{
                \norm{z}
                \mid
                 \iprod{z}{q} \le \norm{v} \text{ when } q \in D S(x|y)(v)
                \}
        \end{split}
    \]
    If \eqref{eq:mordukhovich-qc-1} is not satisfied,
    then by \eqref{eq:deriv-coderiv} the supremum in \eqref{eq:inv-lip-expr} is infinite.
    Thus \eqref{eq:inv-lip-expr} holds whether \eqref{eq:mordukhovich-qc-1} holds or not.
\end{proof}

\subsection{Bounds on Lipschitz factors}
\label{sec:lipbounds}

We now want to approximate the local Lipschitz factor $\lip{\inv H_{\realoptx}}(0|\realoptu)$ 
of $H_{\realoptx}^{-1}$ at $0$ for $\realoptu$. We apply Proposition \ref{proposition:inv-lip},
assuming that $H_{\realoptx}$ is graphically regular 
and $\graph H_{\realoptx}$ is locally closed at $(\realoptu, 0)$. Then
\[
    \lip{\inv H_{\realoptx}}(0|\realoptu) 
    \le
        \sup\{
            \norm{z}
            \mid
             \iprod{z}{q} \le \norm{v} \text{ when } q \in D H_{\realoptx}(\realoptu|0)(v)
            \}.
\]
Writing $u=(x, y)$ and $v=(\xi, \nu)$, we have
\begin{equation}
    \label{eq:dh-expand}
    D H_{\realoptx}(\realoptu|0)(v)
    = \begin{pmatrix}
            D G(\realoptx|-K_{\realoptx}^* \realopty)(\xi) + K_{\realoptx}^* \nu \\
            D F^*(\realopty|K_{\realoptx} \realoptx+c_{\realoptx})(\nu) - K_{\realoptx} \xi
      \end{pmatrix}.
\end{equation}
Suppose there exist self-adjoint linear maps $\barG: X \to X$ and $\barF: Y \to Y$
and (possibly trivial) subspaces $V_G \subset X$ and $V_F \subset Y$ such that
\begin{align}
    \notag
    D G(\realoptx|-K_{\realoptx}^* \realopty)(\xi) & =
    \begin{cases}
        \barG \xi + V_G^\perp, & \xi \in V_G \\
        \emptyset, & \xi \not\in V_G, \quad \text{ and }
    \end{cases}
    \\
    \notag
    D F^*(\realopty|K_{\realoptx} \realoptx+c_{\realoptx})(\nu)& =
    \begin{cases}
        \barF \nu + V_F^\perp, & \nu \in V_F, \\
        \emptyset, & \nu \not\in V_F.
    \end{cases}
\end{align}
Then
\begin{equation}
    \label{eq:dh-form}
    D H_{\realoptx}(\realoptu|0)(v)=
    \begin{cases}
        Av + V^\perp, & v \in V, \\
        \emptyset, & v \not\in V.
    \end{cases}
\end{equation}
for
\[
    A=
    \begin{pmatrix}
        \barG & K_{\realoptx}^* \\
        -K_{\realoptx} & \barF
    \end{pmatrix},
    \quad
    \text{and}
    \quad
    V=V_G \times V_F.
\]
With $P_V$ the orthogonal projection operator into $V$, 
this allows us to approximate
\begin{equation}
    \label{eq:aubin-approx-a}
    \begin{split}
    \lip{\inv H_{\realoptx}}(0|\realoptu) 
    &
    \le
        \sup\{
            \norm{z}
            \mid
             \iprod{z}{A v + p} \le \norm{v}
             \text{ when }
             v \in V,\, p \in V^\perp
            \}
    \\
    &
    =
        \sup\{
            \norm{z}
            \mid
             \norm{P_V A^* z} \le 1
             \text{ when }
             z \in V
            \}
    \\
    &
    =
        \sup\{
            \norm{P_V z}
            \mid
             \norm{P_V A^* P_V z} \le 1
            \}
    \\
    &
    \le
        \inf\{ c^{-1}
            \mid
            c \norm{P_V z} \le \norm{P_V A^* P_V z} \text{ for all } z
            \}.
    \end{split}
\end{equation}
In the following lemma, we show that $c > 0$. To do so, we have to assume 
various forms of boundedness from the involved operators. We introduce the operator
$\Xi$ as a way to make the estimate hold for a range of regularisation
parameters; the details of the procedure will follow the lemma in Section \ref{sec:l1reg}.

\def\gammax{{\tilde\gamma}}
\def\gammat{{\bar\gamma}}
\def\XiPF{\Xi_{V,F}}
\def\XiPG{\Xi_{V,G}}
\def\XiPF{\Xi_{V,F}}
\def\XiPG{\Xi_{V,G}}
\def\Fgamma{F_\gammat}
\def\Ggamma{G_\gammat}
\def\Fgammax{F_\gammax}

\begin{lemma}
    \label{lemma:liplemma}
    Let $\Xi: X \times Y \to X \times Y$ be a self-adjoint positive definite linear operator,
    $\Xi(x, y)=(\Xi_G(x), \Xi_F(y))$. Suppose that $\Xi_G$ commutes with $\barG$ and
    $P_{V_G}$, that $\Xi_F$ commutes with $\barF$ and $P_{V_F}$,
    and that one of the following conditions hold. 
    \begin{enumroman}
        \item
        \label{item:liplemma-gsc-fsc}
        $\barG \ge \gammat \Xi_G$ and $\barF \ge \gammat \Xi_F$ for some $\gammat > 0$.
        \item
        \label{item:liplemma-g0-fsc}
        $\barG=0$, $V_G=X$, $\Xi_G=I$ and $M \Xi_F \ge \barF \ge \gammat \Xi_F$ for some $M, \gammat > 0$,
        as well as
        \begin{equation}
            \label{eq:liplemma-nullspace-cond}
            %P_{V_F} K_{\realoptx} \zeta = 0 \implies \zeta = 0, \quad (\zeta \in X).
            \norm{P_{V_F} K_{\realoptx} \zeta} \ge a \norm{\zeta}, \quad (\zeta \in X).
        \end{equation}
    \end{enumroman}
    Then there exists a constant $c = c(M, \gammat, a)$, such that
    \begin{equation}
        \label{eq:liplemma-lower-bound}
        \norm{P_V A^* P_V z} \ge c \norm{\Xi P_V z}, \quad (z \in X \times Y).
    \end{equation}
\end{lemma}
\begin{proof}
    Let us write
    \[
        A_V \defeq P_V A^* P_V z
        =
        \begin{pmatrix}
            \barG_V & -K_{V}^* \\
            K_{V} & \barF_V
        \end{pmatrix}
    \]
    for $\barG_V \defeq P_{V_G} G P_{V_G}$, $\barF_V \defeq P_{V_F} F P_{V_F}$,
    and $K_{V} \defeq P_{V_F} K_{\realoptx} P_{V_G}$. Then $z=(\zeta, \eta)$ satisfies
    \[
        \begin{split}
        \norm{A_V z}^2
        & =
        \norm{\barG_V \zeta - K_{V}^* \eta}^2
        +
        \norm{\barF_V \eta + K_{V} \zeta}^2
        \\
        &
        =
        \norm{\barG_V \zeta}^2 + \norm{K_{V}^* \eta}^2
        - 2\iprod{\barG_V \zeta}{K_{V}^* \eta}
        %\\
        %&
        %\phantom{=\ }
        + \norm{K_{V} \zeta}^2 + \norm{\barF_V \eta}^2
        + 2\iprod{K_{V} \zeta}{\barF_V \eta}
        \end{split}
    \]

    We consider point \ref{item:liplemma-gsc-fsc} first.
    We write $\barG_V = \Ggamma + \gammat \XiPG$ and $\barF_V = \Fgamma + \gammat \XiPF$
    for $\XiPF \defeq \Xi_F P_{V_F}$ and $\XiPG \defeq \Xi_G P_{V_G}$.
    Observe that the self-adjointness and positivity
    of $\XiPG$ and the commutativity with $\barG$ yield
    \[
        \iprod{\XiPG \zeta}{\Ggamma \zeta}
        =
        \iprod{\XiPG^{1/2} \zeta}{\Ggamma \XiPG^{1/2} \zeta}
        \ge 0.
    \]
    Therefore 
    \[
        \begin{split}
        \norm{G_V \zeta}^2 - 2\iprod{G_V \zeta}{K_{V}^* \eta} + \norm{K_{V}^* \eta}^2
            &
            =
            \norm{\Ggamma \zeta}^2
            +
            2\gammat \iprod{\XiPG \zeta}{\Ggamma \zeta}
            + \gammat^2 \norm{\XiPG \zeta}^2
            \\
            &
            \phantom{ = }
            -2\iprod{\Ggamma \zeta}{K_{V}^* \eta}
            -2\gammat\iprod{\zeta}{K_{V}^* \eta}
            + \norm{K_{V}^* \eta}^2
            \\
            &
            \ge
            \gammat^2 \norm{\XiPG \zeta}^2
            -2\gammat\iprod{\zeta}{K_{V}^* \eta}.
        \end{split}
    \]
    Analogously
    \[
        \norm{K_{V} \zeta}^2 
        + 2\iprod{K_{V} \zeta}{\barF_V \eta}
        + \norm{\barF_V \eta}^2
        \ge
        \gammat^2 \norm{\XiPF \eta}^2
        +2\gammat\iprod{K_{V}\zeta}{\eta},
    \]
    so that
    \[
        \norm{A_V z}^2
            \ge
            \bigl(
            \gammat^2 \norm{\XiPG \zeta}^2
            -2\gammat\iprod{\zeta}{K_{V}^* \eta}
            \bigr)
            +
            \bigl(
            \gammat^2 \norm{\XiPF \eta}^2
            +2\gammat\iprod{K_{V}\zeta}{\eta}
            \bigr)
            =
            \gammat^2 \norm{\Xi P_V z}^2.
    \]
    This shows that \eqref{eq:liplemma-lower-bound} holds with $c=\gammat$
    in case \ref{item:liplemma-gsc-fsc}.
    
    Consider then case \ref{item:liplemma-g0-fsc}.
    Now $P_{V_G}=I$, so $K_{V} = P_{V_F} K_{\realoptx}$ and $\XiPG=I$.
    %By \eqref{eq:liplemma-nullspace-cond}, we moreover have the existence of $c' > 0$,
    %independent of $\Xi$, such that
    %\begin{equation}
    %    \label{eq:liplemma-kv-bound}
    %    \norm{K_V \zeta} \ge c' \norm{\zeta}, \quad (\zeta \in X).
    %\end{equation}
    Let us pick arbitrary $\gammax \in (0, \gammat/2)$.
    Then $\barF_V=\Fgammax + \gammax \XiPF$.
    Expanding, we have
    \begin{equation}
        \notag
        \begin{split}
        \norm{A_V z}^2
        &
        =
        \norm{K_{V}^* \eta}^2
        + \norm{K_{V} \zeta}^2 + \norm{\barF_V \eta}^2
        + 2\iprod{K_{V} \zeta}{\barF_V \eta}
        \\
        &
        =
        \norm{K_{V}^* \eta}^2
        + \norm{K_{V} \zeta}^2
        + \norm{\Fgammax \eta}^2
        \\
        &
        \phantom{=}\ 
        + 2\gammax\iprod{\XiPF \eta}{\Fgammax \eta}
        + \gammax^2\norm{\XiPF\eta}^2
        + 2\iprod{K_{V} \zeta}{\Fgammax \eta}
        + 2\gammax\iprod{K_{V} \zeta}{\eta}.
        \end{split}
    \end{equation}
    Using the Young's inequalities
    \[
        2\gammax\iprod{K_{V} \zeta}{\eta}
        \le \gammax^2\norm{\zeta}^2 + \norm{K_{V}^*\eta}^2,
    \]
    and
    \[
        2\iprod{K_{V} \zeta}{\Fgammax \eta}
        \le
        \mu\norm{K_{V} \zeta}^2 + (1/\mu)\norm{\Fgammax \eta}^2,
        \quad (\mu > 0),
    \]
    give
    %together with \eqref{eq:liplemma-kv-bound} give
    \begin{equation}
        \label{eq:av-case2-est1}
        \norm{A_V z}^2
        \ge
        (1-\mu)\norm{K_{V} \zeta}^2
        + (1-1/\mu)\norm{\Fgammax \eta}^2
        + 2\gammax\iprod{\XiPF\eta}{\Fgammax \eta}
        + \gammax^2\norm{\XiPF\eta}^2
        - \gammax^2\norm{\zeta}^2.
    \end{equation}
    Observe, that the self-adjointness and positivity
    of $\XiPF$ and the commutativity with $\barF$ yield
    \[
        M \XiPF \ge \Fgammax \ge (\gamma/2) \XiPF,
    \]
    as well as
    \[
        \iprod{\XiPF \eta}{\Fgammax \eta}
        =
        \iprod{\XiPF^{1/2} \eta}{\Fgammax \XiPF^{1/2} \eta}
        \ge
        \frac{\gammax}{2} \norm{\XiPF\eta}^2.
    \]
    Therefore, using these estimates and \eqref{eq:liplemma-nullspace-cond}%kv-bound}
    in \eqref{eq:av-case2-est1}, we obtain for $\mu < 1$ the estimate
    \begin{equation}
        %\label{eq:av-initial-below-bound}
        \notag
        \norm{A_V z}^2
        \ge
        \bigl((1-\mu)a-\gammax^2\bigr)\norm{\zeta}^2
        + \bigl((1-1/\mu)M^2 + \gammax\gammat +\gammax^2\bigr)\norm{\XiPF\eta}^2.
    \end{equation}
    We require
    \[
        (1-\mu)a-\gammax^2 > 0
        \quad
        \text{and}
        \quad
        (1-1/\mu)M^2 + \gammax\gammat +\gammax^2 > 0.
    \]
    Solving for $\gammax$, we get the conditions
    \[
        \gammax < g_1(\mu) \defeq \sqrt{(1-\mu)a}
    \]
    and
    \[
        \gammax > g_2(\mu) \defeq \frac{-\gammat + \sqrt{\gammat^2 + 4(1/\mu-1)M^2}}{2}.
    \]
    We have $g_1(1)=g_2(1)=0$. Further $g_1(\mu)/g_2(\mu) \to \infty$ as $\mu \upto 1$
    (L'H{\^o}pital).
    Thus there exists $\mu < 1 $ and $\gammax$ 
    satisfying $g_2(\mu) < \gammax < g_1(\mu)$,
    and dependent only on $M$, $\gammat$, and $a$.
    The existence of $c > 0$ satisfying
    \[
        \norm{A_V z}^2
        \ge
        c^2\norm{\zeta}^2
        +
        c^2\norm{\XiPF\eta}^2
    \]
    follows. Since $\XiPG=I$, the proof is finished.
\end{proof}

\subsection{Regularisation functionals with $L^1$-type norms}
\label{sec:l1reg}

\def\alphaj{\alpha}
\def\baralpha{\bar\alpha}
\def\Bi{{\B(0, \alphaj)}}
\def\Balpha{{\B(0, \alpha)}}

Writing $Y=Y_\NL \times Y_\LIN$ and $y=(\yNL, \yLIN)$, 
we now restrict our attention to
\begin{equation}
    \label{eq:nl-l-split}
    F^*(y)=(F_\NL^*(\yNL), F_{\LIN,\alpha}^*(\yLIN)) 
    \quad
    \text{and}
    \quad
    K(x) = (T(x), \KL x),
\end{equation}
where the operator $\KL: X \to Y_\LIN$ is linear,
$\yLIN = (\yLIN_1, \ldots, \yLIN_N) \in \prod_{i=1}^N \R^{m_j}$, and
\begin{equation}
    \label{eq:flin-l1}
    F_{\LIN,\alpha}^*(\yLIN)=\sum_{i=1}^N
         \left(
         \delta_{\Bi}(\yLIN_j)
          + \frac{\gammalin}{2\alphaj} \norm{\yLIN_j}^2
         \right).
\end{equation}
for some $\gammalin \ge 0$ and $\alphaj > 0$.
Here $\B(0, \alphaj) \subset \R^{m_j}$ are closed Euclidean unit balls of radius $\alphaj$.
We assume that $T \in C^2(X; Y_\NL)$,
and the functionals $G \in C^2(X)$ and $F_\NL^* \in C^2(Y_\NL)$
to be strongly convex, however also allowing $G=0$.

\begin{comment}
Expanding the optimality conditions \eqref{eq:nl-oc}, we see that in this
case the optimality system is
\begin{subequations}
\label{eq:l1-oc}
\begin{align}
    %\notag
    -[\grad T(\realoptx)]^* \realopt\yNL - \KL^* \realopt\yLIN & = \grad G(\realoptx), \\
    %\notag
    T(\realoptx) & = \grad F_\NL^*(\realopt\yNL), 
    \quad \text{and}\\
    %\notag
    \KL \realoptx & \in \subdiff F_{\LIN,\alpha}(\realopt\yLIN).
\end{align}
%The final condition further expands as
%\begin{equation}
%    [\KL \realoptx]_j - \gamma\yLIN_j/\alphaj  \in N_{\B(0, \alphaj)}(\yLIN_j),
%    \quad (i=1,\ldots,N).
%\end{equation}
%\end{subequations}
\end{comment}

\begin{example}[Total variation, $\TV$]
    \label{example:tv}
    Let $\Omega_d = \{0, \ldots, n_1-1\} \times \{0, \ldots, n_2-1\}$
    be a discrete domain, $f \in \R^m$, and
    $T: (\Omega_d \to \R) \to \R^m$ a possibly non-linear forward operator.
    We are interested in the total variation regularised reconstruction problem
    \begin{equation}
        \label{eq:gen-tv-regul}
        \min_v \frac{1}{2}\norm{f-T(v)}^2 + \alpha \TV(v).
    \end{equation}
    Denoting by
    \[
        \grad_d: (\Omega_d \to \R) \to (\Omega_d \to \R^2)
    \]
    a discrete gradient operator, we define the discrete total variation by
    \[
        %\TV_\gamma(u)
        \TV(v)
        \defeq \norm{\grad_d u}_{L^1(\Omega_d)}.
    \]
    We may then write
    \[
        \alpha \TV(v) = \max_\varphi \iprod{\KL u}{\varphi} - F_{\LIN,\alpha}^*(\varphi),
    \]
    where $\varphi: \Omega_d \to \R^2$, $\KL=\grad_d$, and
    \[
        F_{\LIN,\alpha}^*(\varphi)=\sum_{i=1}^{n_1} \sum_{j=1}^{n_2}
            %\left(
                \delta_{\Balpha}(\varphi(i,j))
                %+\frac{\gamma}{2\alpha} \norm{\varphi_{i,j}}^2
            %\right)
            .
    \]
    A non-zero $\gammalin > 0$ in \eqref{eq:flin-l1} corresponds to Huber-regularisation
    of the $L^1$ norm, that is to the functional
    \[
        \TV_{\gammalin}(v)
        = \sum_{i=1}^{n_1} \sum_{j=1}^{n_2} \abs{\grad_d u(i, j)}_{\gammalin},
    \]
    where
    \[
        \abs{g}_\gamma = 
        \begin{cases}
            \norm{g} - \frac{\gamma}{2}, & \norm{g} \ge \gamma,
            \\
            \frac{1}{2\gamma}\norm{g}^2, & \norm{g} < \gamma.
        \end{cases}
    \]
\end{example}

\begin{example}[Second-order total generalised variation, $\TGV^2$]
    \label{example:tgv2}
    Let us now replace $\TV$ by $\TGV^2$ in \eqref{eq:gen-tv-regul}.
    This regularisation functional was introduced in \cite{bredies2009tgv}
    as a higher-order extension of $\TV$ that tends to avoid the stair-casing
    effect while still preserving edges.
    Parametrised by $\alphavec=(\beta,\alpha)$,
    it can according to \cite{sampta2011tgv}, see also \cite{l1tgv},
    be written as the differentiation cascade
    \[
        \TGV^2_\alphavec(v) =
        %\min_{w \in L^2(\Omega; \R^m)} \alpha \norm{Dv-w}_{\Meas(\Omega; \R^m)} + \beta\norm{Ew}_{\Meas(\Omega; \R^{m \times m})},
        \min_w \alpha \norm{Dv-w} + \beta\norm{Ew}.
    \]
    for $E$ the symmetrised gradient.
    The parameter $\alpha$ is the conventional regularisation parameter, whereas the ratio
    $\beta/\alpha$ controls the smoothness of the solution. With a large ratio, one obtains results
    similar to TV, but as the ratio becomes smaller, $\TGV^2$ better reconstructs smooth
    features of the image.

    In the discrete setting, on the two-dimensional domain
    $\Omega_d = \{1, \ldots, n_1\} \times \{1, \ldots, n_2\}$, we may write
    \[
        %\TGV^2_{\alphavec,\gamma}(v) 
        \TGV^2_{\alphavec}(v) 
        = \min_w \max_{\varphi,\psi} \iprod{\KL (v, w)}{(\varphi, \psi)} - F_{\LIN,\alpha}^*(\varphi, \psi),
    \]
    where $w, \varphi: \Omega_d \to \R^2$, $\psi: \Omega_d \to \R^{2 \times 2}$ and
    \begin{equation}
        \label{eq:tgv-flin}
        F_{\LIN,\alpha}^*(\varphi)=\sum_{i=1}^{n_1} \sum_{j=1}^{n_2} 
                \delta_{\Balpha}(\varphi_{i,j})
            + \sum_{i=1}^{n_1} \sum_{j=1}^{n_2}
                %\delta_{\Bbeta}(\psi_{i,j}).
                \delta_{\Balpha}(\psi_{i,j}).
    \end{equation}
    The operator $\KL$ is defined by
    \[
        %\KL(v, w)=(\grad_d u - w, E_d w),
        \KL(v, w)=(\grad_d u - w, (\beta/\alpha)E_d w),
        \quad
        E_d w = (\grad_d' w+(\grad_d' w)^T)/2,
    \]
    for $\grad_d': (\Omega_d \to \R^2) \to (\Omega_d \to \R^{2 \times 2})$ another discrete gradient
    operator. Instead of the ratio $\beta/\alpha$ in front of $E_d$, we
    could remove this and set the radii in \eqref{eq:tgv-flin}
    for $\psi_{i,j}$ to $\beta$. The present approach however
    makes the forthcoming analysis simpler.
\end{example}

We now wish to apply Lemma \ref{lemma:liplemma} to obtain bounds
on the Lipschitz factor of $\inv H_\realoptx$ at a solution 
$\realoptu=(\realoptx,\realopt\yNL,\realopt\yLIN)$ to $0 \in H_\realoptx(\realoptu)$.
From \eqref{eq:dh-expand}, we see that we have to
calculate $D(\subdiff F_{\LIN,\alpha}^*)(\realopt\yLIN|\KL \realoptx)$.
We begin by calculating $D(\subdiff \delta_\Balpha)$ and studying conditions
for the graphical regularity of $\subdiff \delta_\Balpha=N_{\Balpha}$.

\begin{lemma}
    \label{lemma:DsubdiffF}
    Let
    $f(y)=\delta_\Balpha(y)$, %+\gamma/(2\beta)\norm{y}^2$, 
    ($y \in \R^m$),
    and pick $v \in \subdiff f(y)$. Then
    \begin{equation}
        \label{eq:d-subdiff-indicator}
        %D(\subdiff f)(y|v+(\gamma/\alpha) y)(w) = 
        D(\subdiff f)(y|v)(w) = 
        \begin{cases}
            %(\norm{v} +\gamma) w / \alpha + \R y, 
            \norm{v} w / \alpha + \R y, 
             & \norm{y}=\alpha, \norm{v}>0, \iprod{y}{w}=0, \\
            %\gamma w/\alpha
             [0, \infty) y, 
             & \norm{y}=\alpha, \norm{v}=0, \iprod{y}{w} \le 0,  \\
            %\gamma w/\alpha
            0, & \norm{y} < \alpha \\
            \emptyset, & \text{otherwise}.
        \end{cases}
    \end{equation}
    Moreover $\subdiff f$ is graphically regular and $\graph \subdiff f$ is locally closed 
    at $(y, v)$ for $v \in \subdiff f(y)$ whenever
    either $\norm{v}>0$ or $\norm{y} < \alpha$.
\end{lemma}

\begin{remark}
    \label{rem:strict-compl}
    The condition in the final statement can be seen as a form of \emph{strict complementarity},
    commonly found in the context of primal-dual interior point methods
    \cite{wright1997primal}.
\end{remark}

\begin{proof}
    %Since $g(y) \defeq \gamma/(2\alpha)\norm{y}^2$ is twice continuously differentiable
    %with $\grad^2 g(y) w =\gamma w/\alpha$, by sum rules for graphical differentiation
    %\cite[10.43]{rockafellar-wets-va}, it suffices to consider the case $\gamma=0$.
    We first show graphical regularity and local closedness assuming
    \eqref{eq:d-subdiff-indicator}. Indeed, local closedness of $\subdiff f$
    is a direct consequence of $f$ being convex and lower semicontinuous % $\B(0, \alpha)$ closed
    \cite{rockafellar-convex-analysis}.
    With regard to graphical regularity, let $(y^i, v^i) \to (y, v)$.
    Then for large enough $i$, $\norm{v^i} > 0$. This forces $\norm{y^i}=\alpha$,
    because $v^i \not\in \{0\} = \subdiff f(y^i)$ if $\norm{y^i}<\alpha$.
    Consequently we get from \eqref{eq:d-subdiff-indicator} the expression
    \[
        D(\subdiff f)(y^i|v^i)(w^i) = \norm{v^i} w^i / \alpha + \R y^i
    \]
    for any large enough index $i$ and any $w^i$ with $\iprod{y^i}{w^i}=0$.
    Choosing $w$ with $\iprod{y}{w}=0$, and $z \in \norm{v} w / \alpha + \R y$,
    we can find $w^i \to w$ and $z^i \to z$ with $\iprod{y^i}{w^i}=0$
    and $z^i \in \norm{v^i} w^i / \alpha + \R y^i$. It is now immediate that
    \[
        \liminf_{i \to\infty} \graph D(\subdiff f)(y^i|v^i) \supset \graph D(\subdiff f)(y|v).
    \]
    The inclusion in the other direction is obvious from the definitions.
    This proves that $\regD(\subdiff f)(y|v) = D(\subdiff f)(y|v)$, i.e., graphical 
    regularity in the case $\norm{v}>0$. The case $\norm{y}<\alpha$ is trivial, because
    $\subdiff f(y^i)=\{0\}$ for any $y^i$ close enough to $y$.

    Let us now prove \eqref{eq:d-subdiff-indicator}.
    We do this by calculating the second-order subgradient $\mathrm{d}^2 f(y|v)$.
    In the present situation, writing
    \[
         C \defeq \Balpha = \{ x \in \R^m \mid g(y) \in D \},
         \quad
         D \defeq (-\infty, \alpha^2/2],
         \quad
         g(y)=\norm{y}^2/2,
    \]
    the latter is given by \cite[13.17]{rockafellar-wets-va} as
    \begin{equation}
        \label{eq:Dsubdiff-expr}
        \mathrm{d}^2 f(y|v)(w)=\delta_{K(y, v)}(w)
            + \max_{x \in X(y, v)} \iprod{w}{x \grad^2 g(y)w},
    \end{equation}
    provided the following constraint qualification is satisfied:
    \begin{equation}
        \label{eq:Dsubdiff-CQ}
        x \in N_D(g(y)),\ x\grad g(y)=0 \implies x=0.
    \end{equation}
    Here
    \[
        K(y, v) \defeq \{w \in \R^m \mid \grad g(y) w \in T_D(g(y)), \iprod{w}{v}=0\}
    \]
    is the normal cone to $N_C(y)$ at $v$, and
    \[
        X(y, v) \defeq \{
            x \in N_D(g(y)) \mid v - x \grad g(y) = 0 \}.
    \]
    
    The constraint qualification \eqref{eq:Dsubdiff-CQ} is trivially
    satisfied: if $0 \ne x \in N_D(g(y))$, then necessarily $\norm{y}=\alpha$,
    so that $x\grad g(y) = xy \ne 0$.
    We may thus proceed to expanding \eqref{eq:Dsubdiff-expr}.
    %We calculate $\grad g(y)=y$, $\grad^2 g(y)=I$, and
    %\[
    %    N_C(y)=
    %    \begin{cases}
    %        [0, \infty) y, & \norm{y} = \alpha, \\
    %        \{0\}, & \norm{y} < \alpha, \\
    %        \emptyset, & \text{otherwise}.
    %    \end{cases}
    %\]
    We find for $v \in N_C(y)$ that
    \[
        K(y, v) =
        \begin{cases}
            \{w \in \R^m \mid \iprod{y}{w} = 0 \}, & \norm{y}=\alpha, \norm{v}>0  \\
            \{w \in \R^m \mid \iprod{y}{w} \le 0 \}, & \norm{y}=\alpha, \norm{v}=0  \\
            \R^m, &  \norm{y} < \alpha,
        \end{cases}
    \]
    and
    \[
        X(y, v)
        =
        \norm{v}/\alpha.
        %\begin{cases}
        %    \norm{v}/\alpha, & \norm{y}=\alpha, \\
        %    0, & \norm{y}<\alpha.
        %\end{cases}
    \]
    Thus for $v \in N_C(y)$ we have
    $\max_{x \in X(y, v)} \iprod{w}{x \grad^2 g(y)w}=\norm{v} \norm{w}^2/\alpha$.
    It follows from \eqref{eq:Dsubdiff-expr} that
    \begin{equation}
        \label{eq:Dsubdiff-expand}
        %\begin{split}
        \mathrm{d}^2 f(y|v)(w)
        %&
        %=
        %\delta_{K(y, v)}(w) + \norm{v} \norm{w}^2/\alpha
        %\\
        %&
        =
        \begin{cases}
            \norm{v} \norm{w}^2/\alpha, & \norm{y}=\alpha, \norm{v}>0, \iprod{y}{w} = 0, \\
            0, & \norm{y}=\alpha, \norm{v}=0, \iprod{y}{w} \le 0, \\
            0, & \norm{y}<\alpha. \\
            \emptyset, & \text{otherwise}.
        \end{cases}
        %\end{split}
    \end{equation}
    
    We still have to calculate $D(\subdiff f)(y|v)$. 
    Since $f$ is proper, convex, and lower semi-continuous, it is also prox-regular 
    and subdifferentiably continuous \cite[13.30]{rockafellar-wets-va} in the senses
    defined in \cite{rockafellar-wets-va}, that we introduce here by name just for
    the sake of binding various results from that book rigorously together.
    By \cite[13.17]{rockafellar-wets-va}, $f$ is also twice epi-differentiable. 
    It therefore follows from \cite[13.40]{rockafellar-wets-va} 
    that
    \[
        D(\subdiff f)(y|v) = \subdiff h
        \quad\text{for}\quad
        h \defeq \frac{1}{2}\mathrm{d}^2 f(y|v).
    \]
    Applying this to \eqref{eq:Dsubdiff-expand}, we easily calculate
    \eqref{eq:d-subdiff-indicator}.
    %\[
    %    \subdiff h(w)
    %    =
    %    \begin{cases}
    %       \norm{v} w / \alpha + \R y, & \norm{y}=\alpha, \norm{v}>0, \iprod{y}{w}=0, \\
    %        [0, \infty) y, & \norm{y}=\alpha, \norm{v}=0, \iprod{y}{w} \le 0,  \\
    %        0, & \norm{y} < \alpha \\
    %        \emptyset, & \text{otherwise}.
    %    \end{cases}
    %\]
    %This proves the claim.
\end{proof}

\begin{lemma}
    \label{lemma:l1-properties}
    If $\yLIN=(\yLIN_1, \ldots, \yLIN_N) \in Y_\LIN$ and 
    $v = (v_1, \ldots, v_N) \in \subdiff F_{\LIN,\alpha}(\yLIN)$
    %$v_j \in N_{\Bi}(\yLIN_j)$
    satisfy
    \begin{equation}
        \label{eq:l1-properties-lincond}
        \text{either } \norm{\yLIN_j} < \alphaj\ \text{ or }\ \norm{v_j-\gamma \yLIN_j/\alphaj} > 0, \quad (j=1,\ldots,N),
    \end{equation}
    then $\subdiff F_{\LIN,\alpha}^*$ is graphically regular 
    and $\graph \subdiff F_{\LIN,\alpha}^*$ locally closed at $(\yLIN, v)$, 
    and
    \[
        D(\subdiff F_{\LIN,\alpha})(\yLIN|v)(w)=
        \begin{cases}
            A_{\LIN,\alpha}(\yLIN, v) w + V_{\LIN,\alpha}(\yLIN)^\perp, & w \in V_{\LIN,\alpha}(\yLIN), \\
            \emptyset, & w \not \in V_{\LIN,\alpha}(\yLIN).
        \end{cases}
    \]
    Here $A_{\LIN,\alpha}(\yLIN, v) \defeq (A'_{\alphaj}(\yLIN_1, v_1), \ldots, A'_{\alphaj}(\yLIN_N, v_N))$ 
    and $V_{\LIN,\alpha}(\yLIN) \defeq V_{\alphaj}^{m_i}(\yLIN_1) \times \cdots \times V_{\alphaj}^{m_N}(\yLIN_N)$ 
    with
    \[
        A'_{\alphaj}(\yLIN_j, v_j) w_j \defeq 
        \begin{cases}
            \frac{\norm{v_j-\gamma \yLIN_j/\alphaj}+\gammalin}{\alphaj} w_j, & \norm{\yLIN_j}=\alphaj, \\
            \frac{\gammalin}{\alphaj}w_j, & \norm{\yLIN_j} < \alphaj,
        \end{cases}
        \quad
        \text{and}
        \quad
        V_{\alphaj}^{m_j}(\yLIN_j) \defeq
        \begin{cases}
            (\R \yLIN_j)^\perp & \norm{\yLIN_j}=\alphaj, \\
            \R^{m_j} , & \norm{\yLIN_j} < \alphaj.
        \end{cases}
    \]
\end{lemma}
\begin{proof}
    Let $g_j(\yLIN_j) \defeq \gamma/(2\alphaj)\norm{\yLIN_j}^2$.
    Then $F_{\LIN,\alpha}^*(\yLIN) = \sum_{i=1}^N f_j(\yLIN_j) + g_j(\yLIN_j)$,
    for $f_j=f$ as in Lemma \ref{lemma:DsubdiffF}.
    % for $\beta=\alphaj$.
    It follows that
    \[
        D(\subdiff F_{\LIN,\alpha}^*)(\yLIN|v)(w) = \prod_{i=1}^N D(\subdiff (f_j+g_j))(\yLIN_j|v_j)(w_j),
    \]
    and that $\subdiff F_{\LIN,\alpha}$ is graphically regular and the graph locally closed
    if $\subdiff (f_j+g_j)$ satisfies the same for each $i=1,\ldots,N$. 
    By sum rules for graphical differentiation \cite[10.43]{rockafellar-wets-va}, we have
    \[
        D(\subdiff (f_j+g_j))(\yLIN_j|v_j)(w_j)
        =
        D(\subdiff f_j)(\yLIN_j|v_j-\grad g_j(\yLIN_j))(w_j)
        + \grad^2 g(\yLIN_j).
    \]
    Moreover, since $g$ is twice continuously differentiable,
    $\subdiff (f_j+g_j)$ is graphically regular if $\subdiff f_j$ is.
    By Lemma \ref{lemma:DsubdiffF}, this is the case if 
    $\norm{\yLIN_j} < \alphaj$, or $\norm{v_j-\grad g_j(\yLIN_j)} > 0$.
    Minding that $\grad g_j(\yLIN_j)=\gamma \yLIN_j/\alphaj$,
    referring to Lemma \ref{lemma:DsubdiffF} once again for
    the expression of $D(\subdiff f_j)(\yLIN_j|v_j-\gamma\yLIN_j/\alphaj)(w_j)$,
    the claim of the present lemma follows.
\end{proof}

We now have the necessary results to bound $\lip{\inv H_\realoptx}(0|\realoptu)$.

\begin{proposition}
    \label{prop:aubin-l1}
    Suppose $F_\NL^*$ is twice continuously differentiable and strongly convex,
    and $G=0$ or $G$ is twice continuously differentiable and strongly convex.
    Define $F^*$ and $K$ by \eqref{eq:nl-l-split} for some $\alpha>0$.
    Suppose $0 \in H_\realoptx(\realoptu)$.
    If $\realopt\yLIN=(\realopt\yLIN_1, \ldots, \realopt\yLIN_N)$ and $\realoptx$ satisfy
    for some $a, b>0$ the conditions
    %satisfies $\norm{\realopt\yLIN_j} < \alphaj$ or
    %$\norm{[\KL \realoptx]_j - \gamma \realopt\yLIN_j/\alphaj} > 0$, ($i=1,\ldots,N$),
    %and we have
    \begin{equation}
        \label{eq:aubin-l1-lincond}
        %\gammalin + \min_j \norm{[\KL\realoptx]_j} > 0, 
        \gammalin(1-\norm{\realopt\yLIN_j}/\alphaj)
        +
        \norm{[\KL\realoptx]_j-\gamma\realopt\yLIN_j/\alphaj}
        > b,
        \quad
        (j=1,\ldots,N),
    \end{equation}
    and 
    \begin{equation}
        \label{eq:aubin-l1-nondeg}
        %P_{V_F} K_\realoptx \zeta = 0 \implies \zeta=0,
        %\norm{P_{V_F} K_\realoptx \zeta} \ge a \norm{\zeta},
        %\norm{(P_{V_{\LIN,\alpha}(\realopt\yLIN)} K_\LIN \zeta, \grad K_\NL(\realoptx) \zeta)} \ge a \norm{\zeta},
        \norm{\grad K_\NL(\realoptx) \zeta}^2 + \norm{P_{V_{\LIN,\alpha}(\realopt\yLIN)} \KL \zeta}^2 \ge a^2 \norm{\zeta}^2.
    \end{equation}
    then $\inv H_\realoptx$ has the Aubin property at 0 for $\realoptu$.
    In particular, given $\baralpha > \alphaj$, ($j=1,\ldots,N$),
    there exists a constant $c=c(a, b+\gamma, \Rad, \baralpha)>0$ such that
    \begin{equation}
        \label{eq:lip-constraint-structure}
        \lip{\inv H_\realoptx}(0|\realoptu) \le \inv c.
    \end{equation}
\end{proposition}

\begin{proof}
    We calculate $D(\subdiff F_{\LIN,\alpha}^*)(\realopt\yLIN|\KL \realoptx)$
    using Lemma \ref{lemma:l1-properties}, whose condition \eqref{eq:l1-properties-lincond} 
    at $v=\KL \realoptx$ is guaranteed by \eqref{eq:aubin-l1-lincond}.
    Then we use \eqref{eq:dh-expand}, \eqref{eq:dh-form} to obtain
    \begin{align}
        \notag
        D H_{\realoptx}(\realoptu|0)(v) & =
        \begin{cases}
            A v + V^\perp, & v \in V, \\
            \emptyset, & v \not\in V,
        \end{cases}
        &
        A & =
        \begin{pmatrix}
            \barG & K_{\realoptx}^* \\
            K_{\realoptx} & \barF
        \end{pmatrix},
    \end{align}
    where
    \begin{align}
        \notag
        \barG
        & =
        \grad^2 G(\realoptx)
        &
        V & = X \times Y_\NL \times V_{\LIN,\alpha}(\realopt\yLIN),
        \\
        \notag
        \barF
        & =
        \begin{pmatrix}
            \grad^2 F_\NL^*(\realopt\yNL) & 0 \\
            0 & %\inv\beta 
                A_{\LIN,\alpha}(\realopt\yLIN, \KL \realoptx)
        \end{pmatrix}
        %\quad
        %\text {and}
        %\quad
        &
        K_\realoptx
        & =
        \begin{pmatrix}
            \grad T(\realoptx) & 0 \\
            0 & \KL
        \end{pmatrix}.
    \end{align}
    Here $A_{\LIN,\alpha}$ and $V_{\LIN,\alpha}$ are given by Lemma \ref{prop:aubin-l1}.
    The lemma together with \eqref{eq:aubin-l1-lincond} also
    show that $\subdiff F_\LIN$ is graphically regular and $\graph \subdiff F_\LIN$ 
    is locally closed at $(\realopt\yLIN, \KL \realoptx)$.
    From our standing assumptions, both $F^*$ and $G$ are convex and lower semicontinuous. 
    Therefore $\graph \subdiff F$ and $\graph \subdiff G$ are locally closed.
    Moreover $\subdiff F_\NL^*$ and $\subdiff G$ are graphically 
    regular, $F_\NL^*$ and $G_\NL$ being twice continuously differentiable. 
    It follows that $H_\realoptx$ is graphically regular and $\graph H_\realoptx$ 
    is locally closed at $(0, \realoptu)$.
    %Moreover, \eqref{eq:aubin-l1-lincond} guarantees $A_{\LIN,\alpha} > 0$.
    
    Let us define
    \[
        \Xi_\alpha(x,\yNL,\yLIN)=(x, \Xi_{\alpha,F}(\yNL,\yLIN))
        \quad
        \text{with}
        \quad
        \Xi_{\alpha,F}(\yNL, \yLIN)=(\yNL, \yLIN_1/\alpha, \ldots, \yLIN_N/\alpha).
    \]
    The operator $\Xi_{\alpha,F}$ commutes with $\barF$ and $P_{V_F}$. Moreover,
    using \eqref{eq:aubin-l1-lincond} and the strong convexity of $F_\NL^*$, 
    we see that there exist $M=M(\Rad)>0$ and $\gammat=\gammat(\baralpha,b+\gamma)$ such that
    \[
        M \Xi_{\alpha,F} \ge \barF \ge \gammat \Xi_{\alpha,F}.
    \]
    The dependence of $M$ on $\Rad$ comes through $\sup_{\norm{\yNL}\le\Rad} \norm{\grad^2 F_\NL^*(\yNL)}$.
    Since commutativity with $G=0$ is automatic, and
    \eqref{eq:aubin-l1-nondeg} guarantees
    \eqref{eq:liplemma-nullspace-cond}, we may therefore apply 
    Lemma \ref{lemma:liplemma} to derive the bound
    \begin{equation}
        %\label{eq:a-c1}
        \notag
        \norm{P_V A^* P_V z} \ge c \norm{\Xi P_V z}, \quad (z \in X \times Y).
    \end{equation}
    Here the constant $c=c(M,\gammat,a)=c(a,b+\gamma,\Rad,\baralpha)$.
    Recalling \eqref{eq:aubin-approx-a}, this proves \eqref{eq:lip-constraint-structure}.
\end{proof}

\begin{remark}[Huber regularisation]
    Condition \eqref{eq:aubin-l1-lincond} is difficult to satisfy
    without Huber regularisation.
    Indeed, with $\gamma=0$, the condition becomes $\norm{[\KL\realoptx]_j} > 0$ 
    for each $i=1,\ldots,N$.
    In case of total variation regularisation in Example \ref{example:tv},
    this says that we cannot have $[\grad_d \realoptx]_j=0$; there can be no
    flat areas in the image $x$. This kind of requirement is, however,
    almost to be expected: The dual variable $\realopt\yLIN_j$, solving
    \[
        \max_{\yLIN_j \in \B(0, \alphaj)} \iprod{\yLIN_j}{[\KL\realoptx]_j}
    \]
    is not uniquely defined. Any small perturbation of $x$ can send it
    anywhere on the boundary $\BD \B(0, \alphaj)$.
    
    This oscillation is avoided by Huber regularisation., i.e., $\gamma > 0$.
    In this case optimal $\realopt\yLIN_j$ for $\realoptx$ solves
    \[
        \max_{\yLIN_j \in \B(0, \alphaj)}
        \iprod{\yLIN_j}{[\KL \realoptx]_j} + \gamma\norm{\yLIN_j}^2/(2\alphaj),
    \]
    If $\norm{[\KL\realoptx]_j}<\gamma$, necessarily $\norm{\realopt\yLIN_j} < \alphaj$.
    Clearly \eqref{eq:aubin-l1-lincond} follows.
    If, on the other hand, $\norm{[\KL\realoptx]_j}>\gamma$, necessarily
    $\norm{[\KL\realoptx]_j-\gamma\realopt\yLIN_j/\alphaj} > 0$.
    Thus \eqref{eq:aubin-l1-lincond} holds again.
    
    Even with $\gamma>0$, we do however have a problem when $\norm{[\KL\realoptx]_j}=\gamma$:
    the solution is not necessarily strictly complementary in the sense
    \eqref{eq:l1-properties-lincond}.
    A way to avoid this theoretical problem would be to replace the 2-norm cost
    in \eqref{eq:flin-l1} by a barrier function of $\Bi$. This would, however, cause
    the resolvent $\inv{(I+\sigma \subdiff F^*)}(y)$ to become very expensive to calculate.
    We therefore do not advise this in practise.
\end{remark}

\subsection{Squared $L^2$ cost functional with $L^1$-type regularisation}
\label{sec:l2l1reg}

We now make further assumptions on our problem, and limit ourselves
to reformulations of
\begin{equation}
    \label{eq:l22-l1}
    \min_x \frac{1}{2}\norm{f-T(x)}^2 + \alpha R(x).
\end{equation}
Here we assume that $T \in C^2(X; \R^m)$, and that the regularisation 
term $\alpha R(x)$ be for some linear operator $\KL: X \to Y_\LIN$ 
and $F^*_{\LIN,\alpha}$ as in \eqref{eq:flin-l1}, be written in the form
\[
    \alpha R(x)=\max_{\yLIN} \iprod{\yLIN}{\KL x} - F^*_{\LIN,\alpha}(\yLIN).
\]
This covers in particular Example \ref{example:tv} and Example \ref{example:tgv2}.

Setting
\[
    F_\NL^*(\yNL) \defeq \frac{1}{2}\norm{\yNL}^2 + \iprod{f}{\yNL}
\]
we may write
\[
    \frac{1}{2}\norm{f-T(x)}^2
    =\max_{\yNL}\ 
        \iprod{\yNL}{T(x)} - 
        F_\NL^*(\yNL).
        %\iprod{\yNL}{f} - \frac{1}{2}\norm{\yNL}^2\bigr).
\]
Observe that $F_\NL^*$ is strongly convex.
Thus with $y=(\yNL, \yLIN)$, we may reformulate \eqref{eq:l22-l1} 
in the saddle point form
\begin{equation}
    \label{eq:l22-l1-saddle}
    \min_x \max_y G(x) + \iprod{K(x)}{y} - F^*(y)
\end{equation}
for
\[
    G=0, \quad K(x)=(T(x), \KL x), \text{ and} \quad F^*(y)=(F_\NL^*(\yNL), F_{\LIN,\alpha}^*(\yLIN)).
\]

The optimality conditions \eqref{eq:nl-oc} presently expand to
$\realoptu_\alpha=(\realoptx_\alpha, \realopt\yNL_\alpha, \realopt\yLIN_\alpha)$
satisfying
\begin{subequations}
\label{eq:l1-l2-oc}
\begin{align}
    \label{eq:l1-l2-oc-1}
    [\grad T(\realoptx_\alpha)]^* \realopt\yNL_\alpha + \KL^* \realopt\yLIN_\alpha & = 0, \\
    \label{eq:l1-l2-oc-2}
    T(\realoptx_\alpha) - f & = \realopt\yNL_\alpha,
    \quad \text{and}\\
    \label{eq:l1-l2-oc-3}
    [\KL \realoptx_\alpha]_j - \gamma\realopt\yLIN_{\alpha,j}/\alpha & \in N_{\B(0, \alpha)}(\realopt\yLIN_{\alpha,j}),
    \quad (j=1,\ldots,N).
\end{align}
\end{subequations}
Since Proposition \ref{prop:aubin-l1} only shows the
Aubin property of $\inv H_{\realoptx_\alpha}$ without any guarantees
of smallness of the Lipschitz factor, we have 
to make $\norm{\Pnl \realopty_\alpha}$ small in order to satisfy
\eqref{eq:pnl-realopty-bound} for Theorem \ref{thm:conv}.
As $\Pnl \realopty_\alpha = \realopt\yNL_\alpha$, a solution
to \eqref{eq:l1-l2-oc} necessarily satisfies
\[
    \norm{\Pnl \realopty_\alpha}=\norm{f-T(\realoptx_\alpha)}.
\]
We will discuss how to make this small after the next proposition, 
rewriting Theorem \ref{thm:conv} for the present setting.

\begin{proposition}
    \label{prop:l1-l2-conv}
    Suppose $\realoptu_\alpha=(\realoptx_\alpha,\realopt\yNL_\alpha,\realopt\yLIN_\alpha)$ 
    satisfies the optimality conditions \eqref{eq:l1-l2-oc}, and for 
    some $a,b > 0$ the strict complementarity condition
    \begin{equation}
        \label{eq:strict-compl}
        %\gammalin + \min_j \norm{[\KL\realoptx]_j} > 0, 
        \gammalin(1-\norm{\realopt\yLIN_{\alpha,j}}/\alpha)
        +
        \norm{[\KL\realoptx_\alpha]_j-\gamma\realopt\yLIN_{\alpha,j}/\alpha}
        > b,
        \quad
        (j=1,\ldots,N),
    \end{equation}
    as well as the non-degeneracy condition
    \begin{equation}
        \label{eq:nasty}
        %P_{V_F} K_\realoptx \zeta = 0 \implies \zeta=0.
        %\grad T(\realoptx) \zeta = 0 \text{ and } P_{V_\LIN} \KL \zeta  = 0 \implies \zeta =0,
        \norm{\grad T(\realoptx_\alpha) \zeta}^2 + \norm{P_{V_{\LIN,\alpha}(\realopt\yLIN_\alpha)} \KL \zeta}^2 \ge a^2 \norm{\zeta}^2.
    \end{equation}
    Suppose, moreover, that $K$ and the step lengths $\tau,\sigma>0$ satisfy \eqref{eq:ass-k}.
    Pick $\bar \alpha>\alpha$, and let $c=c(a,b+\gamma,\Rad,\bar\alpha)$ 
    be given by Proposition \ref{prop:aubin-l1}.
    Then there exists $\delta_1 > 0$ such that Algorithm \ref{algorithm:nl-cp} applied 
    to the saddle point form \eqref{eq:l22-l1-saddle} of \eqref{eq:l22-l1}
    converges provided
    \[
        \norm{u^1-\realoptu_\alpha} \le \delta_1,
    \]
    and
    \begin{equation}
        \label{eq:l2-l1-conv-constant-cond}
        \norm{f-T(\realoptx_\alpha)} < \frac{1-1/\sqrt{1+c^2/(2\MMax^4)}}
                                     {\inv c \MCond L_2}.
    \end{equation}
\end{proposition}
\begin{proof}
    %First of all, we show that \eqref{eq:ass-exist} holds. 
    %Indeed, with $\nu=(0, \lambda, 0) \in \Ynl$, \eqref{eq:ass-exist}
    %simply says that there exists a solution to the convex problem
    %\[
    %    \min_x \frac{1}{2}\norm{f-T(\bar x)+\grad T(\bar x)(x-\bar x)+\nu}^2 + \alpha R(x).
    %\]
    %This is guaranteed by standard results.
    %
    We verify the assumptions of Theorem \ref{thm:conv}.
    By Lemma \ref{lemma:discrepancy-estimate}, it only remains to prove
    \eqref{eq:pnl-realopty-bound}.
    As \eqref{eq:strict-compl} and \eqref{eq:nasty} hold, Proposition \ref{prop:aubin-l1}
    shows that $\inv H_{\realoptx_\alpha}$ has the Aubin property at $0$ for $\realoptu_\alpha$
    with $\lip{\inv H_{\realoptx_\alpha}} \le \inv c$ for some constant $c>0$.
    Condition \eqref{eq:l2-l1-conv-constant-cond} now implies \eqref{eq:pnl-realopty-bound}. 
    The claim of the present proposition thus follows from Theorem \ref{thm:conv}.
\end{proof}

\begin{remark}[Non-degeneracy condition]
    The non-degeneracy condition \eqref{eq:nasty} may also be stated
    \[
        \grad T(\realoptx_\alpha) \zeta = 0 \text{ and } P_{V_{\LIN,\alpha}(\realopt\yLIN_\alpha)} \KL \zeta  = 0 \implies \zeta =0.
    \]
    Verifying this is easy if $\grad T(\realoptx_\alpha)$ has full range, but otherwise it 
    can be quite unwieldy thanks to the projection $P_{V_{\LIN,\alpha}(\realopt\yLIN_\alpha)}$. 
    Unfortunately, we have found no way to avoid this condition
    or \eqref{eq:l2-l1-conv-constant-cond}.
\end{remark}

We conclude our theoretical study with a simple exemplary result
on the satisfaction of \eqref{eq:l2-l1-conv-constant-cond}.
The problem with simply letting $\alpha \downto 0$ in 
order to get $\norm{f-T(\realoptx_\alpha)} \to 0$ is that the constant $c$
might blow up, depending on $\realoptu_\alpha$ through $a$ and $b$.
We therefore need to study the uniform satisfaction
of these conditions. 
In order to keep the present paper at a reasonable length, 
we limit ourselves to a very simple result that assumes the
existence of a convergent sequence $\{u_\alpha\}$ of solutions
to \eqref{eq:l1-l2-oc} as $\alpha \downto 0$. 
The entire topic of the existence of such a sequence merits
an independent study involving set-valued implicit function 
theorems (e.g.~\cite{dontchev1995implicit}) 
and source conditions on the data
(cf.,~e.g.~\cite{schuster2012regularization}).
Besides the existence of the minimising sequence,
we assume the existence of $x^*$ such that $f=T(x^*)$.
It is not difficult to formulate and prove equivalent results
for the noisy case, where we only have the bound
$\inf_x \norm{f-T(x)} \le \sigma$ for small $\sigma$.

\begin{proposition}
    \label{prop:lowalpha}
    Suppose $f=T(x^*)$ for some $x^* \in X$.
    Let $\realoptu_\alpha=(\realopt x_\alpha, \realopt\yNL_\alpha, \realopt\yLIN_\alpha)$ 
    solve \eqref{eq:l1-l2-oc} for $\alpha > 0$, and suppose that
    \begin{equation}
        \label{eq:lowalpha-conv}
        (\realopt x_\alpha, \realopt\yLIN_\alpha/\alpha) \to (\bar x, \bar \yLIN),
        \quad (\alpha \downto 0).
    \end{equation}
    If $\bar x$ and $\bar\yLIN$ satisfy for some $\bar a, \bar b>0$ 
    the strict complementarity condition
    \begin{equation}
        \label{eq:strict-complX}
        %\gammalin + \min_j \norm{[\KL\realoptx]_j} > 0, 
        \gammalin(1-\norm{\bar\yLIN_j})
        +
        \norm{[\KL\bar x]_j-\gamma\bar\yLIN_j}
        > \bar b,
        \quad
        (j=1,\ldots,N),
    \end{equation}
    and the non-degeneracy condition
    \begin{equation}
        \label{eq:nastyX}
        %P_{V_F} K_\realoptx \zeta = 0 \implies \zeta=0.
        %\grad T(\bar x) \zeta = 0 \text{ and } P_{V_\LIN} \KL \zeta  = 0 \implies \zeta =0,
        \norm{\grad T(\bar x) \zeta}^2 + \norm{P_{V_{\LIN,\alpha}(\realopt\yLIN)} \KL \zeta}^2 \ge \bar a^2 \norm{\zeta}^2,
    \end{equation}
    then there exists $\alpha^* > 0$ such that
    \eqref{eq:strict-compl}--\eqref{eq:l2-l1-conv-constant-cond}
    hold for $\realoptu = \realoptu_\alpha$
    whenever $\alpha \in (0, \alpha^*)$.
    
    In consequence, there exist $\bar\alpha>0$ and $\delta_1 > 0$ such that 
    Algorithm \ref{algorithm:nl-cp} applied to the saddle point 
    form \eqref{eq:l22-l1-saddle} of \eqref{eq:l22-l1} converges 
    whenever $\alpha \in (0, \bar\alpha)$, and
    the initial iterate $u^1=(x^1, \yNL^1, \yLIN^1)$ satisfies
    \[
        \yNL^1=T(x^1)-f
        \quad
        \text{and}
        \quad
        \norm{(x^1, \yLIN^1/\alpha)-(\bar x, \bar \yLIN)} \le \delta_1.
    \]
\end{proposition}

\begin{remark}
    If each $\realoptx_\alpha$, ($\alpha>0$), solves
    the minimisation problem \eqref{eq:l22-l1}, instead of
    just the first-order optimality conditions \eqref{eq:l1-l2-oc},
    so that $\norm{f-T(\realoptx_\alpha)} \to 0$,
    then the limit $\bar x$ solves
    $$\min R(x) \quad\text{such that}\quad T(x) = f.$$
\end{remark}

\begin{proof}
    Let us begin by defining
    \[
        \Psi_{\alpha}(x, \yLIN)=(x, \yLIN_1/\alpha, \ldots, \yLIN_N/\alpha).
    \]
    Minding the renormalisation of $\yLIN$ through $\Psi_\alpha$
    in \eqref{eq:lowalpha-conv},
    the conditions \eqref{eq:strict-complX} and \eqref{eq:nastyX}
    arise in the limit from \eqref{eq:strict-compl} and \eqref{eq:nasty}.
    Indeed, if there were sequences $\alpha^i \downto 0$, $\zeta^i \to \zeta$,
    $\norm{\zeta} > 0$,
    and $\Psi_{\alpha^i}(\realoptx_{\alpha^i}, \realopt\yLIN_{\alpha^i}) \to (\bar x,\bar\yLIN)$,
    such hat \eqref{eq:strict-compl} or \eqref{eq:nasty} would not eventually hold 
    for $a=\bar a/2$ and $b=\bar b/2$, we would find a contradiction 
    to \eqref{eq:strict-complX} or \eqref{eq:nastyX}, respectively. We may thus pick
    $\alpha^{**}>0$ such that \eqref{eq:strict-compl} and \eqref{eq:nasty} hold
    for $a=\bar a/2$ and $b=\bar b/2$ and $\alpha \in (0, \alpha^{**})$.
    Denoting by $H_\alpha$ the operator $H_{\realoptx_\alpha}$
    for a specific choice of $\alpha$, Proposition \ref{prop:aubin-l1}  now shows that
    \[
        \lip{\inv H_\alpha}(0|\realopt u_\alpha) \le \inv c,
        \quad
        (\alpha \in (0, \alpha^{**})),
    \]
    for some $c=c(a,b+\gamma,\Rad,\alpha^{**})$ independent of $\alpha$ and $\realoptu_\alpha$.
    Observing that $\MCond$ and $\MMax$ are independent of $\alpha$
    (which is encoded into $F_{\LIN,\alpha}$ in our formulation), it follows that
    \eqref{eq:l2-l1-conv-constant-cond} holds if
    \[
        \norm{f-T(\realoptx_\alpha)} < \epsilon
        \defeq \frac{1-1/\sqrt{1+c^2/(2\MMax^4)}}{\inv c\MCond L_2}.
    \]
    Because $f=T(x^*)$, this can be achieved whenever $\alpha \in (0, \alpha^*)$ 
    for some $\alpha^* \in (0, \alpha^{**})$.

    Proposition \ref{prop:l1-l2-conv} now provides $\delta > 0$
    such that Algorithm \ref{algorithm:nl-cp} converges if
    \[
        \norm{\realoptu_\alpha-u^1} \le \delta.
    \]
    Setting $\yNL^1=T(x^1)-f$ and using $T \in C^2(X; \R^m)$, we have
    \[
        \norm{\realopt \yNL_\alpha-\yNL^1} 
        =
        \norm{T(\realoptx_\alpha)-T(x^1)} 
        \le
        \ell\norm{\realoptx_\alpha-x^1}
    \]
    for $\ell$ the Lipschitz factor of $T$ on a suitable 
    compact set around $\bar x$.
    Therefore, whenever $\alpha \in (0, \alpha^{*})$, we may
    estimate with $C=(1+\ell)\max\{1,\alpha^{*}\}$ that
    \[
        \begin{split}
        \norm{\realoptu_\alpha-u^1}
        &
        \le
        (1+\ell)
        \norm{(\realoptx_\alpha,\realopt\yLIN_\alpha)-(x^1,\yLIN^1)}
        \\
        &
        \le
        (1+\ell) \norm{\inv{\Psi_\alpha}} \norm{\Psi_\alpha((\realoptx_\alpha,\realopt\yLIN_\alpha)-(x^1,\yLIN^1))}
        \\
        &
        \le
        C\left(
        \norm{\Psi_\alpha(\realoptx_\alpha,\realopt\yLIN_\alpha)-(\bar x,\bar \yLIN)}
        +
        \norm{\Psi_\alpha(x^1, \yLIN^1)-(\bar x, \bar \yLIN)}\right).
        \end{split}
    \]
    By the convergence $\Psi_\alpha(\realopt x_\alpha, \realopt\yLIN_\alpha) \to (\bar x, \bar \yLIN)$,
    choosing $\bar \alpha \in (0, \alpha^{*})$ small enough, we can force
    \[
        C \norm{\Psi_\alpha(\realoptx_\alpha,\realopt\yLIN_\alpha)-(\bar x,\bar \yLIN)} \le \delta/2,
        \quad
        (\alpha \in (0, \bar \alpha)).
    \]
    Choosing $\delta_1=\delta/(2C)$, we may thus conclude the proof.
\end{proof}

\section{Applications and computational experience}
\label{sec:appl}

With convergence theoretically studied, and total variation type regularisation 
problems reformulated into the minimax form, we now move on to studying the 
numerical performance of NL-PDHGM. We do this by applying the method to two problems 
from magnetic resonance imaging: velocity imaging in Section \ref{sec:velmri},
and diffusion tensor imaging in Section \ref{sec:dti}.

\subsection{Phase reconstruction for velocity-encoded MRI}
\label{sec:velmri}

As our first application, we consider the phase reconstruction problem \eqref{eq:phase-recon}
for magnetic resonance velocity imaging. We are given a complex
sub-sampled $k$-space data $f=S\FF u^* +\nu \in \R^m$ corrupted by noise $\nu$.
Here $S$ is the sparse sampling operator, $\FF$ the discrete two-dimensional
Fourier transform, and $u^* \in L^1(\Omega_d; \C)$ the noise-free complex image in 
the discrete spatial domain $\Omega_d = \{0, \ldots, n-1\}^2$.
We seek to find a complex image $u=\rhoC \exp(i\varphi)$ approximating $u^*$.
Motivated by the discoveries in \cite{tuomov-phaserec}, we wish to regularise
$\rhoC$ and $\varphi$ separately. 
Introducing non-linearities into the reconstruction problem, 
we therefore define the forward operator
\[
    u = T(\rhoC, \varphi) \defeq S\FF[x \mapsto \rhoC(x)\exp(i\varphi(x))].
\]
For the phase $\varphi$, we choose to use second order total generalised
variation regularisation, and for the magnitude $\rhoC$, total variation
regularisation. Choosing regularisation parameters
$\alpha_\rhoC, \alpha_\varphi, \beta_\varphi > 0$ appropriate for
the data at hand, we then seek to solve
\begin{equation}
    \label{eq:vel-recon}
    \min_{\rhoC, \varphi \in L^1(\Omega_d)}
    \frac{1}{2}\norm{f-T(\rhoC, \varphi)}^2
    + \alpha_\rhoC \TV(\rhoC)
    + \TGV^2_{(\beta_\varphi, \alpha_\varphi)}(\varphi).
\end{equation}

For our experiments, due to trouble obtaining real data, we use a simple synthetic phantom
on $\Omega=[-1,1]^2$, depicted in Figure \ref{fig:velimg-res/velimgres-tau0.5sigma1.9-orig}.
It simulates the speed along the $y$ axis of a fluid rotating in a ring. 
Specifically
\[
    \rhoC(x, y)=\chi_{0.3 < \sqrt{x^2+y^2} < 0.9}(x, y),
    \quad
    \text{and}
    \quad
    \varphi(x,y)=x/\sqrt{x^2+y^2}.
\]

The image is discretised on a $n \times n$ grid with $n = 256$.
To the discrete Fourier-transformed $k$-space image, we add pointwise
Gaussian noise of standard deviation $\sigma=0.2$. For the sub-sampling operator
$S$, we choose a centrally distributed Gaussian sampling pattern with variance 
$0.15 \cdot 128$ and $15\%$ coverage of the $n \times n$ image in $k$-space.
For the regularisation parameters, which we do not claim to have chosen optimally in 
the present algorithmic paper, we choose $\alpha_\varphi=0.15 \cdot (2/n)$, $\beta_\varphi=0.20 \cdot (2/n)^2$, 
and $\alpha_\rhoC=2/n$. (The factors $2/n$ are related to spatial step 
size $1/n$ of the discrete differential on the $[-1, 1]^2$ domain. 
When the step size is omitted, i.e., implicitly taken as $1$, 
the factors disappear.)

We perform computations with Algorithm \ref{algorithm:nl-cp} (Exact NL-PDHGM),
Algorithm \ref{algorithm:nl-cp-lin} (Linearised NL-PDHGM), and the Gauss-Newton
method. As we recall, the latter is based on linearising the non-linear operator $T$ 
at the current iterate, solving the resulting convex problem, and repeating until
hopeful convergence. We solve each of the inner convex problems by PDHGM, \eqref{eq:cp-lin}.
We initialise each method with the backprojection
$(S\FF)^*f$ of the noisy sub-sampled data $f$. We use two different choices
for the PDHGM step length parameters $\sigma$ and $\tau$. The first one has equal
primal and dual parameters, both $\sigma, \tau=0.95/L$. 
Recalling Remark \ref{rem:steplength}, we update
$L = \sup_{k=1,\ldots,i} \norm{\grad K(x^k)}$ dynamically for NL-PDHGM,;
for linear PDHGM, used within Gauss-Newton, this simply reduces to $L=\norm{K}$. 
This choice of $\tau$ and $\sigma$ is somewhat na\"ive, and it has been observed 
that sometimes choosing $\sigma$ and $\tau$ in different proportions can improve the 
performance of PDHGM for linear operators \cite{pock2011iccv}. Therefore,
as our second step length parameter choice we use $\tau=0.5/L$ and $\sigma=1.9/L$.

We limit the number of PDHGM iterations (within each Gauss-Newton iteration) 
to 100000, and limit the number of Gauss-Newton iterations to 100.
As the primary stopping criterion for the NL-PDHGM methods, we use
$\norm{x^i-x^{i+1}}<\rho$ for $\rho = 1\ee^{-4}$. The same criterion is used 
to stop the outer Gauss-Newton iterations. 
As the stopping criterion of PDHGM within each Gauss-Newton iteration, we use 
the decrease of the pseudo-duality gap to less than $\rho_2 = 1\ee^{-3}=\rho/10$; 
higher accuracy than this would penalise the computational times of the
Gauss-Newton method too much in comparison to NL-PDHGM. Much lower accuracy
would almost reduce it to NL-PDHGM, if convergence would be observed at all; 
we will get back to this in our next application.

The pseudo-duality gap is discussed in detail in \cite{tuomov-dtireg}. 
We use use it to work around the fact that in reformulations of
the linearised problem into the form \eqref{eq:nonlinear-problem},
$G=0$. This causes the duality gap to be in practise infinite. 
This could be avoided if we could make $G$ to be non-zero, for example
by taking $G(x)=\delta_{B(0, M)}(x)$ for $M$ a bound on $x$. 
Often the primal variable $x$ can indeed be shown to be bounded. 
The problem is that the exact bound is not known. The idea of the pseudo-duality
gap, therefore, is to update the bound $M$ dynamically. Assuming it large
enough, it does not affect the PDHGM itself. It only affects the duality 
gap, and is updated whenever the duality gap is violated. 
For the artificially dynamically bounded problem the duality gap is
finite, and called the pseudo-duality gap. 

We perform the computations with OpenMP parallelisation on 6 cores
of an Intel Xeon E5-2630 CPU at 2.3GHz, with 64GB (that is, enough!)
random access memory available. The results are displayed 
in Table \ref{table:velimg} for the first choice of $\sigma$ and $\tau$,
and in Table \ref{table:velimg-unequal} for the second choice.
In the tables we report the number of PDHGM and Gauss-Newton iterations taken
along with the computational time in seconds, as well as the PSNR
of both the magnitude and the phase for the reconstructions.
For the calculation of the PSNR of $\varphi$, we have only included 
the ring $0.3 < \sqrt{x^2+y^2} < 0.9$, as the phase is meaningless
outside this, the magnitude being zero.
The results of the second parameter choice are also visualised in 
Figure \ref{fig:velimg}. It shows the original noise-free synthetic data, the
backprojection of the noisy sub-sampled data, and the obtained reconstructions, 
which have little difference between the methods, validating the results.
Also, because there is no difference Linearised and Exact NL-PDHGM, we
only display the results for the latter.

The main observation from Table \ref{table:velimg}, with equal $\tau$ and $\sigma$,
is that Exact NL-PDHGM, Algorithm \ref{algorithm:nl-cp}, is significantly faster
than Gauss-Newton, taking only around two minutes in comparison
to almost an hour for Gauss-Newton. Gauss-Newton nevertheless appears
to converge for the present problem, having taken 13 iterations
to each the stopping criterion. Another observation is that
although Linearised NL-PDHGM, Algorithm \ref{algorithm:nl-cp-lin},
takes exactly as many iterations as Exact NL-PDHGM, it requires
far more computational time, as the linearisation of $K$ is more
expensive to calculate than $K$ itself. Observing
Table \ref{table:velimg-unequal}, we see that choosing $\tau$ and $\sigma$,
in unequal proportion significantly improves the performance
of NL-PDHGM, taking less than a minute. In case of Gauss-Newton 
the improvement is merely marginal. Finally, studying Figure
\ref{fig:velimg}, we may verify that the method has converged
to a reasonable solution.

\begin{table}[p]%[t!]
    \caption{Phase/magnitude reconstruction using
        Exact NL-PDHGM, %(Algorithm \ref{algorithm:nl-cp}),
        Linearised NL-PDHGM, %(Algorithm \ref{algorithm:nl-cp-lin}),
        and the Gauss-Newton method: equal primal and dual step length parameters.
        For Gauss-Newton, PDHGM is used in the inner iterations,
        and the number of PDHGM iterations reported is the
        total over all Gauss-Newton iterations.
        The PSNR of $\varphi$ excludes the area outside the ring, where $\rhoC=0$.
        The stopping criterion for NL-PDHGM and Gauss-Newton is
        $\norm{x^i-x^{i+1}} < \rho$. The stopping
        criterion for linear PDHGM, used for inner Gauss-Newton iterations, 
        is pseudo-duality gap less than $\rho_2$.
        %Moreover, the maximum number %of PDHGM iterations (per Gauss-Newton iteration) is 100000.
    }
    \label{table:velimg}
    \centering
    
    $\rho=1\ee^{-4}$, $\rho_2=1\ee^{-3}$, $\tau=0.95/L$, $\sigma=0.95/L$ for $L = \sup_i \norm{\grad K(x^i)}$
    \medskip
    
    \begin{tabular}{l|rrrrr}
        Method & PDHGM iters. & GN iters. & Time & PSNR($\rhoC)$ & PSNR($\varphi$) \\
        \hline
        Backprojection & -- & -- & -- & 19.2 & 41.0\\
Exact NL-PDHGM & 15800 & -- & 129.1s & 25.5 & 52.7\\
Linearised NL-PDHGM & 15800 & -- & 171.2s & 25.5 & 52.7\\
Gauss-Newton & 137900 & 12 & 1353.7s & 25.5 & 53.9\\

    \end{tabular}
%\end{table}

    \medskip
    
%\begin{table}[t!]
    \caption{Phase/magnitude reconstruction using
        Exact NL-PDHGM, %(Algorithm \ref{algorithm:nl-cp}),
        Linearised NL-PDHGM, %(Algorithm \ref{algorithm:nl-cp-lin}),
        and the Gauss-Newton method: unequal primal and dual step length parameters.
        For Gauss-Newton, PDHGM is used in the inner iterations,
        and the number of PDHGM iterations reported is the
        total over all Gauss-Newton iterations.
        The PSNR of $\varphi$ excludes the area outside the ring, where $\rhoC=0$.
        The stopping criterion for NL-PDHGM and Gauss-Newton is
        $\norm{x^i-x^{i+1}} < \rho$. The stopping
        criterion for linear PDHGM, used for inner Gauss-Newton iterations, 
        is pseudo-duality gap less than $\rho_2$.
        %Moreover, the maximum number %of PDHGM iterations (per Gauss-Newton iteration) is 100000.
    }
    \label{table:velimg-unequal}
    \centering
    
    $\rho=1\ee^{-4}$, $\rho_2=1\ee^{-3}$, $\tau=0.5/L$, $\sigma=1.9/L$ for $L = \sup_i \norm{\grad K(x^i)}$
    \medskip
    
    \begin{tabular}{l|rrrrr}
        Method & PDHGM iters. & GN iters. & Time & PSNR($\rhoC)$ & PSNR($\varphi$) \\
        \hline
        Backprojection & -- & -- & -- & 19.2 & 41.0\\
Exact NL-PDHGM & 8200 & -- & 57.8s & 25.5 & 51.2\\
Linearised NL-PDHGM & 8200 & -- & 87.6s & 25.5 & 51.2\\
Gauss-Newton & 112145 & 12 & 1119.1s & 25.5 & 53.7\\

    \end{tabular}
\end{table}

\newlength{\resfigw}
\setlength{\resfigw}{0.23\textwidth}
\newlength{\colw}
\newlength{\w}

\newcommand{\resplotz}[3][]{%
    \setlength{\colw}{\resfigw}%
    \begin{subfigure}[t]{\colw}%
    \begin{tikzpicture}[outer sep=0pt,inner sep=0pt]%
        \pgftext[at=\pgfpoint{0}{\colw+3pt},left,bottom]{%
            \includegraphics[width=\colw]{{#2_abs}.png}
        }%
        \pgftext[at=\pgfpoint{0}{0},left,bottom]{%
            \includegraphics[width=\colw]{{#2_phase}.png}
        }%
        {#1}%
    \end{tikzpicture}%
    \caption{#3}
    \ifdefined\subfigprefix\label{\subfigprefix#2}\else\relax\fi%
    \end{subfigure}\relax%
    }

\begin{figure}[p]%[f]
    %\centerfloat
    \centering
    \def\subfigprefix{fig:velimg-}
    \begin{minipage}[t]{0.94\textwidth}
    \vspace{0pt}
    \resplotz{res/velimgres-tau0.5sigma1.9-orig}{Original}
    \resplotz{res/velimgres-tau0.5sigma1.9-backproj}{Backprojection}
    \resplotz{res/velimgres-tau0.5sigma1.9-cpock-exact}{Exact NL-PDHGM}
    %\resplotz{res/velimgres-cpock-lin}{Linearised NL-PDHGM}
    \resplotz{res/velimgres-tau0.5sigma1.9-gn}{Gauss-Newton}
    \end{minipage}%
    \hspace{2pt}%
    \begin{minipage}[t]{0.015\textwidth}
        \vspace{0pt}
        \setlength{\colw}{\resfigw}%
        \begin{tikzpicture}[outer sep=0pt,inner sep=0pt]%
            \node at (0, 2\colw+3pt) {};
            \node at (0, 0) {};
            \pgftext[at=\pgfpoint{0}{1.5\colw+3pt},left,center]{%
                \includegraphics[height=0.7\resfigw]{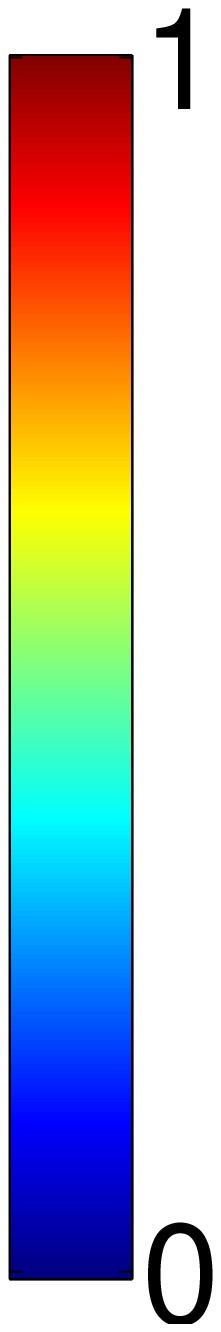}
            }%
            \pgftext[at=\pgfpoint{0}{.5\colw},left,center]{%
                \includegraphics[height=0.7\resfigw]{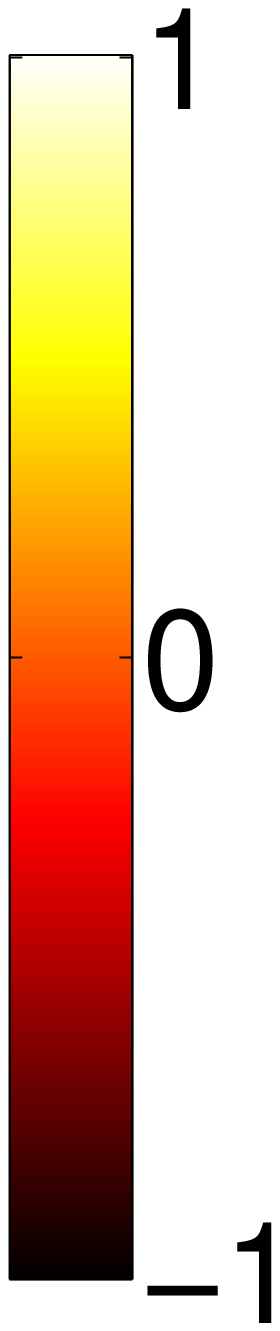}
            }%
        \end{tikzpicture}%
    \end{minipage}
    
    \caption{Non-linear phase/magnitude reconstruction using
        Exact NL-PDHGM %(Algorithm \ref{algorithm:nl-cp}),
        and the Gauss-Newton method. Also pictured is the original
        noise-free test data, and the backprojection of the noisy sub-sampled data.
        The upper image in each column is the magnitude $\rhoC$, and the lower
        image the phase $\varphi$. 
        }
    \label{fig:velimg}
\end{figure}

A full and fair comparison of the non-linear model \eqref{eq:vel-recon}
against the linear model \eqref{eq:phase-recon-lin} is out of the scope 
of the present paper. Especially, considering that we employed in 
\cite{tuomov-phaserec} Bregman iterations \cite{oshburgolxuyin} for 
contrast enhancement, without the introduction of a contrast enhancement
technique for the non-linear model as in \cite{bachmayr2009iterative},
it is not fair to directly compare against the results therein. 
Nevertheless, in order to justify the non-linear model, we have a simple 
comparison in Table \ref{table:linvelimg} and Figure \ref{fig:linvelimg}.
We computed solutions to the 
linear model for varying $\alpha$ in the interval $[0.01, 3.0]\cdot (2/n)$ 
(spacing $0.005$ below $\alpha=1.0\cdot(2/n)$, and $0.1$ above), $\beta=1.1\alpha\cdot(2/n)$, 
and chose the optimal $\hat \alpha$ by three different criteria. 
These Morozov's discrepancy principle \cite{Morozov}, i.e., the 
largest $\alpha$ such that $\norm{f-S\FF u_\alpha}$ is below the noise 
level, as well as the optimal PSNR for both the amplitude $\rhoC$ and 
phase $\varphi$. For the non-linear model, we chose the parameters
manually. We find this reasonable, because for multi-dimensional parameters 
we do not at this time have to our avail something practical like the discrepancy principle,
and a scan of the parameter space is not doable in practical applications.
Moreover, in order to show that the non-linear model improves over the linear model,
it is only necessary to pick parameters for the linear model optimally.
With this in mind, we specifically picked
$\alpha_\rhoC=1\cdot(2/n)$, $\alpha_\varphi=0.20\cdot(2/n)$, and 
$\beta_\varphi=1.1\alpha_\varphi\cdot(2/n)$ for the non-linear model.
We also increased the stopping threshold of NL-PDHGM to $\rho=1\ee^{-5}$
for better quality solutions. For the linear model, we set the pseudo-duality
gap threshold to $\rho_2=1\ee^{-4}$. This appears to be enough for the comparison,
keeping in mind that the stopping criteria are not comparable.
What we can draw from Table \ref{table:linvelimg} is that for most criteria,
the linear model fails to balance between PSNR($\varphi$) and PSNR($\rhoC$) 
as well as the non-linear model.
With the discrepancy principle as parameter choice criterion,
and $\TV$ as the regulariser, the linear model however beats 
the non-linear model in terms of both PSNRs.
Inspecting Figure \ref{fig:linvelimg-res/velimgres-lincompare_lin1}, we however observe extensive
stair-casing in both the amplitude and phase. This is avoided by the excellent reconstruction
by the non-linear model in Figure \ref{fig:linvelimg-res/velimgres-lincompare_nonlin}.
With the higher $\alpha$ in Figure \ref{fig:linvelimg-res/velimgres-lincompare_lin2},
the phase reconstruction is good even by the linear model, but the amplitude has become
smoothed out. This is also avoided by non-linear model.

%without trying to optimise $\alpha_{\varphi}$;
%cf.~Table \ref{table:velimg} or Table \ref{table:velimg-unequal}.
%The difference is not huge, 
%Although we did not optimise the linear model over the $\beta$
%parameter of $\TGV^2$, similar results hold for $\TV$,

\begin{table}[t]
    \caption{Results for the linear model \eqref{eq:phase-recon-lin} in comparison
            to the non-linear model \eqref{eq:vel-recon}.
            Optimal $\hat \alpha$ has been chosen by Morozov's
            discrepancy principle, and the PSNR for both the
            amplitude $\rhoC$ and phase $\varphi$. For each
            criterion the PSNRs and optimal $\hat \alpha$ are reported.
            In case of the non-linear model, $\TV$ regularisation is used for the amplitude,
            and $\TGV^2$ regularisation for the phase, with parameters
            $\alpha_\rhoC=1$, $\alpha_\varphi=0.20$, and $\beta_\varphi=1.1\alpha_\varphi$.
            }
    \label{table:linvelimg}
    \centering
    \begin{tabular}{llr|rrrr}
        Criterion & Regulariser & $\hat \alpha \cdot n/2$ & PSNR($\rhoC$) & PSNR($\varphi$) \\ 
        %& SSIM($\rhoC$) & SSIM($\varphi$) \\
        \hline
        %Discrepancy principle & $\TV$ & 0.25 & 26.7 & 52.0 & 0.79 & 0.98 \\ 
        %PSNR($\rhoC$) & $\TV$ & 0.20 & 26.7 & 51.1 & 0.80 & 0.98 \\ 
        %PSNR($\varphi$) & $\TV$ & 1.25 & 22.1 & 57.8 & 0.70 & 0.99 \\ 
        %\hline
        %Discrepancy principle & $\TGV^2$ & 0.25 & 22.7 & 53.0 & 0.68 & 0.96 \\ 
        %PSNR($\rhoC$) & $\TGV^2$ & 0.15 & 24.8 & 50.8 & 0.65 & 0.95 \\ 
        %PSNR($\varphi$) & $\TGV^2$ & 2.60 & 17.8 & 57.3 & 0.57 & 0.98 \\ 
        Discrepancy principle & $\TV$ & 0.25 & 26.7 & 52.0 \\ 
        PSNR($\rhoC$) & $\TV$ & 0.20 & 26.7 & 51.1 \\ 
        PSNR($\varphi$) & $\TV$ & 1.25 & 22.1 & 57.8 \\ 
        \hline
        Discrepancy principle & $\TGV^2$ & 0.25 & 22.7 & 53.0  \\ 
        PSNR($\rhoC$) & $\TGV^2$ & 0.15 & 24.8 & 50.8  \\ 
        PSNR($\varphi$) & $\TGV^2$ & 2.60 & 17.8 & 57.3  \\ 
        \hline
        Manual  (non-linear model) & Mixed & &  25.2 & 54.0 \\
    \end{tabular}
\end{table}

\begin{figure}[t]%[p]%[f]
    %\centerfloat
    \centering
    \begin{minipage}[t]{0.94\textwidth}
    \vspace{0pt}
    \def\subfigprefix{fig:linvelimg-}
    \resplotz{res/velimgres-tau0.5sigma1.9-orig}{Original}
    \resplotz{res/velimgres-lincompare_nonlin}{Non-linear}
    \resplotz{res/velimgres-lincompare_lin1}{Linear, $\TV$, $\alpha=0.25$}
    \resplotz{res/velimgres-lincompare_lin2}{Linear, $\TV$, $\alpha=1.25$}
    \end{minipage}%
    \hspace{2pt}%
    \begin{minipage}[t]{0.015\textwidth}
        \vspace{0pt}
        \setlength{\colw}{\resfigw}%
        \begin{tikzpicture}[outer sep=0pt,inner sep=0pt]%
            \node at (0, 2\colw+3pt) {};
            \node at (0, 0) {};
            \pgftext[at=\pgfpoint{0}{1.5\colw+3pt},left,center]{%
                \includegraphics[height=0.7\resfigw]{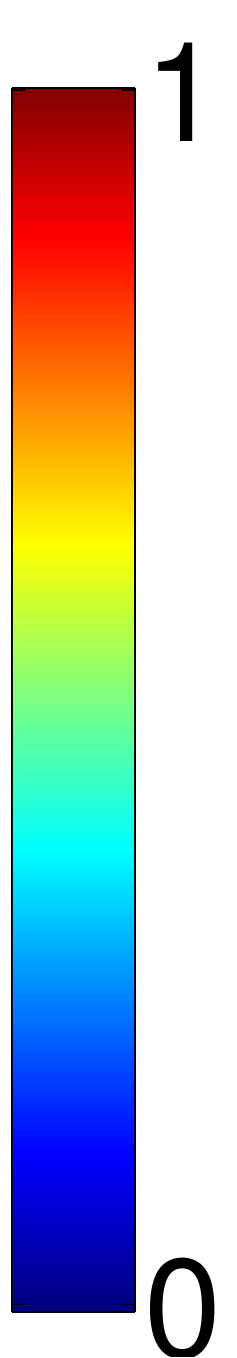}
            }%
            \pgftext[at=\pgfpoint{0}{.5\colw},left,center]{%
                \includegraphics[height=0.7\resfigw]{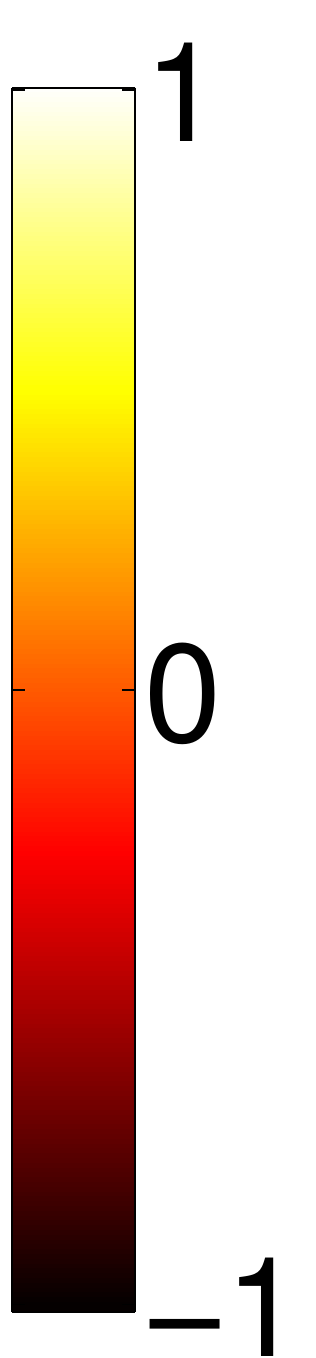}
            }%
        \end{tikzpicture}%
    \end{minipage}
    
    \caption{Phase/magnitude reconstruction using
        linear and non-linear model. Also pictured is the original
        noise-free test data; the backprojection of the noisy data may be 
        found in Figure \ref{fig:velimg-res/velimgres-tau0.5sigma1.9-backproj}.
        The upper image in each column is 
        the magnitude $\rhoC$, and the lower image the phase $\varphi$. 
        Only results for $\TV$ regularisation are shown for the linear model.
        }
    \label{fig:linvelimg}
\end{figure}

\subsection{Diffusion tensor imaging}
\label{sec:dti}

As a continuation of our earlier work on $\TGV^2$ denoising of diffusion tensor 
MRI (DTI) \cite{tuomov-dtireg,ipmsproc,escoproc} using linear models, 
we now consider an improved model. Here the purpose of the non-linear 
operator $T$ is to model the so-called Stejskal-Tanner equation
%\cite{basser2002diffusion}
\begin{equation}
    \label{eq:stejskal-tanner}
    s_j(x)=s_0(x) \exp(\iprod{b_j}{v(x)b_j}),
    \quad (j=1,\ldots,N).
\end{equation}
Here $v: \Omega \to \Sym^2(\R^3)$, $\Omega \subset \R^3$ is a mapping to symmetric second 
order tensors (representable as symmetric $3 \times 3$ matrices). Each $v(x)$
models the covariance of a Gaussian probability distribution at $x$ for the
diffusion of water molecules. Each of the diffusion-weighted MRI 
measurements $s_j$, ($j=1,\ldots,N$), is obtained with a different 
non-zero diffusion sensitising gradient $b_j$, while $s_0$ is obtained
with a zero gradient. After correct the original $k$-space data for
coil sensitivities, each $s_j$ is assumed real. As a consequence, $s_j$ 
has in effect Rician noise distribution \cite{gudbjartsson1995rician}.

Our goal is to denoise $v$. We therefore consider
\begin{equation}
    \label{eq:dti-recons}
    \min_v \sum_{j=1}^N \frac{1}{2} \norm{s_j-T_j(v)}^2 + \TGV^2_{(\beta,\alpha)}(v),
\end{equation}
where
\[
    T_j(v)(x) \defeq s_0(x) \exp(\iprod{b_j}{v(x)b_j}),
    \quad
    (j=1,\ldots,N).
\]
Due to the Rician noise of $s_j$, the Gaussian noise model implied
by the $L^2$-norm is not entirely correct. However, the $L^2$ model 
is not necessary too inaccurate, as for suitable parameters the 
Rician distribution is not too far from a Gaussian distribution. 
For correct modelling, there would be approaches. One would be 
to modify the fidelity term to model Rician noise, as was done
in \cite{getreuer2011rician,martin2013tgvdti} for single MR images. 
The second option would be to include the coil sensitivities 
in an $L^2$ model, either by knowing them, or by estimating 
them simultaneously, as was done in \cite{knoll2012parallel} 
for parallel MRI. This kind of models with direct tensor
reconstruction will be the subject of a future study.
For the present work, we are content with the simple $L^2$ model,
which already presents computational challenges through the 
non-linearities of the Stejskal-Tanner equation.

As our test data set, we have an in vivo measurement of the human brain,
%for which the author would like to thank Florian Knoll.
The data set is three-dimensional with 25 slices of size $128 \times 128$
for 21 different diffusion sensitising gradients, including the 
zero gradient. Moreover, 9 measurement were measured to construct
by averaging the ground truth depicted in Figure \ref{fig:nldtires-res/nldtires-tau0.5sigma1.9-f0}.
Only the first of the measurements is used for the backprojection 
in Figure \ref{fig:nldtires-res/nldtires-tau0.5sigma1.9-f}
and the reconstructions with \eqref{eq:dti-recons}.
It has about 8.5 million data points. 
The diffusion tensor image $v$ is correspondingly $128 \times 128 \times 25$ with
6 components for each symmetric tensor element $v(x) \in \R^{3 \times 3}$.
In addition we have the additional variable $w$, and dual variables.
This gives altogether 42 values per voxel in the reconstruction space,
or 17 million values. Considering the size of a double data type (8 bytes),
and the need for copies and temporary variables, the (C language) program 
solving this problem has a rather significant memory footprint of about 
one gigabyte.

%The measurements of a healthy volunteer were performed on a clinical 3T 
%system (Siemens Skyra), using a conventional 20 channel head/neck coil. 
%A diffusion weighted single shot SE-EPI sequence was used with diffusion 
%sensitizing gradients in 20 directions (b=1000s/mm2) and one additional 
%scan without diffusion encoding. Sequence parameters were: TR 3700ms, TE
%95ms, acquisition matrix 128×128, 25 slices, in-plane resolution 1.7×1.7mm2, slice
%thickness 1.7mm, 6/8 partial Fourier in the phase encoding direction and GRAPPA
%with an acceleration factor of 2 using 38 autocalibration lines. 
%Two sets of data were acquired: In the first measurement, a very
%high number of 9 averages was used in order to have a high SNR gold standard
%available for reference. The total scan time of this measurement was 12 min.
%Afterwards a second data set was obtained, measuring only a single average
%with a scantime of 1.5 min. Eddy current correction was performed using FSL.

As in Section \ref{sec:velmri}, we evaluate all three, Algorithm \ref{algorithm:nl-cp}
(Exact NL-PDHGM), Algorithm \ref{algorithm:nl-cp-lin} (Linearised NL-PDHGM), and the
Gauss-Newton method. The parametrisation and method setup is the same as in the previous
section, except for the step length parameter choice $\tau=\sigma=0.95/L$, we use the
lower accuracy $\rho=1\ee^{-3}$. For the choice $\tau=0.5/L$, $\sigma=1.9/L$ we use
$\rho=1\ee{-4}$ as before. Also, in both cases, in addition to $\rho_2=\rho/10$,
we perform Gauss-Newton computations with the higher accuracy $\rho_2=\rho$
for the inner PDHGM iterations.  The regularisation parameters are also naturally different.
We choose $\alpha=0.0006/\sqrt{128\cdot 128 \cdot 25}$, and $\beta=0.00066/(128\cdot 128 \cdot 25)$.

\setlength{\resfigw}{0.35\textwidth}

\newcommand{\resplotd}[3][]{%
    \setlength{\colw}{\resfigw}%
    \begin{subfigure}[b]{\colw}%
    \begin{tikzpicture}[outer sep=0pt,inner sep=0pt]%
        \pgftext[at=\pgfpoint{0}{0},left,bottom]{%
            \includegraphics[width=\colw]{{#2}.png}
        }%
        {#1}%
    \end{tikzpicture}%
    \caption{#3}
    \ifdefined\subfigprefix\label{\subfigprefix#2}\else\relax\fi%
    \end{subfigure}\relax%
    }
\newcommand{\resplotdz}[3][]{%
    \setlength{\colw}{.19\textwidth}%
    \begin{subfigure}[t]{\colw}%
    \begin{tikzpicture}[outer sep=0pt,inner sep=0pt]%
        \pgftext[at=\pgfpoint{0}{\colw+3pt},left,bottom]{%
            \includegraphics[width=\colw,bb=32 16 64 48,clip]{{#2}.png}
        }%
        {#1}%
    \end{tikzpicture}%
    \caption{#3}
    \ifdefined\subfigprefix\label{\subfigprefix#2}\else\relax\fi%
    \end{subfigure}\relax%
    }

\newlength{\scf}
\def\shadowshift{0.15ex}
\def\SQl{32}
\def\SQb{16}
\def\SQr{64}
\def\SQt{48}

\def\drawzoomsquare{%
        \setlength{\scf}{0.0078\colw} % \colw/128
        \draw[line width=1.5, color=black, shift={(\shadowshift,-\shadowshift)}] (\SQl\scf, \SQb\scf) rectangle (\SQr\scf, \SQt\scf);%
        \draw[line width=1.5, color=white] (\SQl\scf, \SQb\scf) rectangle (\SQr\scf, \SQt\scf);%
    }
    
\begin{table}%[t!]
    \caption{DTI reconstruction using
        Exact NL-PDHGM, %(Algorithm \ref{algorithm:nl-cp}),
        Linearised NL-PDHGM, %(Algorithm \ref{algorithm:nl-cp-lin}),
        and the Gauss-Newton method: equal primal and dual step length parameters.
        For Gauss-Newton, PDHGM is used in the inner iterations,
        and the number of PDHGM iterations reported is the
        total over all Gauss-Newton iterations.
        The stopping criterion for NL-PDHGM and Gauss-Newton is
        $\norm{x^i-x^{i+1}} < \rho$. 
        The stopping criterion for linear PDHGM within Gauss-Newton iterations 
        is pseudo-duality gap less than $\rho_2$.
        The number of PDHGM iterations (per Gauss-Newton iteration) is limited to 
        100000, and the number of Gauss-Newton iterations to 100.
    }
    \label{table:nldti}
    \centering
    
    $\rho=1\ee^{-3}$, $\tau=0.95/L$, $\sigma=0.95/L$ for $L = \sup_i \norm{\grad K(x^i)}$
    \medskip
    
    \begin{tabular}{l|rrrrr}
        Method & PDHGM iters. & GN iters. & Time & PSNR \\
        \hline
        Backprojection & -- & -- & -- & 17.6\\
Exact NL-PDHGM & 4600 & -- & 1351.8s & 20.1\\
Linearised NL-PDHGM & 4600 & -- & 2068.9s & 20.1\\
Gauss-Newton; $\rho_2=\rho$ & 36600 & 6 & 8222.4s & 20.6\\
Gauss-Newton; $\rho_2=\rho/10$ & 5338 & 100 & 1384.8s & 21.9\\

    \end{tabular}
%\end{table}

    \medskip
    
%\begin{table}[t!]
    \caption{DTI reconstruction using
        Exact NL-PDHGM, %(Algorithm \ref{algorithm:nl-cp}),
        Linearised NL-PDHGM, %(Algorithm \ref{algorithm:nl-cp-lin}),
        and the Gauss-Newton method: primal step length smaller than dual.
        For Gauss-Newton, PDHGM is used in the inner iterations,
        and the number of PDHGM iterations reported is the
        total over all Gauss-Newton iterations.
        The stopping criterion for NL-PDHGM and Gauss-Newton is
        $\norm{x^i-x^{i+1}} < \rho$.
        The stopping criterion for linear PDHGM within Gauss-Newton iterations
        is pseudo-duality gap less than $\rho_2$.
        The number of PDHGM iterations (per Gauss-Newton iteration) is limited to 
        100000, and the number of Gauss-Newton iterations to 100.
    }
    \label{table:nldti-unequal}
    \centering
    
    $\rho=1\ee^{-4}$, $\tau=0.5/L$, $\sigma=1.9/L$ for $L = \sup_i \norm{\grad K(x^i)}$
    \medskip
    
    \begin{tabular}{l|rrrrr}
        Method & PDHGM iters. & GN iters. & Time & PSNR \\
        \hline
        Backprojection & -- & -- & -- & 17.6\\
Exact NL-PDHGM & 8700 & -- & 2381.9s & 20.3\\
Linearised NL-PDHGM & 8700 & -- & 3302.5s & 20.3\\
Gauss-Newton; $\rho_2=\rho$ & 77264 & 8 & 18679.1s & 20.4\\
Gauss-Newton; $\rho_2=\rho/10$ & 33601 & 100 & 8442.2s & 21.2\\

    \end{tabular}
\end{table}

\begin{figure}[t]%[f]
    %\centerfloat
    %\centering
    \def\subfigprefix{fig:nldtires-}
    \begin{minipage}[t]{0.7\textwidth}
    \vspace{0pt}
    \resplotd[\drawzoomsquare]{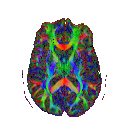}{Ground truth}
    \resplotd{res/nldtires-tau0.5sigma1.9-f}{Backprojection}
    \resplotd{res/nldtires-tau0.5sigma1.9-tgv2-nl-cpock-exact-unconstr}{Exact NL-PDHGM}
    %\resplotd{res/nldtires-tau0.5sigma1.9-tgv2-nl-cpock-lin-unconstr}{Linearised NL-PDHGM}
    \resplotd{res/nldtires-tau0.5sigma1.9-tgv2-nl-gn-unconstr}{Gauss-Newton}
    \end{minipage}
    \setlength{\colw}{0.25\textwidth}%
    \setlength{\w}{0.18\textwidth}%
    \begin{minipage}[t]{0.3\textwidth}
    \vspace{1cm}
    \def\legendstyle{\scriptsize}
    \input{legend-dirplot3d.tikz}
    \end{minipage}

    \caption{Visualisation of one slice of the DTI reconstruction using
        Exact NL-PDHGM %(Algorithm \ref{algorithm:nl-cp}),
        and the Gauss-Newton method. Also pictured is the ground truth
        and the backprojection reconstruction.
        The data displayed is colour-coded principal eigenvector.
        The area outside the brain has been masked out.
        The rectangle in (\subref{fig:nldtires-res/nldtires-tau0.5sigma1.9-f0})
        indicates the region displayed in Figure \ref{fig:nldti-lincompare-psnr} 
        and Figure \ref{fig:nldti-lincompare-discr}.
        }
    \label{fig:nldti}
\end{figure}

The results of the computations are reported in Table \ref{table:nldti},
Table \ref{table:nldti-unequal}, and Figure \ref{table:nldti}.
They confirm our observations in the velocity imaging example,
but we do have some convergence issues.
In case of Gauss-Newton, we do not observe convergence with accuracy
$\rho_2=\rho/10$ for the inner PDHGM iterations, and have to use
$\rho_2=\rho$. We find this quite interesting, since NL-PDHGM
gives better convergence results, and is essentially Gauss-Newton with
just a single step in the inner iteration.
Although not reported in the tables, we also observed that if using the more 
accurate stopping threshold $\rho = 1\ee^{-4}$, instead of $1\ee^{-3}$ for the 
computations in Table \ref{table:nldti}, with equal $\sigma$ and $\tau$, 
NL-PDHGM did not appear to convergence until reaching the maximum 
iteration count of 100000.
This may be due to starting too far from the solution, or due to
the fact that convergence of even the linear PDHGM can become very slow
in the limit. Nevertheless, with the stopping threshold $\rho = 1\ee^{-3}$,
we quickly obtained satisfactory solutions, as can be observed by
comparing the PSNRs between Table \ref{table:nldti} and
Table \ref{table:nldti-unequal}, as well as Figure \ref{table:nldti},
which visualises the results from the latter. Minding the large scale
of the problem, we also had reasonably quick convergence to a greater
accuracy with the unequal parameter choice. This gives us confidence 
for further application and study of NL-PDHGM in future work.

Finally, studying Table \ref{table:nldti}, the reader may observe that the PSNR of the 
unconverged solution produced by the Gauss-Newton method is better than the other 
solutions;  this is simply because we did not choose the 
regularisation parameters optimally, only being interested in studying 
convergence of the methods for the present paper. Therefore we also
do not compare the non-linear model completely rigorously against our
earlier efforts with linear models. This will be the topic of future research.
However, in order to justify the non-linear model, we have a simple comparison to
report in Table \ref{table:nldti-lincompare},
Figure \ref{fig:nldti-lincompare-psnr}, and Figure \ref{fig:nldti-lincompare-discr}.
In order to facilitate comparison, we have zoomed the plots into the rectangle indicated
in Figure \ref{fig:nldtires-res/nldtires-tau0.5sigma1.9-f0}.
We compare the model non-linear model \eqref{eq:dti-recons} against the linear
model of \cite{tuomov-dtireg} with fidelity term $\norm{f-u}^2$ in tensor space,
where we first solve the noisy tensor field $f$ by linear regression from 
\eqref{eq:stejskal-tanner}. Observe that we may linearise the equation by 
taking the logarithm on both sides. We also include in our comparison the model 
of \cite{ipmsproc} involving this linearisation in the fidelity term.
We calculated solutions to all of these models for
$\alpha \in [0.0002, 0.00110] / \sqrt{128 \cdot 128 \cdot 25}$, spacing $0.00005$, and picked again the
optimal $\hat \alpha$ by both the PSNR and the discrepancy principle.
Reading Table \ref{table:nldti-lincompare}, we see that the non-linear 
performs clearly the best when $\hat \alpha$ is chosen by the discrepancy 
principle. When $\hat \alpha$ is chosen by optimal PSNR, the tensor space 
linear model is better. 
However, studying Figure \ref{fig:nldti-lincompare-psnr-res/nldtires_lincompare_psnr-tgv2-ten-cpock-unconstr}, 
the result is significantly smoothed, reminding us that the PSNR is a poor quality criterion for
imaging problems. In case of selection of $\hat\alpha$ by the discrepancy principle
in Figure \ref{fig:nldti-lincompare-discr}, the situation is not so dire,
but still the linear models exhibit a level of smoothing.
The non-linear model \eqref{eq:dti-recons} performs visually significantly better.

\begin{table}%[t]
    \caption{DTI reconstruction using linear and non-linear models}
    \label{table:nldti-lincompare}
    
    \centering
    
    \begin{tabular}{ll|rr}
        Method & Criterion & $\hat \alpha \cdot \sqrt{128\cdot 128 \cdot 25}$ & PSNR \\
        \hline
        Linear, tensor space & Disrepancy principle & 0.00025 & 19.8 \\
Linear, DWI space & Disrepancy principle & 0.00030 & 18.5 \\
Non-linear & Disrepancy principle & 0.00020 & 21.1 \\

        \hline
        Linear, tensor space & PSNR & 0.00090 & 22.8 \\
Linear, DWI space & PSNR & 0.00110 & 19.7 \\
Non-linear & PSNR & 0.00020 & 21.1 \\

    \end{tabular}
\end{table}

\begin{figure}
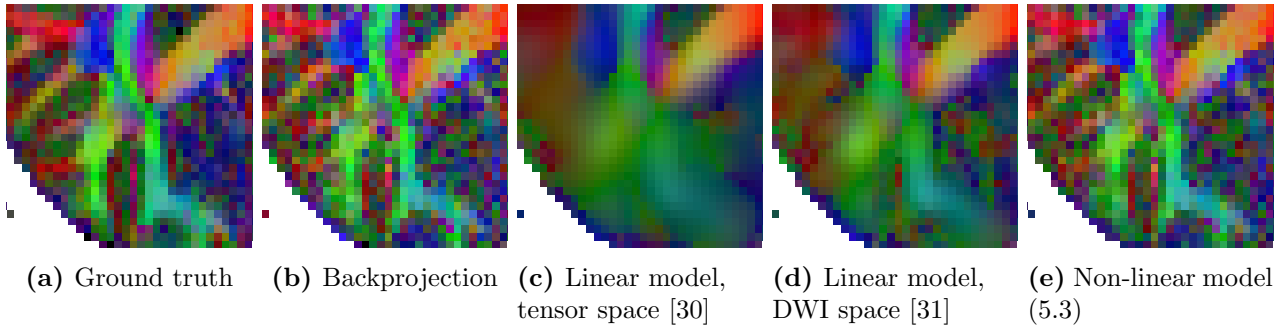
%[t]%[f]
    %\centerfloat
    \centering
    \def\subfigprefix{fig:nldti-lincompare-psnr-}
    %\begin{minipage}[t]{0.7\textwidth}
    %\vspace{0pt}
    \resplotdz{res/nldtires_lincompare_psnr-f0}{Ground truth}
    \resplotdz{res/nldtires_lincompare_psnr-f}{Backprojection}
    \resplotdz{res/nldtires_lincompare_psnr-tgv2-ten-cpock-unconstr}{Linear model,\\ tensor space \cite{tuomov-dtireg}}
    \resplotdz{res/nldtires_lincompare_psnr-tgv2-raw-cpock-unconstr}{Linear model,\\ DWI space \cite{ipmsproc}}
    \resplotdz{res/nldtires_lincompare_psnr-tgv2-nl-cpock-exact-unconstr}{Non-linear model \eqref{eq:dti-recons}}
    %\end{minipage}
    %
    %\setlength{\colw}{0.25\textwidth}%
    %\setlength{\w}{0.18\textwidth}%
    %\begin{minipage}[t]{0.3\textwidth}
    %\vspace{1cm}
    %\def\legendstyle{\scriptsize}
    %\input{legend-dirplot3d.tikz}
    %\end{minipage}
    %

    \caption{DTI reconstruction using linear and non-linear models, $\alpha$ selected by best PSNR.
        The region indicated by the rectangle in Figure \ref{fig:nldtires-res/nldtires-tau0.5sigma1.9-f0} is plotted.}
    \label{fig:nldti-lincompare-psnr}
\end{figure}

\begin{figure}[t]
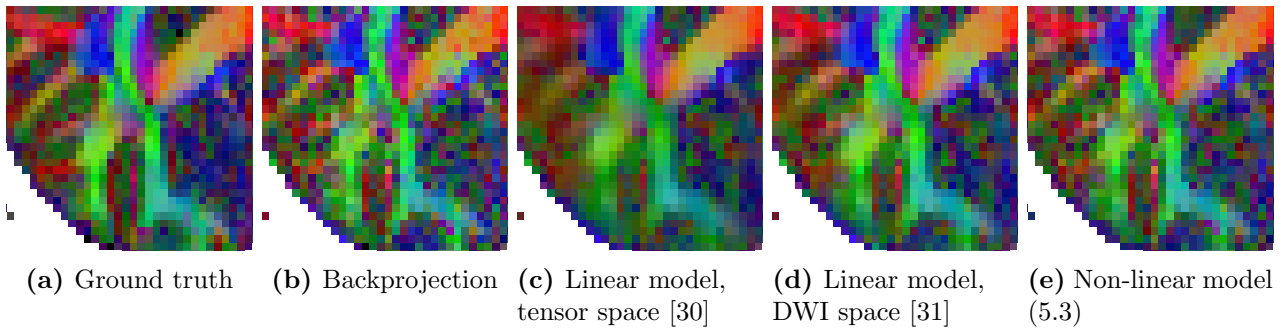
%[f]
    %\centerfloat
    \centering
    \def\subfigprefix{fig:nldti-lincompare-discr-}
    %\begin{minipage}[t]{0.7\textwidth}
    %\vspace{0pt}
    \resplotdz{res/nldtires_lincompare_discr-f0}{Ground truth}
    \resplotdz{res/nldtires_lincompare_discr-f}{Backprojection}
    \resplotdz{res/nldtires_lincompare_discr-tgv2-ten-cpock-unconstr}{Linear model,\\ tensor space \cite{tuomov-dtireg}}
    \resplotdz{res/nldtires_lincompare_discr-tgv2-raw-cpock-unconstr}{Linear model,\\ DWI space \cite{ipmsproc}}
    \resplotdz{res/nldtires_lincompare_discr-tgv2-nl-cpock-exact-unconstr}{Non-linear model \eqref{eq:dti-recons}}
    %\end{minipage}
    %
    %\setlength{\colw}{0.25\textwidth}%
    %\setlength{\w}{0.18\textwidth}%
    %\begin{minipage}[t]{0.3\textwidth}
    %\vspace{1cm}
    %\def\legendstyle{\scriptsize}
    %\input{legend-dirplot3d.tikz}
    %\end{minipage}
    %

    \caption{DTI reconstruction using linear and non-linear models, $\alpha$ selected by the discrepancy principle.
        The region indicated by the rectangle in Figure \ref{fig:nldtires-res/nldtires-tau0.5sigma1.9-f0} is plotted.}
    \label{fig:nldti-lincompare-discr}
\end{figure}

\def\shrinkspace{\vspace{-2.8ex}}

\shrinkspace
\section*{Acknowledgements}

This work has been financially supported by the King Abdullah University
of Science and Technology (KAUST) Award No.~KUK-I1-007-43 as well as the
EPSRC / Isaac Newton Trust Small Grant ``Non-smooth geometric 
reconstruction for high resolution MRI imaging of fluid transport
in bed reactors'' and the EPSRC first grant Nr.~EP/J009539/1
``Sparse \& Higher-order Image Restoration''.

\nobreak

The author is grateful to Florian Knoll for providing the in vivo DTI data set.

\shrinkspace
\section*{Software implementation}

Our C language implementation of the method for $\TV$ and $\TGV^2$ regularised
problems may be found under \url{http://iki.fi/tuomov/software/}.

\shrinkspace
%\bibliography{abbrevs,bib,bib-dti,bib-own,bib-trans,bib-l1tgv,bib-bd,bib-opt,bib-mri}
\providecommand{\noopsort}[1]{}\providecommand{\homesiteprefix}{http://iki.fi/tuomov/mathematics}\providecommand{\eprint}[1]{\href{http://arxiv.org/abs/#1}{arXiv:#1}}

\end{document}